\documentclass[11pt, letterpaper]{amsart}
\usepackage{EKstylesheet}
\newtheorem{ques}[theorem]{Question}
\makeatletter
\let\@wraptoccontribs\wraptoccontribs
\makeatother

\allowdisplaybreaks
\title{Semiclassical Measures on Hyperbolic Manifolds}
\author{Elena Kim}
\address{Department of Mathematics\\Massachusetts Institute of Technology\\Cambridge, MA 02139}
\email{elenakim@mit.edu}
\author{Nicholas Miller}
\address{Department of Mathematics\\University of Oklahoma\\Norman, OK 73019}
\email{nickmbmiller@ou.edu}

\begin{document}

\begin{abstract}
We examine semiclassical measures for Laplace eigenfunctions on compact hyperbolic $(n+1)$-manifolds. We prove their support must contain the cosphere bundle of a compact immersed totally geodesic submanifold of dimension at least 2. 
Our proof adapts the argument of \cite{DJ18} and \cite{ADM24} and classifies the closures of unipotent orbits using Ratner theory. An important step in the proof is a generalization of the higher-dimensional fractal uncertainty principle of Cohen \cite{Co24} to Fourier integral operators, which may be of independent interest.  
\end{abstract}

\maketitle

\section{Introduction}\label{section:introduction}

Let $(M, g)$ be a compact hyperbolic $(n+1)$-dimensional manifold, that is a compact Riemannian manifold with constant curvature $-1$.
A key area of research in quantum chaos is semiclassical measures, detailed in Definition \ref{def:semiclassical_measure},
which capture the high frequency limit of the mass of eigenfunctions of the Laplacian. The study of semiclassical measures is guided by the Quantum Unique Ergodicity conjecture of \cite{RS94}, which poses that on negatively curved compact manifolds, the whole sequence of eigenfunctions converges semiclassically to the Liouville measure. Intuitively put, all eigenfunctions equidistribute in the semiclassical limit.

In this paper, we study the possible \emph{supports of semiclassical measures}.
Results on these supports have found applications to control
estimates~\cite{Ji18}, exponential decay for the damped wave equation~\cite{Ji20},
and bounds on restrictions of eigenfunctions~\cite{GZ21}.

\subsection{Main Results}
We let $F^*M$ denote the \emph{coframe bundle}. Elements of $F^*M$ are of the form $q= (x, \xi_1, \ldots,  \xi_{n+1})$, where $x \in M$ and  $\xi_1, \ldots, \xi_{n+1} \in T^*_x M$ form a positively oriented orthonormal basis.
Fix $\pi_S$ to be the following submersion:
$$\pi_S: F^*M \rightarrow S^*M, \quad (x, \xi_1, \ldots,  \xi_{n+1}) \mapsto (x, \xi_1).$$

Let $U_1^-$ be the vector field on $F^*M$ defined in \eqref{eq:frame}. This vector field is the generator of the one-parameter unipotent flow $e^{s U_1^-}$. 

It is well-known (see for example ~\cite{Zw12}*{\S \S5.1--2}) that if $\mu$ is a semiclassical measure, $\mu$ is invariant under the geodesic flow $\phi_t: S^*M \rightarrow S^*M$ and $\supp \mu \subset S^*M$.

Our first result describes the support of the semiclassical measures. 

\begin{theorem}
\label{thm:support}
Let $M$ be a compact hyperbolic manifold. If $\mu$ is a semiclassical measure on $M$, then for some $q \in F^*M$, 
$$\pi_S \overline{\{\phi_t ( e^{sU_1^-}(q)):t, s \in \R\}} \subset \supp \mu.$$
\end{theorem}

In the next result, we further characterize the orbit closure appearing in Theorem  \ref{thm:support}.
\begin{theorem}\label{thm:orbit_closure}
For $q =(x, \xi_1, \ldots, \xi_{n+1}) \in F^*M$,
$$\pi_S \overline{\{\phi_t( e^{sU_1^-}(q)) : t, s \in \R\}} = S^* \Sigma_q,$$ where $\Sigma_q$ is the minimal compact immersed totally geodesic submanifold in $M$ such that $x \in \Sigma_q$ and $\xi_1$, $\xi_2 \in T_x \Sigma_q$. 
\end{theorem}

As discussed in \S \ref{section:TGexamples}, there are many 3-dimensional hyperbolic manifolds with no totally geodesic surfaces. On any such manifold, Theorem~\ref{thm:orbit_closure} implies semiclassical measures must have full support.

We remark here that the presence of cosphere bundles of totally geodesic submanifolds in the higher-dimensional statement of both Theorem \ref{thm:orbit_closure} and the main theorem of \cite{ADM24} is intimately related to the failure of one-parameter unipotent flows on $F^*M$ to be uniquely ergodic.
This should be juxtaposed with the $2$-dimensional case \cite{DJ18}, where the unique ergodicity of the horocycle flow is a key tool in concluding that semiclassical measures have full support.

\subsection{Previous results}

We briefly review previous results on semiclassical measures. For
additional context, see the surveys~\cite{Ze19} and~\cite{Dy23}.

Developed in the 1970s and 80s, the \emph{Quantum Ergodicity theorem} of  Shnirelman~\cite{Sh74}, \linebreak Zelditch~\cite{Ze87},  and  Colin de
Verdi\`{e}re~\cite{CdV85}  implies that on negatively curved compact manifolds, a density-one sequence of Laplace eigenfunctions converges semiclassically to the Liouville measure. 

The \emph{Quantum Unique Ergodicity (QUE) conjecture} of Rudnick and Sarnak ~\cite{RS94} asserts that on negatively curved compact manifolds, the whole sequence of eigenfunctions converges semiclassically to the Liouville measure.  
The geodesic flow on a negatively curved compact manifold is Anosov. This condition is necessary: Donnelly \cite{Do03} and Hassell \cite{Ha10} have constructed examples of manifolds with ergodic, but not Anosov, geodesic flows with quasimodes and eigenfunctions that do not equidistribute.

QUE remains open in general. However, it has been partially resolved in the arithmetic setting, i.e., under the restriction to the joint eigenfunctions of the Laplacian and Hecke operators. In this setting,
 QUE is known as  \emph{Arithmetic Quantum Unique Ergodicity (AQUE)}.
Specifically, Lindenstrauss~\cite{Li06} and Soundararajan~\cite{So10} proved AQUE on compact arithmetic surfaces and on the modular surface $\SL_2(\Z) \backslash \bH^2$.  
Recent work of Shem-Tov and Silberman \cite{ShemSilb} established AQUE for compact arithmetic hyperbolic $3$-manifolds. In dimension $4$, Shem-Tov--Silberman \cite{ShemSilb2} also have partial results towards AQUE. Namely, any semiclassical measure on a compact arithmetic hyperbolic $4$-manifold is a convex sum of the Liouville measure and the Liouville measure on the cosphere bundles of (the infinitely many) immersed totally geodesic $3$-dimensional submanifolds. However, they are unable to remove the possibility that the measure might concentrate on such submanifolds.
At present, there are no arithmetic results in dimensions greater than $4$.

The starting point of \cite{Li06} to prove AQUE in dimension $2$ is the use of homogeneous dynamics to prove a Ratner--style measure classification theorem. This was vastly generalized by Einsiedler--Lindenstrauss in \cite{EinLind}.
More precisely, if a measure is geodesic flow invariant and satisfies $3$ additional rigidity hypotheses, Einsiedler--Lindenstrauss show that it is a homogeneous measure, that is, its ergodic components are Haar measures on orbit closures of subgroups generated by unipotents in $\Gamma\backslash G$.
The striking feature of this result is that, in contrast to the work of Ratner \cites{RatnerMeasure, RatnerTop}, it applies to measures invariant under the geodesic flow which are generally too flexible to admit a meaningful classification.
To verify AQUE, Lindenstrauss and Shem-Tov--Silberman first verify the $3$ hypotheses of \cites{Li06,EinLind} for semiclassical measures (which is non-trivial in its own right). They are therefore reduced to the homogeneous setting. From this, one can immediately conclude AQUE in dimension $2$. However, in dimension $3$, one must additionally rule out measures with positive mass on the cosphere bundle of a totally geodesic surface.
The inability to rule out this phenomenon in dimension $4$ is precisely the reason for the conditional in \cite{ShemSilb2}.
We mention this to draw parallels between these results and the statements of Theorems \ref{thm:support} and \ref{thm:orbit_closure}.
That is, the condition of containing the cosphere bundle of a totally geodesic submanifold is natural from these perspectives and reminiscent of, though strictly weaker than, what one obtains from this Ratner--style measure rigidity in the arithmetic setting.

Returning to the original QUE conjecture, one branch of work has focused on proving positive lower bounds for the Kolmogorov--Sina\u{\i} entropy of semiclassical measures, known as \emph{entropy bounds}. These entropy bounds were established by Anantharaman~\cite{An08}; Anantharaman and Nonnenmacher~\cite{AN07}; Anantharaman, Koch, and Nonnenmacher~\cite{AKN09}; Rivi\`{e}re~\cites{Ri10-1, Ri10-2}; and Anantharaman and Silberman~\cite{AS13}. Specifically, the work of~\cite{AN07} shows that for an $(n+1)$-dimensional manifold with constant negative curvature, a semiclassical measure must have entropy at least $n/2$. Note that the Liouville measure has entropy $n$ and is the unique measure of maximal entropy, while the $\delta$-measure on a closed geodesic has entropy 0. 

Theorems \ref{thm:support} and \ref{thm:orbit_closure} exclude different semiclassical measures than the entropy bounds. 
This distinction is perhaps easiest to state in dimension $3$, where the existence of infinitely many commensurability classes of compact hyperbolic $3$-manifolds with no immersed, totally geodesic surfaces is known \cite[Theorem 9.5.1]{MRBook}.
For such examples, Theorem \ref{thm:orbit_closure} asserts that any semiclassical measure has full support, in contrast to the abundance of measures with entropies greater than $1$ with proper support.
In higher dimensions, one can make similar statements, though a limiting factor is our understanding of the behavior of geodesic submanifolds in non-arithmetic manifolds (see \S \ref{section:TGexamples}).
However, the entropy bounds exclude some measures that our results do not. One example is a semiclassical measure of the form $\mu = \alpha \mu_L + (1-\alpha) \mu_0$, where $\mu_L$ is the Liouville measure, $\mu_0$ is a $\delta$-measure on a closed geodesic, and $0<\alpha<1/2$. Such a measure clearly has full support, but has entropy $\alpha n<n/2$, and therefore is ruled out by~\cite{AN07}.

Our work, specifically Theorem~\ref{thm:support}, builds off of \cite{DJ18}, wherein Dyatlov and Jin proved that all semiclassical measures on compact hyperbolic surfaces $M= \Gamma \backslash \bH^2$ have full support. This result was generalized to compact surfaces with negative curvature by Dyatlov, Jin, and Nonnenmacher in ~\cite{DJN22}. Both \cite{DJ18} and \cite{DJN22} relied on the 1-dimensional fractal uncertainty principle \cite{BD18}, which restricted their results to surfaces. However, in 2023, Cohen \cite{Co24} generalized the fractal uncertainty principle to higher dimensions. His proof used the work of Han and Schlag from~\cite{HS20}.  We exploit this breakthrough to examine semiclassical measures on higher-dimensional manifolds.

Recently, Athreya, Dyatlov, and Miller in \cite{ADM24} studied compact complex hyperbolic manifolds. They found that the support of a semiclassical measure must contain $S^* \Sigma$, where $\Sigma$ is a compact immersed totally geodesic complex submanifold. 
Although their result holds in higher dimensions, similarly to the work on quantum cat maps in \cite{DJ23}, they use only the 1-dimensional fractal uncertainty principle. For both complex and real hyperbolic manifolds, the geodesic flow expands/contracts on the unstable/stable subspaces. However, on complex hyperbolic manifolds, the geodesic flow expands and contracts fastest each in a single direction. These directions are 1-dimensional and thus can be analyzed using the 1-dimensional fractal uncertainty principle. 
However, in our case of real hyperbolic manifolds, the rate of expansion/contraction of the geodesic flow is uniform on the unstable/stable subspaces. For $n+1>2$, the unstable/stable subspaces have dimension greater than $1$ and thus require the higher-dimensional fractal uncertainty principle. 

\subsection{Proof outline for Theorem~\ref{thm:support}}
Let $u_j$ be a sequence of normalized eigenfunctions of $-\Delta$ with eigenvalues $h_j^{-2} \rightarrow \infty$ that converge semiclassically to $\mu$. We assume towards a contradiction that the projection onto $S^*M$ of every orbit of $U_1^-$ in $F^*M$ intersects  $S^*M \setminus \supp \mu$. Once a contradiction is found, Theorem~\ref{thm:support} follows by the $\phi_t$-invariance of $\mu$.

 We construct a partition of unity $a_1 + a_2 =1$ on $S^*M$ such that
 \begin{enumerate_eq}
     \item $\supp a_1 \cap \supp \mu = \emptyset$ \label{eq:outline_support_property}
     \item and the projection onto $S^*M$ of every orbit of $U_1^-$ in $F^*M$ intersects $S^*M \setminus \supp a_1$, $S^*M \setminus \supp a_2$. \label{eq:outline_U1-_dense}
 \end{enumerate_eq}
We then dynamically refine and quantize this partition of unity. Specifically,  we quantize $a_1 \circ \phi_k$ and $a_2  \circ \phi_k$ to obtain operators $A_1(k)$ and $A_2(k)$ for $0 \leq k \leq 2T_1 -1 \approx 2 \rho \log h^{-1}$, where $\rho \in (3/4, 1)$. For $\w=w_0 \cdots w_{2T_1-1} \in \{1, 2\}^{2T_1}$, set $A_\w = A_{w_{2T_1-1}}(m-1) \cdots A_{w_0}(0)$. We split the set of words  $\w \in \{1, 2\}^{2T_1}$ into two parts. Loosely, let $\cY$ be the set of $\w$ with a large fraction of 1's and let $\cX$ be the set of  $\w$ with a small fraction of 1's. Set 
$A_\cY$ to be the sum of $A_\w$ with $\w \in \cY$  and  $A_\cX$ to be the sum of $A_\w$ with $\w \in \cX$. Note that $A_\cY + A_\cX = I$ on $S^*M$. 

To reach a contradiction, it suffices to prove  $\|A_\cY u_j\|_{L^2}$
and $\|A_\cX\|_{L^2 \rightarrow L^2}$ both converge to 0 in the semiclassical limit.
We use the fact that $a_1$ is supported away from $\mu$ to obtain the decay of  $\|A_\cY u_j\|_{L^2}$. To show $\|A_\cX\|_{L^2 \rightarrow L^2}$ converges to zero, we employ the triangle inequality: it suffices to bound $\# \cX$ and prove the decay of $\|A_\w\|_{L^2 \rightarrow L^2}$ for $\w \in \{1,2\}^{2T_1}$. We can adjust our definition of $\cX$ to control $\# \cX$; the main difficulty is showing $\|A_\w\|_{L^2 \rightarrow L^2} \rightarrow 0$. This is the statement of Lemma~\ref{lem:A_w_decay}, the proof of which requires the fractal uncertainty principle. 

The fractal uncertainty principle requires two sets, one that is line porous and one that is ball porous, see Definitions~\ref{def:porous_on_balls}~and~\ref{def:porous_on_lines}. We adapt these definitions to hyperbolic manifolds to define the notions of hyperbolic line and ball porosity. We show hyperbolic line and ball porosity imply a fractal uncertainty principle.

The choice of $T_1$ is too large for $A_\w$ to be pseudodifferential. However, we  split $\w$ into two equal parts, $\w= \w_+ \w_-$, and use $\w_\pm$ to construct two operators that are pseudodifferential with symbols $a_\pm$. 
We then show that $\supp a_\pm$ give rise to hyperbolic line and ball porous sets. This argument works because of the construction of $a_\pm$ from propagated symbols $a_i \circ \phi_k$,  the property ~\eqref{eq:outline_U1-_dense}, and the mixing of the geodesic flow. The property~\eqref{eq:outline_U1-_dense} is needed for hyperbolic line porosity, while mixing of the geodesic flow is needed for hyperbolic ball porosity. We can then apply the fractal uncertainty principle and deduce $\|A_\w\|_{L^2 \rightarrow L^2}$.

\subsection{Structure of the paper}
\begin{itemize}
    \item In \S\ref{section:preliminaries}, we review the required preliminaries for this paper. Specifically, in \S\ref{subsection:hyperbolic_manifolds}, we survey the geometric and dynamical properties of hyperbolic manifolds; in \S\ref{subsection:safe_sets}, we exploit the hyperbolic structure to construct a partition of unity; and in \S\ref{subsection:semiclassical_definitions}--\ref{subsection:symbol_class} we recall the essential semiclassical definitions.
    \item In \S\ref{section:orbit}, we give a self-contained proof of Theorem \ref{thm:orbit_closure} and discuss known examples of hyperbolic $n$-manifolds.
    \item In \S\ref{section:FUP}, we state and generalize the higher-dimensional fractal uncertainty principle.
    \item In \S\ref{section:reduction}, we reduce the proof of Theorem~\ref{thm:support} to showing Lemma~\ref{lem:A_w_decay}.
    \item In \S\ref{section:apply_FUP}, we prove Lemma~\ref{lem:A_w_decay} using the fractal uncertainty principle.
\end{itemize}

\bigskip\textbf{Acknowledgements:} EK would like to thank Semyon Dyatlov (partially supported by NSF--DMS-2400090 and NSF CAREER grant DMS--1749858) for suggesting and advising this project.
EK is supported by NSF GRFP under grant No. 1745302 and NM is partially supported by NSF DMS--2405264. 


\section{Preliminaries}
\label{section:preliminaries}


\subsection{Hyperbolic manifolds}
\label{subsection:hyperbolic_manifolds}
We begin by following the exposition in ~\cite{DFG15}*{\S\S 3.1--2}. Let $\bH^{n+1}$ be hyperbolic $(n+1)$-space. We use the notation $\R^{1,n+1}$ to denote $\R^{n+2}$ endowed with the Minkowski metric $g_M = -dx_0^2 + \sum_{j=1}^{n+1} dx_j^2$. We write the corresponding scalar product as $\lrang{\cdot, \cdot}_M$. Then the hyperbolic space of dimension $n+1$ is defined as 
\begin{equation}\label{eqn:hyperbolicspace}
\bH^{n+1} \coloneqq \{x \in \R^{1, n+1} : \lrang{x,x}_M =-1, x_0 >0\}.
\end{equation}
The hyperbolic metric is given by $g \coloneqq g_M\vert_{T\bH^{n+1}}$. 
Moreover, the boundary of $\bH^{n+1}$ is given by the set of all null lines in $\R^{1,n+1}$. 
We normalize this to give the following description of the boundary
\begin{equation}\label{eqn:hyperbolicspaceboundary}
\partial \bH^{n+1} \coloneqq \{x \in \R^{1, n+1} : \lrang{x,x}_M =0, x_0 =1\}.
\end{equation}

Let $G \coloneqq \SO_0(1, n+1) \subset \SL(n+2, \R)$ be the group of orientation-preserving isometries of $\bH^{n+1}$.
That is, $\SO_0(1, n+1)$ is the connected component of the identity of $\SO(1,n+1)$, which are precisely the matrices in $\SO(1,n+1)$ which preserve $\bH^{n+1}$.
This is a connected, noncompact Lie group.
By considering the orbit of $x=(1,0,\dots,0)\in\bH^{n+1}$, $B\in\SO(1,n+1)$ is an element of $\SO_0(1,n+1)$ if and only if the upper-left entry of this matrix, $B_{0,0}$, is positive.

Denote by $\vec{e}_0, \ldots, \vec{e}_{n+1}$ the canonical basis of $\R^{1, n+1}$. Let $\pi_K: G \rightarrow \bH^{n+1}$, $B \mapsto B\vec{e}_0$.
Thus, we can write $\bH^{n+1} \simeq G/K$, where 
\begin{equation}\label{eqn:Kequation}
K \coloneqq  \{B \in G : B\vec{e}_0 = \vec{e}_0\}= \left\{ 
\begin{bmatrix}
 1& 0 & \cdots &0 \\
\hspace{5pt}\begin{matrix} 0\\ \vdots \\ 0 \end{matrix}   && \mathlarger{\mathlarger{\mathlarger{\mathlarger{B}}}} \end{bmatrix} : B \in \SO(n+1) \right\} \simeq \SO(n+1).\end{equation}
In particular, $K$ is the \emph{isotropy group} of $\vec{e}_0$ as well as a maximal compact subgroup of $G$. 

Let $M$ be a compact hyperbolic $(n+1)$-dimensional manifold. We consider $M$ via group actions as $M = \Gamma \backslash \bH^{n+1} \simeq \Gamma \backslash G /K$ for a torsion-free, cocompact lattice $\Gamma \subset G$.

We turn our attention to $SM$, which we identify with $S^*M$ using the Riemannian metric. We know $SM = \Gamma \backslash S \bH^{n+1}$, where 
\begin{equation}\label{eq:unittangentbundle}
S\bH^{n+1} \coloneqq \{(x, \xi) : x \in \bH^{n+1}, \xi \in \R^{1, n+1}, \lrang{\xi, \xi}_M =1, \lrang{x, \xi}_M=0\}.
\end{equation}

Similarly to the above, we have a map $\pi_{K_0}: G \rightarrow S\bH^{n+1}$, $B \mapsto (B \vec{e}_0, B \vec{e}_1)$. 
Then $S\bH^{n+1} \simeq G /{K_0}$, where 
\begin{equation}\label{eqn:K0defn}{K_0} \coloneqq \{B \in G : B\vec{e}_0 = \vec{e}_0,~B\vec{e}_1 = \vec{e}_1\}= \left\{ 
\begin{bmatrix}
 1& 0 & 0 &\cdots & 0 \\
 0& 1 & 0 &\cdots & 0 \\
\hspace{5pt}\begin{matrix} 0 \\ \vdots \\ 0 \end{matrix}   & \hspace{5pt}\begin{matrix} 0 \\ \vdots \\ 0 \end{matrix} && \mathlarger{\mathlarger{\mathlarger{\mathlarger{B}}}} \end{bmatrix} : B \in \SO(n) \right\} \simeq \SO(n).\end{equation}

Therefore, $SM \simeq \Gamma \backslash G/{K_0}$. 
For $1 \leq i, j \leq n+1, 2 \leq k \leq n+1$, the Lie algebra of $G$ is spanned by the matrices
$$X \coloneqq E_{0,1} + E_{1,0}, \quad A_k \coloneqq E_{0,k} + E_{k,0}, \quad R_{i,j} \coloneqq E_{i,j} - E_{j,i},$$
where $E_{i,j}$ is the elementary matrix with all entries zero, except for the $(i,j)$th entry, which is equal to 1.
${K_0}$ is spanned by $R_{i+1, j+1}$ for $1 \leq i, j \leq n$.

For $1 \leq i < j \leq n$, the following forms a basis for the Lie algebra of $G$
\begin{equation}\label{eq:frame}
X, \quad R_{i+1, j+1}, \quad U_i^+ \coloneqq -A_{i+1} - R_{1, i+1}, \quad U_i^- \coloneqq -A_{i+1} + R_{1, i+1}.
\end{equation}

For $1 \leq i, j \leq n$ and $i \neq j$, we have the following commutator relations
\begin{equation}\label{eq:commutator}
\begin{split}
[X, U_i^\pm] = \pm U_i^\pm, \quad [U_i^\pm, U_j^\pm]=0, \quad [U_i^+, U_i^-] &=2X, \quad [U_i^\pm, U_j^\mp]=2R_{i+1, j+1},\\  [R_{i+1, j+1}, X]=0, \quad [R_{i+1, j+1}, U_k^\pm] =& \delta_{jk}U_i^\pm -\delta_{ik}U_j^\pm.
\end{split}
\end{equation}

Let $F^*M$ denote the \emph{coframe bundle}. Elements of $F^*M$ are of the form $(x, \xi_1, \ldots, \xi_n, \xi_{n+1})$, where $x \in M$ and  $\xi_1, \ldots, \xi_{n+1} \in T^*_x M$ form a positively oriented orthonormal basis.
Let $\pi_S$ be the following submersion
$$\pi_S: F^*M \rightarrow S^*M, \quad (x, \xi_1, \ldots, \xi_n, \xi_{n+1}) \mapsto (x, \xi_1).$$

Let $\pi_\Gamma^F: F^*\bH^{n+1} \rightarrow F^*M$ be the covering map. Using $\pi_\Gamma^F$, we have a right action of $G$ on $F^*M$. Thus, viewing the elements of the Lie algebra of $G$ as left-invariant vector fields, each element of $G$ induces a vector field on $F^*M$. 

The geodesic flow $\phi_t: S \bH^{n+1} \rightarrow S \bH^{n+1}$ (under the parametrization \eqref{eq:unittangentbundle}) is given by $\phi_t(x, \xi) = (x \cosh t + \xi \sinh t, x \sinh t + \xi \cosh t)$. We see $\phi_t(\pi_{K_0}(g)) = \pi_{K_0}(g e^{tX})$ for $g \in G$. 
As $X$ is invariant under right multiplication by elements of the subgroup ${K_0}$, we can use $\pi_{K_0}$ to push forward the left-invariant vector field on $G$ generated by $X$ to obtain the generator of the geodesic flow. By an abuse of notation, we use $X$ to also denote the generator of the geodesic flow $\phi_t$ on $S^* M$, i.e., 
$$\phi_t = e^{tX}.$$  We extend $\phi_t$ to $T^* M \setminus 0$ by setting it to be homogeneous. By homogeneity, we mean  $[X, \xi \cdot \partial_\xi] =0$, where $\xi \cdot \partial_{\xi}$ is the generator of dilations.

There is an isomorphism $G \simeq F^*\bH^{n+1}$ given by $B \mapsto (B \vec{e}_0, B \vec{e}_1, \ldots, B \vec{e}_{n+1})$. Then, $\Gamma \backslash G \simeq F^*M$.

We know that $U_i^\pm$ generate the unipotent flows $e^{sU_i^\pm}$. In other words, if $q \in F^* \bH^{n+1}$, then for all $s \in \R$, the geodesic starting at $\pi_S(e^{sU_i^\pm} q)$ has the same positive limiting point for $U_i^+$ and negative limiting point for $U_i^-$  as the geodesic starting at $\pi_S(q)$.

Differentiating the conditions for $S\bH^{n+1}$ given by \eqref{eq:unittangentbundle}, we get 
$$T_{(x,\xi)}(S\bH^{n+1}) = \left\{(v_x, v_\xi) \in (\R^{1,n+1})^2 : \lrang{x, v_x}_M =0, \lrang{x, v_\xi}_M + \lrang{\xi, v_x}_M =0, \lrang{\xi, v_\xi}_M=0\right\}.$$
Following the exposition of ~\cite{DZ16}*{\S 4.1},
for each $(x, \xi) \in T^*M \setminus 0$, we have the following decomposition of the tangent space at each $(x, \xi) \in T^*M \setminus 0$
\begin{equation}\label{eq:tangent_space_decomp}
T_{(x, \xi)}(T^*M) = \R X \oplus \R(\xi \cdot \partial_\xi) \oplus E_s(x, \xi) \oplus E_u(x, \xi),
\end{equation}
where $E_s$ and $E_u$ are respectively the $n$-dimensional stable and unstable bundles. These stable and unstable bundles were defined over $S^* \bH^{n+1}$ in ~\cite{DFG15}*{(3.14)} by setting
\begin{equation}
\label{eq:Es_Eu_def}
\begin{split}
E_s(x, \xi) \coloneqq\{(v, -v) : \lrang{x, v}_M = \lrang{\xi, v}_M =0 \},\\
 E_u(x, \xi) \coloneqq\{(v, v) : \lrang{x, v}_M = \lrang{\xi, v}_M =0 \}.
\end{split}
\end{equation}

This definition can be extended to $|\xi|_g >0$ by setting $E_s$ and $E_u$ to be homogeneous.

From \cite{DFG15}*{\S 3.3}, we know that vector subbundles $E_s$ and $E_u$ are spanned respectively by $U_1^+, \ldots, U_n^+$ and $U_1^-, \ldots, U_n^-$ in the following sense:
\begin{equation}\label{eq:preimage_Es_Eu}
\pi_S^* E_s = \spn(U_1^+, \ldots, U_n^+) \oplus \mathfrak{h}, \quad \pi_S^* E_u = \spn(U_1^-, \ldots, U_n^-) \oplus \mathfrak{h},
\end{equation}
where $\mathfrak{h}$ is the left-translation of the Lie algebra of ${K_0}$ or equivalently the kernel of $d\pi_S$. Note that $\mathfrak{h}$ is spanned by $R_{i+1, j+1}$ for $1 \leq i < j \leq n$. Importantly, each individual vector field $U_i^\pm$ is not invariant under right multiplication by elements of ${K_0}$ for $n+1>2$. Therefore, $U_i^\pm$ do not descend to vector fields on $S\bH^{n+1}$ by the map $\pi_S$. However by ~\ref{eq:commutator}, $\spn (U_1^\pm, \ldots, U_n^\pm)$ are invariant under ${K_0}$.

In $T_{(x, \xi)}(T^*M)$, $E_s$ and $E_u$ are the images of the stable and unstable bundles of $\bH^{n+1}$ under the covering map 
\begin{equation}\label{eq:pi_gamma_def}
\pi_\Gamma: T^*\bH^{n+1} \rightarrow T^*M.
\end{equation}

Furthermore, $E_s$ and $E_u$ are invariant under the geodesic flow $\phi_t$. The projection map
$T_{(x, \xi)} T^* M \rightarrow T_x M$ is an isomorphism from $E_s(x, \xi)$ (or $E_u(x, \xi)$) onto the space $\{ \eta \in T_x M : \lrang{\xi, \eta} =0 \}$. Thus, using the projection map, we can canonically pull back the metric $g_x$ to $E_s(x, \xi)$  (or to $E_u(x, \xi)$). Following ~\cite{DFG15}*{\S 3.3}, we know
\begin{equation*}\label{eq:quant_stable/unstable}
|d\phi_t(x, \xi) w|_g= \begin{cases}e^t|w|_g, & w \in E_u(x, \xi);\\
e^{-t}|w|_g, & w \in E_s(x, \xi).\end{cases}
\end{equation*}
For each $(x, \xi) \in T^*M \setminus 0$, we call
\begin{equation}\label{eq:LsLu_def}
L_s(x, \xi) \coloneqq \R X(x, \xi) \oplus E_s(x, \xi) \quad \text{and} \quad L_u(x, \xi) \coloneqq \R X(x, \xi) \oplus E_u(x, \xi)
\end{equation}
respectively the \emph{weak stable subspace} and \emph{weak unstable subspace}. 

We have the following definition, as in ~\cite{DZ16}*{Definition 3.1}.
\begin{definition}\label{def:lagrangian_foliation}
Let $M$ be a manifold, $U \subset T^*M$ be an open set, and for $(x, \xi) \in U$, let $L_{(x, \xi)} \subset T_{(x, \xi)}(T^*M)$ be a family of subspaces depending smoothly on $(x, \xi)$. We say that $L$ is a \emph{Lagrangian foliation} on $U$ if 
\begin{itemize}
\item $L_{(x, \xi)}$ is integrable. Namely, if $X, Y \in C^\infty(U; L)$ are vector fields, then $[X, Y] \in C^\infty(U; L)$.
\item $L_{(x, \xi)}$ is a Lagrangian subspace of $T_{(x, \xi)}(T^*M)$ for each $(x, \xi) \in U$. 
\end{itemize}
\end{definition} 

From ~\cite{DZ16}*{Lemma 4.1}, we know that
 $L_s$ and $L_u$ are Lagrangian foliations on $T^*M \setminus 0$.


\subsection{$U_1^-$-dense sets}
\label{subsection:safe_sets}


Inspired by $V$-dense sets in \cite{ADM24} and safe sets in \cite{DJ18}, we define the following.

\begin{definition}\label{def:safe}
\leavevmode  \begin{itemize}
    \item A \emph{$U_1^-$-orbit} is a set of the form $e^{\R U_1^-}q$, where $q \in F^*M$.
    \item  A \emph{$U_1^-$-segment} is a set of the form
$\{e^{t U_1^-}q: a \leq t \leq b\}$, where $q \in F^*M$ and $a< b \in \R$. 
    \item A set $U \subset S^*M$ is \emph{$U_1^-$-dense} if $\pi^{-1}_S(U)$ intersects every $U_1^-$-orbit in $F^*M$. In other words, for each $q \in F^* M$, $U$ intersects $\pi_S(e^{\R U_1^-}q)$. 
    \item Finally, a set $U \subset S^*M$ is \emph{$U_1^-$-sparse} if for all $q \in  F^*M$, $\{t \in \R: \pi_S(e^{t U_1^-}(q)) \in U\}$ is countable.
\end{itemize}
\end{definition}
Clearly, if $U$ is $U_1^-$-sparse, then $S^*M \setminus U$ is $U_1^-$-dense. 

In the following two lemmas, we prove properties of $U_1^-$-dense sets that will help us construct a particular partition of unity. 

\begin{lemma}\label{lem:finite_safe}
If $U \subset S^*M$ is $U_1^-$-dense, then there exists some $T>0$ such that for all $q \in F^* M$, $\{\pi_S(e^{t U_1^-}q) : |t| \leq T\} \cap U \neq  \emptyset$.
\end{lemma}

\begin{proof}
Let $U$ be $U_1^-$-dense  and suppose that no such $T$ exists.  Then for all $j \in \N$, there exists $q_j \in F^*M$ such that $\pi_S(e^{cU_1^-}q_j) \cap U = \emptyset$ for all $c \in [-j,j ]$.  We pass to a convergent subsequence $q_j \rightarrow q_\infty \in F^*M$. Then for all $c \in \R$, $\pi_S (e^{c U_1^- } q_\infty) \cap U = \emptyset$, a contradiction.
\end{proof}

\begin{lemma}\label{lem:safe_compact}
If $U \subset S^*M$ is open and $U_1^-$-dense, there exists a compact and $U_1^-$-dense  $K \subset U$. 
\end{lemma}
\begin{proof}
We first take a compact exhaustion of $U$. Specifically, let $U = \bigcup_{j \in \N} K_j$, where each $K_j$ is compact and $K_j \subset K_{j+1}^\circ$. Suppose that none of the $K_j$'s are $U_1^-$-dense. In other words, for each $j$, there exists $q_j \in F^*M$ such that $\pi_S(e^{\R U_1^-} q_j ) \cap K_l = \emptyset$ for all $l \leq j$. 

We pass to a convergent subsequence $q_j \rightarrow q_\infty \in F^*M$. Then for all $l \in \N$, $\pi_S (e^{\R U_1^-}q_\infty) \cap  K_l^\circ = \emptyset$, a contradiction.
\end{proof}

We next construct a function $F$ which has a particular relationship with $U_1^-$-orbits. This relationship will be used to build a set $D \subset S^*M$ that is both $U_1^-$-sparse and $U_1^-$-dense, which in turn will be used to create a  partition of unity  in Lemma~\ref{lem:a1_a2_safe}.  Lemma~\ref{lem:a1_a2_safe}  is a key step in the proof of   Lemma~\ref{lem:partition_of_unity}. 
\begin{lemma}\label{lem:jets}
Let $\rho_0 \in S^*M$. There exists $N = N(n) \in \N$ and $F \in C^\infty(U_{\rho_0}; \R)$, where $U_{\rho_0} \subset S^*M$ is a coordinate neighborhood of $\rho_0$ such that
\begin{enumerate}[(1)]
    \item $F(\rho_0) =0$; \label{item:condition1}
    \item for each $q \in \pi_S^{-1}(\rho_0)$, there exists $j = j_q \in \{1, \ldots, N\}$ such that $\partial_t^j F(\pi_S(e^{tU_1^-} q))|_{t=0} \neq 0$; \label{item:condition2}
\end{enumerate}
\end{lemma}

\begin{proof}
Let $U_{\rho_0} \subset S^*M$ be a coordinate neighborhood of $\rho_0$. We know that for any $f \in C^\infty(U_{\rho_0}; \R)$,  $\partial_t^j f(\pi_S(e^{t U_1^-} q))|_{t=0} = (U_1^-)^j (f \circ \pi_S)|_{q}$.    
Therefore, ~\ref{item:condition2} depends only on the $N$-jet of $f$ at $\rho_0$, i.e., using local coordinates in $U_{\rho_0}$, on $\partial^\alpha f(\rho_0)$ for all $\alpha \in \N^{2n +1}$ with $0<|\alpha| \leq N$. Set $K \coloneqq K(N) = \# \{\alpha \in \N^{2n+1} : 0< |\alpha| \leq N\}$. We set $c_f \in \R^K$  to be the coefficients of the $N$-jet of $f$ at $\rho_0$, excluding the first coefficient.  

For some surjective and linear function $T_q: \R^K \rightarrow \R^N$, the condition $\partial_t^j f(\pi_S(e^{t U_1^-}q))|_{t=0} =0$ for $j=1, \ldots, N$ is equivalent to $T_q(c_f) = 0$. Note that $T^{-1}_q(0)$ is a $q$-dependent subspace of dimension $K-N$. 

We examine $\bigcup_{q \in \pi_S^{-1}(\rho_0)} T^{-1}_q(0)$. Note that $\dim (\pi_S^{-1} (\rho_0)) = \dim (\SO(n)) = \tfrac{n(n-1)}{2}$. Therefore, for $N > \tfrac{n(n-1)}{2}$, there exists $c \in \R^K$ such that $c \notin \bigcup_{q \in \pi_S^{-1}(\rho_0)} T^{-1}_q(0)$.

We conclude the proof by taking a function $F \in C^\infty(U_{\rho_0}, \R)$ with $N$-jet at $\rho_0$ determined by $c$ and $F(\rho_0)= 0$. 
\end{proof}

\begin{lemma}\label{lem:sparse}
Let $\rho_0 \in S^*M$ and let $F \in C^\infty(U_{\rho_0})$ be the function constructed in Lemma~\ref{lem:jets}. There exists some $U'_{\rho_0} \subset U_{\rho_0}$ and $c = c_{\rho_0}>0$ such that 
for each $q \in \pi_S^{-1}(U_{\rho_0}')$, for  some $j_q \in \{1, \ldots, N\}$,
\begin{equation}\label{eq:bounded_away_from_zero}
\left|\partial_t^{j_q} F(\pi_S(e^{t U_1^-} q))|_{t=0} \right| \geq c.
\end{equation} Additionally, for any $r \in \R$, $F^{-1}(r) \cap U'_{\rho_0}$ is $U_1^-$-sparse.  
\end{lemma}

\begin{proof} 
For each $q \in \pi_S^{-1}(\rho_0)$, there exists  $j_{q} \in \{1, \ldots, N\}$ such that $\partial_t^{j_q} F(\pi_S(e^{tU_1^-} q))|_{t=0} \neq 0$. 
Thus, for each  $q \in \pi_S^{-1}(\rho_0)$, we can select open $U_q \subset \pi_S^{-1}(U_{\rho_0})$ such that $q \in U_q$ and for some $c_q>0$, if $\tilde{q} \in U_q$, then $|\partial_t^{j_q} F(\pi_S(e^{t U_1^-} \tilde{q}))|_{t=0}| \geq c_q$. By the compactness of $\pi_S^{-1}(\rho_0)$, we know 
$$U'_{\rho_0} \coloneqq \bigcap_{q \in \pi^{-1}_S(\rho_0)} \pi_S(U_q) \subset U_{\rho_0}$$ 
is a nonempty open set and  $c \coloneqq \inf_{q \in \pi_S^{-1}(\rho_0)} c_q >0$. 
Therefore, for  each $q \in \pi_S^{-1}(U_{\rho_0}')$, there exists some $j_q \in \{1, \ldots, N\}$ such that 
\begin{equation*}
\left|\partial_t^{j_q} F(\pi_S(e^{t U_1^-} q))|_{t=0} \right| \geq c,
\end{equation*}
which completes the proof of ~\eqref{eq:bounded_away_from_zero}.

Now we show for any $r \in \R$, $F^{-1}(r) \cap U_{\rho_0}'$ is $U_1^-$-sparse.
Note that the projection onto $S^*M$ of each $U_1^-$-orbit has countably many segments that intersect  $F^{-1}(r) \cap U_{\rho_0}'$. 
Thus,
it suffices to show that for all $q \in F^*M$ and $t_0 \in \R$, if $\pi_S(e^{t_0 U_1^-} q) \in F^{-1}(r) \cap U_{\rho_0}'$, then there exists $j$ such that $\partial_t^j F_\rho(\pi_S(e^{t U_1^-} q))|_{t = t_0} \neq 0$.
As $e^{t_0 U_1^-} q \in \pi_S^{-1}(U_{\rho_0}')$, this follows from  \eqref{eq:bounded_away_from_zero}. 
\end{proof}

Finally, we construct our partition of unity, partly inspired by \cite{DJ23}*{Lemma 3.5} and~\cite{ADM24}*{Lemma 3.2}. 
\begin{lemma}\label{lem:a1_a2_safe}
Let $U \subset S^*M$ be open and $U_1^-$-dense. Then there exist $\chi_1, \chi_2 \in C^\infty(S^*M; [0,1])$ such that  $\chi_1 + \chi_2 =1$, $\supp \chi_1 \subset U$, and the complements $S^*M \setminus \supp \chi_1$, $S^*M \setminus \supp \chi_2$ are $U_1^-$-dense. 
\end{lemma}

\begin{proof}
1. We first  use  Lemma~\ref{lem:sparse} to construct a set $D \subset S^*M$ that is both $U_1^-$-dense and $U_1^-$-sparse. Fix $\rho_0 \in S^*M$ and take $F =F_{\rho_0} \in C^\infty(U_{\rho_0}; \R)$ from Lemma ~\ref{lem:jets} and $U_{\rho_0}' \subset U_{\rho_0}$  and $c_{\rho_0}$ from Lemma ~\ref{lem:sparse}.
Fix an open $U_{\rho_0}'' \Subset U_{\rho_0}'$. Then pick $R=R_{\rho_0}>0$ such that for every $q \in \pi_S^{-1}(U_{\rho_0}'')$, $\{\pi_S(e^{t U_1^-}q) : |t| \leq R\} \subset U_{\rho_0}'$. We assume that $R \leq 1$.

Let $q \in \pi_S^{-1}(U_{\rho_0}'')$. By~\eqref{eq:bounded_away_from_zero}, we can define 
$$P_q(t) = \left(\partial_t^{j_q} F(\pi_S(e^{t U_1^-} q))|_{t=0}\right)^{-1} \sum_{k=1}^{j_q} \partial_t^k F(\pi_S(e^{t U_1^-} q))|_{t=0} \frac{t^k}{k!}.$$
Using Chebyshev polynomials, one can show that among monic polynomials of degree $n$, the smallest maximal absolute value on the interval $[-1, 1]$ is 
$2^{1-n}$. Therefore,
we know that $\max_{|t| \leq R} |P_q(t)| \geq 2^{1-j_q} \geq 2^{1-N}$.

As $F$ is smooth, for $|t| \leq R$ and $1 \leq k \leq N+1$, $|\partial_t^k F(\pi_S(e^{t U_1^-} q))| \leq C$, where $C$ is uniform in $t$, $k$, and $q \in \pi^{-1}_S(U_{\rho_0}'')$.  Therefore, taking the Taylor expansion in $t$ and using \eqref{eq:bounded_away_from_zero},
\begin{align*}
\max_{|t| \leq R} |F(\pi_S(e^{t U_1^-} q)) - F(\pi_S(q))| &\geq \max_{|t| \leq R} \left(\left|\sum_{k=1}^{j_q} \partial_t^k F(\pi_S(e^{t U_1^-} q))|_{t=0} \frac{t^k}{k!} \right| - C \frac{|t|^{j_q +1}}{(j_q +1)!}\right)\\
&\geq c_{\rho_0}2^{1-N} - CR.    
\end{align*}

Possibly shrinking the value of $R$, there exists $\varepsilon = \varepsilon_{\rho_0}>0$ such that $c_{\rho_0}2^{1-N} - CR \geq \varepsilon$. We conclude
\begin{equation}\label{eq:grow_by_epsilon}
\max_{|t| \leq R} \left|F(\pi_S(q)) - F(\pi_S(e^{t U_1^-} q))\right| \geq \varepsilon.
\end{equation}

Now, by the continuity of $F$, there exists $r>0$ such that for all $\rho  \in U_{\rho_0}''$, $|F(\rho)| < r$. Set
$$D_{\rho_0} \coloneqq U'_{\rho_0} \cap \bigcup_{\substack{j \in \Z \\ |j| \leq \lrceil{\frac{r}{\varepsilon}}}} F^{-1}(j \varepsilon).$$
By \eqref{eq:grow_by_epsilon}, for each $q \in \pi_S^{-1}(U_{\rho_0}'')$, $\pi_S(e^{\R U_1^-} q)$ intersects $D_{\rho_0}$.

Now take a finite open cover of $S^*M$ by $U_{\rho}''$, indexed by $P=\{\rho_1, \ldots, \rho_K\} \subset S^*M$. 
We see that 
$D \coloneqq \bigcup_{\rho \in P} D_{\rho}$
is $U_1^-$-dense. 

By Lemma~\ref{lem:sparse}, $D$ is the finite union of $U_1^-$-sparse sets. Therefore, $D$ is  $U_1^-$-sparse.

2. Note that both $D$ and $U \setminus D$ are  $U_1^-$-dense. 
From Lemma ~\ref{lem:safe_compact}, we know there exists a $U_1^-$-dense and compact $K_1 \subset U \setminus D$. As $D \subset S^*M \setminus K_1$, we know $S^*M \setminus K_1$ is $U_1^-$-dense. Again by Lemma ~\ref{lem:safe_compact}, we can find a $U_1^-$-dense compact $K_2 \subset S^*M \setminus K_1$. We take a partition of unity on $\chi_1, \chi_2 \in C^\infty(S^*M; [0,1])$, $\chi_1 +\chi_2 =1$  subordinate the the cover of $S^*M = (U \setminus K_2) \cup (S^*M \setminus K_1)$:
$$\supp \chi_1 \subset U \setminus K_2, \quad \supp \chi_2 \subset S^*M \setminus K_1,$$
which completes the proof. 
\end{proof}


\subsection{Semiclassical definitions}
\label{subsection:semiclassical_definitions}


In this section, we introduce the semiclassical analysis used in this paper. First, we establish the following notational conventions. 
\begin{notation}\label{notation}
Suppose $(F, \|\cdot \|_F)$ is a normed vector space and $f_h \in F$ is a family depending on a parameter $h>0$. If $\|f_h\|_F = \cO(h^\alpha)$, we write $f_h = \cO(h^\alpha)_F$. 
\end{notation}

\begin{notation}\label{notation2}
We use $C$ to denote a positive constant, the value of which varies in each appearance. 
\end{notation}

We follow the exposition in ~\cite{DZ16}*{\S 2.1}, starting by recalling the standard class of semiclassical pseudodifferential operators $\Psi_h^k(M)$ with symbols $S^k_h(T^*M)$. A function $a(x, \xi;h) \in C^\infty(T^*M)$ is in $S^k_h(T^*M)$ if
\begin{enumerate}[(1)]
    \item for each compact set $K \subset M$, $|\partial_x^\alpha \partial_\xi^\beta a(x, \xi)| \leq C_{\alpha \beta K} \lrang{\xi}^{k -|\beta|}$ for $x \in K$;
    \item $a(x, \xi; h) \sim \sum_{j=0}^\infty h^j a_j(x, \xi)$ as $|\xi| \rightarrow \infty$, where $a_j$ is positively homogeneous in $\xi$ of degree $k -j$.
\end{enumerate}

For $A \in \Psi_h^k(M)$, $\WF_h(A)$ is a closed subset of the fiber-radially compactified
cotangent bundle $\overline{T}^* M \supset T^*M$. Let $\Psi_h^{\comp}(M)$ be the subset of $A \in \Psi_h^k(M)$ for which $\WF_h(A)$ is a compact subset of $T^*M$.

For $M= \R^n$, we use the standard quantization,
\begin{equation*} \label{eq:quantization}
\op_h(a)f(y)=(2\pi h)^{-(n+1)}\int_{\R^{2n+2}}e^{i\lrang{y-y', \eta}/h}a(y,\eta)f(y')\,dy'd\eta, \quad a \in S^k(T^* \R^{n+1}).
\end{equation*}

Following~\cite{DZ19}*{\S E.1.7} and~\cite{Zw12}*{\S 14.2.2}, we can generalize this quantization to pseudodifferential operators on manifolds $\Psi^k_h(M)$. We denote the  \emph{principal symbol map} by 
$$\sigma_h: \Psi^k_h(M) \rightarrow S^k(T^*M).$$

We now define semiclassical measures. 

\begin{definition}\label{def:semiclassical_measure}
Suppose that $u_j$ is a sequence of $L^2$-normalized eigenfunctions of $-\Delta$ with eigenvalues $h_j^{-2} \rightarrow \infty$, i.e.,
$$(-h_j^2 \Delta -I)u_j =0, \quad \|u_j\|_{L^2} =1.$$
We assume $h_j>0$. We say that $u_j$ \emph{converge semiclassically} to a probability measure $\mu$ on $T^*M$ if
$$\lrang{\op_{h_j}(a) u_j, u_j}_{L^2(M)} \rightarrow \int_{T^*M} a \, d\mu \quad \text{as } j \rightarrow \infty \text{ for all } a \in C_c^\infty(T^*M).$$
We call such a measure $\mu$ a \emph{semiclassical measure}.
\end{definition}
Semiclassical measures are geodesic-flow invariant probability measures with support contained in $S^*M$ (see \cite{Zw12}*{\S \S 5.2--1}).

For $A, B \in \Psi^k_h(M)$ and an open set $U \subset \overline{T}^*M$, we say that 
$$A=B +\cO(h^\infty) \text{ \emph{microlocally in} } U$$
if $\WF_h(A-B) \cap U = \emptyset$.

Now let $B=B(h): \cD'(M) \rightarrow C^\infty_c(M)$ be an $h$-tempered family of smoothing operators. Further assume $\WF'_h(B) \subset \overline{T}^*(M \times M)$ is a compact subset of $T^*(M \times M)$. We say that $B$ is \emph{pseudolocal} if $\WF'_h(B)$ is contained in the diagonal $\Delta(T^*M) \subset T^*(M \times M)$. For a pseudolocal operator $B$, we define $\WF_h(B) \subset T^*M$ to be the set that satisfies
\begin{equation*}\label{eq:WF'_def}
\WF'_h(B) = \{(x, \xi, x, \xi) : (x, \xi) \in \WF_h(B)\}.
\end{equation*}

Elements of $\Psi_h^{\comp}(M)$ are pseudolocal with a wavefront set that matches the above definition. 


\subsection{Propagating operators} \label{subsection:propagate}


Recall from \S\ref{subsection:hyperbolic_manifolds} that $\phi_t: T^*M \setminus 0 \rightarrow T^*M \setminus 0$ is the homogeneous geodesic flow.
Let $p \in C^\infty(T^*M \setminus 0)$ be given by the following:
\begin{equation}\label{eq:p_def}
p(x, \xi)=|\xi|_g, \quad (x, \xi) \in T^*M\setminus 0.
\end{equation}

We see that $-h^2 \Delta$ lies in $\Psi_h^2(M)$ and 
$$\sigma_h(-h^2 \Delta) = p^2.$$

We fix
$$\psi_P \in C_c^\infty((0, \infty); \R), \quad \psi_P(\lambda) = \sqrt{\lambda} \quad \text{for } \frac{1}{16} \leq \lambda \leq 16,$$
and define 
\begin{equation}\label{eq:P_def}
P \coloneqq \psi_P(-h^2 \Delta), \quad P^* = P.    
\end{equation}

For further background on $\psi_P(-h^2 \Delta)$, see ~\cite{Zw12}*{\S 14.3.2}.

Then, 
$$P \in \Psi_h^{\comp}(M) \quad \text{and} \quad \sigma_h(P) =p \quad \text{on} \quad \{1/4 \leq |\xi|_g \leq 4\}.$$

To quantize the flow $\phi_t$, we define the unitary operator
\begin{equation*}\label{eq:U_def}
U(t) \coloneqq \exp\left(-\frac{itP}{h} \right) : L^2(M) \rightarrow L^2(M).    
\end{equation*}
Then for a bounded operator $A:L^2(M) \rightarrow L^2(M)$, define the time-dependent symbol $A(t)$ by propagating by $U(t)$:
\begin{equation}\label{eq:T_notation}
A(t) \coloneqq U(-t)AU(t).
\end{equation}


\subsection{Symbol class}\label{subsection:symbol_class}


Suppose $A \in \Psi^{\comp}_h(M)$ with $\WF_h(A) \subset \{1/4 < |\xi|_g < 4\}$ and $t$ is uniformly bounded in $h$. Then, Egorov's theorem ~\cite{Zw12}*{Theorem 11.1} gives 
\begin{equation*}\label{eq:egorov_A(t)}
A(t) \in \Psi_h^{\comp}(M), \quad \sigma_h(A(t)) = \sigma_h(A) \circ \phi_t.
\end{equation*} 
However, if instead $t$ grows with $h$, the derivatives of $\sigma_h(A) \circ \phi_t$ may grow exponentially with $t$. Thus, $A(t)$ may no longer lie in $\Psi_h^{\comp}(M)$. To handle this possibility, we describe a symbol class  $S^{\comp}_{L, \rho, \rho'}(U)$, first introduced in ~\cite{DJ18}*{\S A.1}. Specifically, we define a symbol class that contains $\sigma_h(A(t))$ for $0 \leq t \leq \rho \log h^{-1}$, $\rho <1$. 

\begin{definition}\label{def:symbol_class}
Fix two parameters $\rho, \rho'$ such that
$$0 \leq \rho <1, \quad 0 \leq \rho' \leq \frac{\rho}{2}, \quad \rho + \rho'<1.$$ 
Let $L$ be a Lagrangian foliation. 
We say that an $h$-dependent symbol $a$ lies in the class $S^{\comp}_{L, \rho, \rho'}(U)$ for an open set $U \subset T^*M$ if 
\begin{itemize}
    \item $a(x, \xi; h)$ is smooth in $(x, \xi) \in U$,  defined for $0 < h \leq 1$, and supported in an $h$-independent compact subset of $U$;
    \item $a$ satisfies the derivative bounds
    \begin{equation}\label{eq:derivative_bounds}
    \sup_{(x, \xi) \in U} |Y_1 \ldots Y_m Z_1 \ldots Z_k a(x, \xi;h)| \leq Ch^{-\rho k - \rho' m}, \quad 0 < h \leq 1,
     \end{equation}
    for all vector fields $Y_1, \ldots, Y_m, Z_1, \ldots, Z_k$ on $U$ such that $Y_1, \ldots, Y_m$ are tangent to $L$. The constant $C$ depends on the vector fields, but must be uniform in $h$.
\end{itemize}
\end{definition}

Recall the definitions of the Lagrangian foliations $L_s$ and $L_u$ from \eqref{eq:LsLu_def}. 
From the argument in \cite{DJ18}*{\S 2.3}, if $a \in C_c^\infty(T^*M \setminus 0)$ is an $h$-independent symbol,  uniformly in  $t \in [0, \rho \log h^{-1}]$,
\begin{equation}\label{eq:log_propagate}
a \circ \phi_t \in S^{\comp}_{L_s, \rho, 0} (T^*M \setminus 0) \text{ and } a \circ \phi_{-t} \in S^{\comp}_{L_u, \rho, 0} (T^*M \setminus 0).
\end{equation}
This follows from using \eqref{eq:preimage_Es_Eu} to rewrite the derivative bounds \eqref{eq:derivative_bounds} in terms of the frame in \eqref{eq:frame}, then using  the commutation relations in \eqref{eq:commutator}.   

From ~\cite{DJ18}*{Lemma A.8}, we know for an $h$-independent $a\in C_c^\infty(\{1/4<|\xi|_g<4\})$, uniformly in $t\in [0,\rho\log h^{-1}]$,
\begin{equation}  \label{eq:egorov_log}
\begin{split}
U(-t)\op_h(a)U(t)&=\op_h^{L_s}(a\circ\varphi_t)+\mathcal O(h^{1-\rho}\log h^{-1})_{L^2\to L^2},\\
U(t)\op_h(a)U(-t)&=\op_h^{L_u}(a\circ\varphi_{-t})+\mathcal O(h^{1-\rho}\log h^{-1})_{L^2\to L^2}.
\end{split}
\end{equation}
We remark that~\cite{DJ18}*{Lemma A.8} is stated for hyperbolic surfaces, however the proof holds in all dimensions. 

\subsubsection{Fourier integral operators}\label{subsection:FIO}
We now review Fourier integral operators associated to symplectomorphisms (also known as canonical transformations), summarizing the exposition in ~\cite{DZ16}*{\S 2.2}.
Let $\kappa: U_2 \rightarrow U_1$ be a symplectomorphism, where $U_j \subset T^*M_j$ are open sets and $M_j$ are manifolds of equal dimension. Define the graph of $\kappa$ by 
\begin{equation*}\label{eq:graph_def}
\Gr(\kappa) \coloneqq \{(x, \xi, y, \eta) : (y, \eta) \in U_2, (x, \xi) = \kappa(y, \eta)\} \subset T^*(M_1 \times M_2).
\end{equation*}
Let $\xi dx$ and $\eta dy$ be the canonical 1-forms on $T^*U_1$ and $T^*U_2$, respectively. We require that $\kappa$ is an \emph{exact symplectomorphism}, in other words, $(\xi dx -\eta dy)|_{\Gr(\kappa)}$ is an exact form. Fix an antiderivative $F \in C^\infty(\Gr(\kappa))$, i.e., $(\xi dx - \eta dy)|_{\Gr(\kappa)} = dF$. 

Let $I^{\comp}_h(\kappa)$ be the class of compactly supported and microlocalized Fourier integral operators associated to $\kappa$. We have $I^{\comp}_h(\kappa): \cD'(M_2) \rightarrow C_c^\infty(M_1)$. For the properties of $I^{\comp}_h(\kappa)$ and further references, see ~\cite{DZ16}*{\S 2.2}.

Assume that $B \in I^{\comp}_h(\kappa)$, $B' \in I^{\comp}_h(\kappa^{-1})$. Then, $BB' \in \Psi_h^{\comp}(M_1)$, $B'B \in \Psi_h^{\comp}(M_2)$, $\WF_h(BB') \subset U_1$, $\WF_h(B'B) \subset U_2$, and 
\begin{equation*}\label{eq:composition_BB'}
\sigma_h(B'B) = \sigma_h(BB') \circ \kappa.
\end{equation*}

We say that $B$, $B'$ \emph{quantize} $\kappa$ \emph{near} $V_1 \times V_2$ (for compact subsets $V_j \subset U_j$ such that $\kappa(V_2) = V_1$) if 
\begin{equation}\label{eq:quantize_kappa}
\begin{split}
BB' &= 1 + \cO(h^\infty) \text{ microlocally near } V_1,\\
B'B &= 1 + \cO(h^\infty) \text{ microlocally near } V_2.
\end{split}
\end{equation}

Following \cite{DZ16}*{\S 2.2}, it can be shown that $B, B'$ exist if $V_2$ is a sufficiently small neighborhood of a point.

\subsubsection{Quantization}
\label{subsubsection:quantization}
Let $L$ be a Lagrangian foliation on $U \subset T^*M$. We introduce a quantization $\op_h^L$, first formulated in ~\cite{DZ16}*{\S 3.3}, that respects the structure of $L$. We follow the presentation in ~\cite{DJ18}*{\S A.4}.

Using standard coordinates $(y, \eta)$ on $T^*\R^{n+1}$, we denote the vertical foliation on $T^*\R^{n+1}$ by 
$$L_0 \coloneqq \spn(\partial_{\eta_1}, \ldots, \partial_{\eta_{n+1}}) = \ker dy.$$

We call $(U', \kappa, B, B')$ a \emph{chart} for $L$ if $U' \subset U$ is an open set, $\kappa: U' \rightarrow T^*\R^n$ is an exact symplectomorphism onto its image with $d \kappa(x, \xi) \cdot L_{(x, \xi)} = (L_0)_{\kappa(x,\xi)}$, $B \in I^{\comp}_h(\kappa)$, and $B' \in I^{\comp}_h(\kappa^{-1})$. From ~\cite{DZ16}*{Lemma 3.6} and the paragraph following ~\cite{DZ16}*{(2.12)}, for each $(x_0, \xi_0) \in U$, there exists a chart $(U', \kappa, B, B')$ such that $\sigma_h(B'B)(x_0, \xi_0) \neq 0$. 

\begin{definition}\label{def:Op_h^L}
Let $a \in S_{L, \rho, \rho'}^{\comp}(U)$ and let $(U_l, \kappa_l, B_l, B_l')$ be a collection of charts for $L$ such that $U_l \subset U$ form a locally finite cover of $U$, $\sigma_h(B_l' B_l) \in C_c^\infty(U_l)$ is a partition of unity on $U$. Choose $\chi_l \in C_c^\infty(U_l)$  equal to 1 in a neighborhood of $\supp \sigma_h(B_l' B_l)$. Then,
$$\op_h^L(a) \coloneqq \sum_l B_l' \op_h(a_l) B_l, \quad a_l \coloneqq (\chi_l a) \circ \kappa_l^{-1} \in S^{\comp}_{L_0, \rho, \rho'}(T^* \R^n).$$
\end{definition}

We know $\op_h^L$ depends on the choice of charts, but the class of operators does not.

For a compactly supported
operator $A:L^2(M)\to L^2(M)$, we say that $A\in\Psi^{\comp}_{h,L,\rho,\rho'}(U)$
if $A=\op_h^L(a)+\mathcal O(h^\infty)_{L^2\to L^2}$ for some
$a\in S^{\comp}_{L,\rho,\rho'}(U)$. We cite the following properties of the quantization
procedure $\op_h^L$ from ~\cite{DJ18}*{\S A.4}.
\begin{enumerate_eq}{}
\item For each $a\in S^{\comp}_{L,\rho,\rho'}(U)$, the operator
$\op_h^L(a):L^2(M)\to L^2(M)$ is compactly supported and
bounded uniformly in $h$. \label{eq:bounded_property} 
\item \label{eq:FIO_conjugation} Assume that $M_1,M_2$ are manifolds of the same dimension,
$U_j\subset T^*M_j$ are open sets,
$L_j$ are Lagrangian foliations on $U_j$,
$U'_j\subset U_j$ are open,
$\varkappa:U'_2\to U'_1$ is an exact symplectomorphism mapping
$L_2$ to $L_1$,
and $B\in I^{\comp}_h(\varkappa)$, $B'\in I^{\comp}_h(\varkappa^{-1})$.
Then for each $a_1\in S^{\comp}_{L_1,\rho,\rho'}(U_1)$, there exists
$a_2\in S^{\comp}_{L_2,\rho,\rho'}(U_2)$ such that
\begin{equation*}
\begin{aligned}
B'\op_h^{L_1}(a_1)B&=\op_h^{L_2}(a_2)+\mathcal O(h^\infty)_{L^2\to L^2},\\
a_2&=(a_1\circ\varkappa)\sigma_h(B'B)+\mathcal O(h^{1-\rho})_{S^{\comp}_{L_2,\rho,\rho'}(U_2)},\\
\supp a_2&\subset \varkappa^{-1}(\supp a_1).
\end{aligned}
\end{equation*}
\item 
\label{eq:composition_formula}
For each $a,b\in S^{\comp}_{L,\rho,\rho'}(U)$, there exists
$a\#_L b\in S^{\comp}_{L,\rho,\rho'}(U)$ such that
\begin{equation*}
\begin{aligned}
\op_h^L(a)\op_h^L(b)&=\op_h^L(a\#_L b)+\mathcal O(h^\infty)_{L^2\to L^2},\\
a\#_L b&=ab+\mathcal O(h^{1-\rho-\rho'})_{S^{\comp}_{L,\rho,\rho'}(U)},\\
\supp (a\#_L b)&\subset \supp a\cap \supp b.
\end{aligned}
\end{equation*}
\end{enumerate_eq}

Finally, define the symbol class
$$S_{L, \rho}^{\comp}(T^*M \setminus 0) \coloneqq \bigcap_{\varepsilon>0} S^{\comp}_{L, \rho + \varepsilon, \varepsilon} (T^*M \setminus 0).$$
When working with this symbol class, we often employ the following notation. 
\begin{notation}\label{notation:epsilon}
We write $f(h) = \cO(h^{\alpha-})$ if $f(h) = \cO(h^{\alpha-\varepsilon})$ for all $\varepsilon>0$.
\end{notation}


\section{Totally geodesic submanifolds and orbit closures}\label{section:orbit}


In this section, we classify orbit closures in $F^*M$ under the $AU^\pm$-action, where $A=\{e^{tX}\}$, $U^\pm=\{e^{tU_1^\pm}\}$.
This will prove Theorem \ref{thm:orbit_closure} upon projecting to $S^*M$. 
After proving Theorem \ref{thm:orbit_closure}, we will also discuss what is presently known about totally geodesic submanifolds in higher-dimensional hyperbolic manifolds in \S \ref{section:TGexamples}.


\subsection{Generalities on geodesic submanifolds and their frame bundles}\label{subsection:TGsubmanifolds} 


Throughout this section, we continue to let $(M,g)$ denote a compact hyperbolic $(n+1)$-manifold. 
Recall from \S \ref{subsection:hyperbolic_manifolds}, that we may identify $\bH^{n+1}$ with $G/K$, where $G=\SO_0(1,n+1)$ and $K\subset G$ is the maximal compact subgroup defined in \eqref{eqn:Kequation}.
Given any natural number $2\le \ell\le n+1$, there is an isometric embedding of $\bH^\ell$ into $\bH^{n+1}$ given by setting $x_j=0$ for all $j>\ell$ in \eqref{eqn:hyperbolicspace}.
We call such an embedding the \emph{standard embedding} of $\bH^\ell$ into $\bH^{n+1}$ and denote it by $\bH^\ell_{\mathrm{std}}$.
We use $W_{\ell}$ to denote the subgroup of $G$ defined by
\begin{equation}\label{eqn:Wequation}
W_{\ell} \coloneqq   \left\{ 
\begin{bmatrix}
B&0_{(\ell+1)\times (n-\ell+1)}\\
0_{(n-\ell+1)\times(\ell+1)}&\mathrm{Id}_{(n-\ell+1)\times (n-\ell+1)}\end{bmatrix} : B \in \SO_0(1,\ell) \right\},
\end{equation}
which we call a \emph{standard subgroup} of $G$.
We also use the notation $K_{\ell}=K\cap W_{\ell}\cong\SO(\ell)$ to denote a maximal compact subgroup of $W_{\ell}$, embedded similarly.

With \eqref{eqn:Wequation} in mind, the identification of $\bH^{n+1}$ with $G/K$ yields the left $W_{\ell}$-equivariant identifications
$$\bH^\ell= W_{\ell}/ K_\ell\cong \bH^\ell_{\mathrm{std}}\simeq W_\ell K/K\subseteq G/K\simeq \bH^{n+1}.$$
Moreover, as $G$ acts transitively on isometric copies of $\bH^\ell$ in $\bH^{n+1}$, it follows that totally geodesic embeddings of $\bH^\ell$ in $\bH^{n+1}$ are in one-to-one correspondence with subsets of $G/K$ of the form $gW_\ell K/K$ for elements $g\in G$.
We freely pass between $\bH^{n+1}$ and $G/K$ using these identifications.

Given a hyperbolic $\ell$-manifold $X$, we say that $X$ \emph{totally geodesically immerses} in $M$ if there is a proper immersion $\iota:X\to M$ such that some (equivalently, any) lift of $\iota$ to $\widetilde{\iota}:\bH^\ell\hookrightarrow \bH^{n+1}$ is a totally geodesic embedding.
In this instance, we call $\iota(X)$ a \emph{totally geodesic submanifold} and suppress the reliance on $\iota$ in the sequel.
Though closed geodesics in $M$ fit the above definition, our convention in this paper is that a totally geodesic submanifold always has dimension at least $2$.
Note that the stabilizer of the standard embedding of $\bH^\ell_{\mathrm{std}}\simeq W_{\ell}K/K$ in $G$ is precisely given by the block diagonally embedded
$$N_G(W_{\ell})=\mathrm{S}(\mathrm{O}_0(1,\ell)\times \mathrm{O}(n-\ell+1))\subset G,$$
where $\mathrm{O}_0(1,\ell)$ is the subgroup of $\mathrm{O}(1,\ell)$ of index two preserving $\bH^{\ell}_{\mathrm{std}}$.
Therefore the stabilizer of any geodesic plane $gW_{\ell}K/K$ in $G/K$ is given precisely by $N_G(W_{\ell})^g:=gN_G(W_{\ell})g^{-1}$.
This can be deduced from Lemma \ref{lem:Wnormalizer}.
Importantly, the condition that $\iota$ is a proper immersion is equivalent to the condition that $\Gamma\cap N_G(W_{\ell})^g$ is a lattice in $N_G(W_{\ell})^g$ for some (equivalently, any) lift $gW_{\ell}K/K$ of $X$ to $ G/K$.

Given an immersed, totally geodesic $\ell$-submanifold $X\subset M$ and a point $x\in X$, let $F_M(X)$ be the bundles over $X$ given by the subbundle of the frame bundle of $M$ restricted to $X$, $FM\vert_X$, such that the fiber over $x\in X$ is the subset of frames whose first $\ell$ vectors are tangent to $X$. 
This is a principal $\mathrm{S}(\mathrm{O}(\ell)\times \mathrm{O}(n-\ell+1))$-bundle over $X$.
As in \S \ref{subsection:hyperbolic_manifolds}, we identify the frame bundle of $\bH^{n+1}$ with $G$ and such an identification is equivariant for the right $K$- and left $\Gamma$-actions.
In particular, in this way $\Gamma\backslash G$ is identified with $FM$ and $\Gamma\backslash G/K$ is identified with $M$.

Combining the above, a totally geodesic submanifold $X$ is given by a subset $\Gamma\backslash\Gamma gW_{\ell}K/K$ of $\Gamma\backslash G$, which is the image of $gW_{\ell}K/K$ under the map from the universal cover.
It is straightforward to verify that $gN_G(W_{\ell})$ is precisely the bundle $F_{\bH^{n+1}}(gW_{\ell}K/K)$.
The naturality of these bundles with respect to the left and right actions then yields that $F_{M}(X)=\Gamma\backslash\Gamma gN_G(W_{\ell})\subseteq \Gamma\backslash G\simeq FM$.

Continuing to follow \S \ref{subsection:hyperbolic_manifolds}, we identify the coframe bundle $F^*M$ and the frame bundle $FM$ using the Riemannian metric, and similarly the cosphere bundle $S^*M$ and the sphere bundle $SM$.
These identifications are equivariant for the actions described in \S \ref{subsection:hyperbolic_manifolds}.
We also abusively use the same notation for maps when passing between the frame (resp. sphere) and coframe (resp. cosphere) bundles, e.g., we continue to denote by $\pi_S:FM\to SM$, the natural projection.

Recall that the right quotient of $FM\simeq \Gamma\backslash G$ by $K_0$ is identified with the sphere bundle $SM$.
In particular, the map $\pi_S:FM\to SM$ is the natural quotient map, which we also abusively notate the same before and after the corresponding equivariant identification.
Therefore
$$\pi_S\left(F_{\bH^{n+1}}(gW_{\ell}K/K)\right)=\pi_S\left(\Gamma\backslash\Gamma gN_G(W_{\ell})\right)=\Gamma\backslash\Gamma g N_G(W_{\ell})K_0/K_0=\Gamma\backslash\Gamma g W_{\ell}K_0/K_0,$$
where in Lemma \ref{lem:Wnormalizer}, we will compute $N_G(W_{\ell})$ explicitly and in Corollary \ref{cor:WN}, verify that it satisfies $W_{\ell}K_0=N_G(W_{\ell})K_0$ and similarly $W_{\ell}K=N_G(W_{\ell})K$.
Due to the naturality above, the subset $g W_{\ell}K_0/K_0$ is precisely the image of the immersion of the sphere bundle, $SX$, of $X$ into $SM$ induced by inclusion. 
In this way, there is a correspondence between closed subsets of $\Gamma \backslash G$ of the form $\Gamma\backslash\Gamma gS$ for some $W_{\ell}\subseteq S\subseteq N_G(W_{\ell})$ and immersed, totally geodesic $\ell$-dimensional submanifolds $X$ of $M$, as well as immersions of their sphere bundles (see \cite[Lemma 3.2]{BFMS} for more on this).


\subsection{Group theoretic preliminaries}\label{subsection:groups}


In this subsection, we collect some group-theoretic preliminaries which we will need later in this section.
Some of the statements herein may seem unmotivated but their utility will become clear by the end of the section.

Continuing the notation from \S \ref{subsection:hyperbolic_manifolds}, we define 
\begin{align}\label{eqn:Aequation}
A& \coloneqq \langle e^{tX}: t\in\R\rangle=\left\{ 
\begin{bmatrix}
\cosh(t)&\sinh(t)&0_{1\times n}\\
\sinh(t)&\cosh(t)&0_{1\times n}&\\
0_{n\times 1}&0_{n\times 1}&\mathrm{Id}_{n\times n}
\end{bmatrix} : t\in\R \right\},\\\nonumber
U^{\pm}&\coloneqq\langle e^{sU_1^\pm}:s\in \R\rangle=\left\{ 
\begin{bmatrix}
1+s^2/2&\mp s^2/2&s&\cdots&0\\
\pm s^2/2&1-s^2/2&\pm s&\cdots&0\\
s& \mp s&1&\cdots&0\\
\vdots&\vdots&\vdots&\ddots&\vdots\\
0&0&0&\cdots&1
\end{bmatrix} : s\in\R \right\}.
\end{align}
Several times throughout we will use the $KAN$-decomposition of $G$, that is, $G=KAN^\pm$ where $K$ is the maximal compact, $A$ as in \eqref{eqn:Aequation}, and either choice of $N^\pm$ is the full horospherical group given by
\begin{align}
N^\pm&\coloneqq\langle e^{s_1U_1^\pm}\dots e^{s_nU_n^\pm}:(s_1,\dots,s_n)\in \R^n\rangle \nonumber\\
&=\left\{ 
\begin{bmatrix}
1+|\vec{v}|^2/2&\mp |\vec{v}|^2/2&\vec{v}&\\
\pm |\vec{v}|^2/2&1-|\vec{v}|^2/2&\pm\vec{v}\\
\vec{v}^T&\mp\vec{v}^T&\mathrm{Id}_{n\times n}
\end{bmatrix} : \vec{v}\in\R^n \right\}\cong\R^n,\label{eqn:horospherical}
\end{align}
where in \eqref{eqn:horospherical} we consider one entry of $\vec{v}$ (resp. $\vec{v}^T$) per column (resp. row).

A choice of minimal parabolic $P^\pm$ which has unipotent radical $N^\pm$ is therefore given by $P^\pm=K_0AN^\pm$, with $K_0$ as in \eqref{eqn:K0defn}; all other minimal parabolics are $G$-conjugate to $P^\pm$.
This decomposition of $P^\pm$ is known as the Langland's decomposition, where in the literature typically one uses ``$M$'' as opposed to ``$K_0$''.
Importantly, $N_G(N^\pm)=P^\pm$ and there is an isomorphism of $N^{\pm}$ with $\R^n$, induced by sending an element $b\in N^{\pm}$ to the corresponding vector $\vec{v}$, which intertwines the action by conjugation of $K_0$ on $N^{\pm}$ and the linear action of $K_0\cong \SO(n)$ on $\R^n$.
Moreover, under such an isomorphism, the conjugation action by $A$ on $N^{\pm}$ is intertwined with the action of $\R$ on $\R^n$ by scaling; precisely, by scaling by a factor of $\cosh(t)+\sinh(t) = e^t$.

We use these facts to deduce a few group theoretic lemmas, which we require in the sequel.
All of these facts are standard, but we provide proofs for completeness.

\begin{lemma}\label{lem:Unormalizer}
If $U=U^\pm$ and $K_U=(K\cap N_G(U^\pm))$, then $N_G(U^\pm)=N^\pm AK_U$.
Moreover, $K_U$ is explicitly given by the block diagonal embedding of $\mathrm{S}(\mathrm{O}(1)\times \mathrm{O}(n-1))$ in $K_0\cong \SO(n)$ in \eqref{eqn:K0defn}.
\end{lemma}

We remark that the lack of a superscript in $K_U$ is intentional since this group is the same regardless of the choice of $U^\pm$.

\begin{proof}
We give the argument for $U=U^+$, $N=N^+$, as the argument for $U^-$, $N^-$ is similar.

We first prove the statement in the first sentence.
By the remarks immediately preceding this lemma and the fact that $N\cong\R^n$ is abelian, it is clear that $N AK_U\subseteq N_G(U)$.
For the reverse inclusion, let $g\in N_G(U)$, then, using the $KAN$-decomposition, we may write $g=kab$ for $k\in K$, $a\in A$, $b\in N$.\footnote{Throughout, we make the unfortunate notation choice of $b\in N$ to avoid overloading the letter $n$ for dimension.}
As $AN\subset N_G(U)$, it follows that $g\in N_G(U)$ if and only if $k\in N_G(U)$.
Therefore $k\in N_G(U)\cap K=K_U$ and hence $N_G(U)=K_UAN=N AK_U$, as required.
Here, the second equality comes from taking inverses.

For the explicit description in the second sentence, we first reduce to showing that $K_U\subset K_0$.
To this end, note that $U\subset N$ stabilizes the point $\infty=(1,1,0,\dots,0)\in\partial\bH^{n+1}$.
Therefore, if $k\in N_G(U)$ then $k\cdot\infty=\infty$, where this action is the action of $K$ on $\partial \bH^{n+1}$ induced from the linear action of $G$ on $\R^{1,n+1}$.
In particular, the description of $K$ in \eqref{eqn:Kequation} shows immediately that $k\in N_G(U)$ implies $k\in K_0$. 
The result then follows from the remarks preceding this lemma.
Indeed, under the equivariant isomorphism of $N$ with $\R^n$, $U$ is sent to the $1$-dimensional subspace of vectors for which all coordinates are zero except possibly the first.
It is then clear that the linear action of $k\in K_0$ on $\R^n$ stabilizes this subspace if and only if $k\in K_U$, for $K_U$ as described.
\end{proof}

\begin{lemma}\label{lem:Wnormalizer}
Fix any $\ell\ge 2$ then 
\begin{equation}\label{eqn:Wnormalizer}
N_G(W_{\ell})=\mathrm{S}(\mathrm{O}_0(1,\ell)\times \mathrm{O}(n-\ell+1))=(\mathrm{O}(1,\ell)\times \mathrm{O}(n-\ell+1))\cap G,
\end{equation}
where $\mathrm{O}_0(1,\ell)$ is block embedded in the upper lefthand corner of $\mathrm{GL}(n+2,\R)$ and $\mathrm{O}(n-\ell+1)$ is block embedded in the lower righthand corner of $\mathrm{GL}(n+2,\R)$.
\end{lemma}

\begin{proof}
The second equality in \eqref{eqn:Wnormalizer} is straightforward from the definition of $G$, so it suffices to show the first. 
To this end, we define 
$$H=(\mathrm{O}(1,\ell)\times \mathrm{O}(n-\ell+1))\cap G,$$
and verify that this is the full normalizer of $W_\ell$ in $G$.

A straightforward computation shows that $H\subseteq N_G(W_\ell)$, so it suffices to prove the reverse inclusion.
Since $A$ normalizes $N=N^+$, the $KAN$-decomposition of $G$ can be written as $G=KNA$.
Let $g\in G$ and write $g=kba$.
As $A\subset W_{\ell}$, one sees that $g\in N_G(W_{\ell})$ if and only if $kb\in N_G(W_{\ell})$.

We first show that $kb\in N_G(W_\ell)$ implies $k\in K\cap H$.
Indeed, recall that $W_\ell$ stabilizes $\bH^{\ell}_{\mathrm{std}}$ as defined at the beginning of \S \ref{subsection:TGsubmanifolds} and hence stabilizes $\partial \bH^{\ell}_{\mathrm{std}}$ under the induced action of $G$ on $\partial\bH^{n+1}$.
Using \eqref{eqn:hyperbolicspaceboundary}, one verifies that $\partial \bH^{\ell}_{\mathrm{std}}$ is identified with the subset of $\partial \bH^{n+1}$ where $x_j=0$ for all $j>\ell$.
Let $\infty=(1,1,0,\dots,0)\in \partial \bH^{\ell}_{\mathrm{std}}$ as in the proof of Lemma \ref{lem:Unormalizer}, which is stabilized by $N$.
Then $kb\in N_G(W_\ell)$ implies that $kb\cdot \infty=k\cdot \infty\in\partial \bH^{\ell}_{\mathrm{std}}$.
As $K_\ell$ acts transitively on $\partial \bH^{\ell}_{\mathrm{std}}$, there exists $k'\in K_\ell$ for which $k'k\cdot \infty=\infty$.
Therefore $k'k\in\mathrm{stab}_G(\infty)$, which is precisely the parabolic subgroup $P^+$.
In particular, $k'k\in K_0=P^+\cap K$ and hence $k'k$ normalize $N$. 
Again, using the equivariant identification of $N$ with $\R^n$, we see that $W_\ell\cap N$ is identified with the vectors for which the first $\ell$ coordinates are possibly non-zero and the rest are zero. 
The stabilizer of such a subset is precisely given by $\mathrm{S}(\mathrm{O}(\ell-1)\times \mathrm{O}(n-\ell+1))= K_0\cap H$, block diagonally embedded in $K_0$. 
In particular, $k'k\in K_0\cap H$, and since $K_\ell\subset K\cap H$ it follows that $k\in K_\ell(K_0\cap H)\subseteq K\cap H$, as claimed.

Finally, we show that $b\in (N\cap W_\ell)\subset H$ from which it will follow that $N_G(W_\ell)=H$ as claimed.
Assume $b\notin (N\cap W_\ell)$.
Then using \eqref{eqn:horospherical}, $b$ corresponds to a vector $\vec{v}=(v_1,\dots,v_n)\in\R^n$ such that some $v_j\neq 0$ for $j\ge \ell$.
Letting $\infty^-=(1,-1,0,\dots,0)\in\partial \bH^\ell_{\mathrm{std}}$, then under the identification of $\partial \bH^{n+1}$ given in \eqref{eqn:hyperbolicspaceboundary}, one computes that 
$$b\cdot \infty^-=\left(1,\frac{|\vec{v}|^2-1}{|\vec{v}|^2+1},\dots,\frac{2v_i}{|\vec{v}|^2+1},\dots\right).$$
In particular, $x_{j+1}=2v_j/(|\vec{v}|^2+1)\neq 0$ and, since $j+1>\ell$, the point $b\cdot\infty^-\notin\partial \bH^{\ell}_{\mathrm{std}}$.
Hence $b\notin N_G(W_\ell)$, as required.
\end{proof}

\begin{corollary}\label{cor:Wnormalizerdecomposition}
For any $g\in N_G(W_\ell)$, we may write $g=wk$ for some $w\in W_\ell$ and some $k\in K_0$.
Moreover, either $k\in C_G(W_\ell)$ or $k=k_\ell k_0$ where 
\begin{equation}\label{eqn:k0matrix}
k_\ell=\begin{bmatrix}
\mathrm{Id}_{\ell\times \ell}&0_{\ell\times 1}&0_{\ell\times (n-\ell+1)}\\
0_{1\times \ell}&-1&0_{1\times (n-\ell+1)}\\
0_{(n-\ell+1)\times \ell}&0_{(n-\ell+1)\times 1}&\mathrm{Id}_{(n-\ell+1)\times (n-\ell+1)}
\end{bmatrix},
\end{equation}
and $k_0\in\mathrm{O}(n-\ell+1)\subset K_0$ is block embedded as in Lemma \ref{lem:Wnormalizer} and such that $\det(k_0)=-1$.
\end{corollary}
\begin{proof}
Note that $\SO_0(1, \ell)$ is an index two subgroup of $\mathrm{O}_0(1, \ell)$ and one representative of the non-trivial coset is given by $\SO_0(\ell,1)k_\ell$ with $k_\ell$ as in \eqref{eqn:k0matrix}.
By Lemma \ref{lem:Wnormalizer}, if $g\in N_G(W_\ell)=\mathrm{S}(\mathrm{O}_0(1,\ell)\times \mathrm{O}(n-\ell+1))$, then $g=wk_0$ for some $w\in \mathrm{O}_0(1,\ell)$, $k_0\in  \mathrm{O}(n-\ell+1)$ block embedded as above.
If $\det(w)=1$, then $\det(k_0)=1$ and consequently $w\in W_\ell$ and $k=k_0\in\SO(n-\ell+1)=C_G(W_\ell)$. 
If $\det(w)=-1$, then $\det(k_0)=-1$ and $w=w'k_\ell$ for some $w'\in W_\ell$.
In particular, $g=w'k_\ell k_0$ and we conclude the statement of the lemma.
\end{proof}

The following is now immediate from Corollary \ref{cor:Wnormalizerdecomposition}.

\begin{corollary}\label{cor:WN}
Fix any $\ell\ge 2$, then $W_\ell K_0= N_G(W_\ell)K_0$ and $W_\ell K= N_G(W_\ell)K$.
\end{corollary}


\subsection{Classifying $U^{\pm}$ orbit closures in $F^*M$}\label{subsection:orbitclosures}


Our next goal is to use Ratner's theorems to classify orbit closures of $U^\pm$ in the frame bundle $\Gamma\backslash G$ and thereby the coframe bundle after identification using the Riemannian metric. 
Ratner's Orbit Closure Theorem \cite[Theorem A, Corollary A]{RatnerTop} is the following important theorem, where we again abusively also use the map $\pi^F_\Gamma$ from \S \ref{subsection:hyperbolic_manifolds} to denote the map of frame bundles $G\to \Gamma\backslash G$.

\begin{theorem}[Ratner]\label{thm:Ratner}
Suppose that $D\subseteq G$ is a closed, connected subgroup generated by unipotent elements, $g_0\in G$, and $x_0=\pi^F_\Gamma(g_0)$. 
Then there exists a closed, connected subgroup $L\subseteq G$ such that $D\subseteq L$ and $\overline{x_0 D}=x_0L$ in $\Gamma\backslash G$, $L$ acts ergodically on $x_0L$, and $g_0Lg_0^{-1}\cap \Gamma$ is a lattice in $g_0Lg_0^{-1}$.
\end{theorem}

Theorem \ref{thm:Ratner} therefore enables us to classify orbit closures both when $D=U^\pm$ and when $D=W_{\ell}$, which is the content of the next two lemmas.
In what follows, given a subgroup $H\subseteq G$, we use $H^\dagger$ to denote the subgroup of $H$ generated by its unipotent elements\footnote{In the homogeneous dynamics literature, the notation more commonly used is $H^+$, however, we avoid this notation to avoid conflict with $U^\pm$.}.
This is a closed, normal subgroup of $H$.

\begin{lemma}\label{lem:Uorbitclosure}
Suppose that $U=U^\pm$ and let $x_0=\pi_\Gamma^F(g_0)$ for some $g_0\in G$.
Then $\overline{x_0 U}=x_0L$ for some closed, connected, reductive subgroup $L$ such that $U\subset L$.
Moreover, there exists some $k\in K_U$, $b\in N^\pm$, and $\ell\ge 2$ for which
$$W_{\ell}\subseteq bkLk^{-1}b^{-1}\subseteq N_G(W_{\ell}),$$
where $W_{\ell}$ is a standard subgroup of $G$.
\end{lemma}

\begin{proof}
For the first statement, assume that $U$ is any one-parameter unipotent subgroup of $G$.
The existence of $L$ as described is a combination of Ratner's theorem and \cite[Proposition 3.1]{Shah}.
Indeed, in the latter, Shah shows that such an $L$ must either be reductive with compact center or unipotent.
As $\Gamma$ is cocompact, it cannot contain any non-trivial unipotent elements \cite[Lemma 1]{KazMarg} and therefore, if $L$ were unipotent, $g_0Lg_0^{-1}\cap \Gamma$ would be trivial, contradicting that it is a lattice in $g_0Lg_0^{-1}$.
Hence $L$ must be reductive with compact center.

The condition that $U\subset L$ further implies that $L$ is a real rank $1$ subgroup of $G$, that is, contains a conjugate of the torus $A$.
Therefore there exists $\ell\ge 2$ for which $L^\dagger$ is isomorphic to $W_{\ell}$.
As $G$ acts transitively by conjugation on its subgroups isomorphic to $W_{\ell}$, there exists some $g\in G$ for which $gW_{\ell}g^{-1}= L^\dagger\subseteq L$.
Consequently $W_{\ell}\subseteq g^{-1}Lg$ and, as $L^\dagger$ is a normal subgroup of $L$, it follows that $W_{\ell}\subseteq g^{-1}Lg\subseteq N_G(W_{\ell})$.

We now assume that $U=U^\pm$.
The condition that $U\subset L$ implies that $g^{-1}Ug\subset g^{-1}L^\dagger g=W_{\ell}$.
As $W_{\ell}$ acts transitively by conjugation on its one-parameter unipotent subgroups, there exists some $w\in W_{\ell}$ for which $w^{-1}g^{-1}Ugw=U$.
In particular, $w^{-1}g^{-1}\in N_G(U^\pm)=N^\pm AK_U\subset P^\pm$ by Lemma \ref{lem:Unormalizer}.

Since $A$ normalizes $N^\pm$, we also have $N_G(U^\pm)=A N^\pm K_U$ and therefore we may write $w^{-1}g^{-1}=abk$ for some  $a\in A$, $b\in N^\pm$, and $k\in K_U$.
As conjugation by $w$ and $a$ preserves $W_\ell$, we get the following chain of inclusions
$$a^{-1}w^{-1}W_{\ell}wa=W_{\ell}\subseteq a^{-1}w^{-1}g^{-1}L gwa\subseteq a^{-1}w^{-1}N_G(W_{\ell})wa=N_G(W_{\ell}),$$
and therefore it follows that
$$W_\ell\subseteq a^{-1}w^{-1}g^{-1}Lgwa=bkL(bk)^{-1}\subseteq N_G(W_\ell).$$
This is the desired result.
\end{proof}

For orbit closures under actions of standard subgroups, we have the following lemma, which is similar to the proof of \cite[Lemma 3.2]{BFMS}.
However, we require a slightly more refined version, so we recall the argument for completeness.

\begin{lemma}\label{lem:Worbitclosure}
Fix $\ell\ge 2$, let $W_{\ell}$ be a standard subgroup of $G$, and let $x_0=\pi_\Gamma^F(g_0)$ for some $g_0\in G$.
Then $\overline{x_0 W_{\ell}}=x_0H$ for some closed, connected, reductive subgroup $H$ such that $W_\ell \subseteq H$.
In particular, there exists some $k \in K_0$ and some $\ell'\ge \ell$ such that
$$W_{\ell'}\subseteq kHk^{-1}\subseteq N_G(W_{\ell'}).$$
Moreover, $k$ has the form of Corollary \ref{cor:Wnormalizerdecomposition}, that is, either $k\in C_G(W_\ell)\cong \SO(n-\ell+1)$ or $k=k_\ell k_0$ for $k_\ell$ as in \eqref{eqn:K0defn} and some $k_0\in \mathrm{O}(n-\ell+1)$.
\end{lemma}
\begin{proof}
The beginning of the proof is similar to that of Lemma \ref{lem:Uorbitclosure}.
Indeed, a combination of Ratner's theorem and \cite[Proposition 3.1]{Shah} shows that there exists $H$ which is either reductive with compact center or unipotent. 
As $W_{\ell}\subseteq H$, $H$ cannot be unipotent hence it is reductive with compact center.
Therefore $H^\dagger$ is isomorphic to $W_{\ell'}$ for some $\ell'\ge \ell$ and consequently there exists $g\in G$ for which $W_{\ell'}\subseteq gH g^{-1}\subseteq N_G(W_{\ell'})$.

As $W_{\ell}\subseteq W_{\ell'}$, we have $gW_{\ell} g^{-1}\subseteq gH^\dagger g^{-1}= W_{\ell'}$.
Since $W_{\ell'}$ acts transitively by conjugation on its subgroups isomorphic to $W_{\ell}$, it follows that there exists $w\in W_{\ell'}$ such that $wgW_{\ell} g^{-1}w^{-1}=W_{\ell}$.
In particular, $wg\in N_G(W_{\ell})$ and therefore $wg=\hat{w}k$ for some $\hat{w}\in W_{\ell}$ and some $k\in K_0\cap N_G(W_\ell)$ which is furnished by Corollary \ref{cor:Wnormalizerdecomposition}.
Similar to Lemma \ref{lem:Uorbitclosure}, as $\hat{w}^{-1}W_{\ell'}\hat{w}=W_{\ell'}$ and $\hat{w}^{-1}N_G(W_{\ell'})\hat{w}=N_G(W_{\ell'})$, we have that
$$\hat{w}^{-1}w W_{\ell'}w^{-1}\hat{w}=W_{\ell'}\subseteq \hat{w}^{-1}wgHg^{-1}w^{-1}\hat{w}\subseteq \hat{w}^{-1}w N_G(W_{\ell'})w^{-1}\hat{w}=N_G(W_{\ell'}).$$
Therefore we conclude that
$$W_{\ell'}\subseteq kHk^{-1}\subseteq N_G(W_{\ell'}),$$
as desired.
\end{proof}

Our goal in \S \ref{subsection:basepoints} and \S\ref{subsection:nowheredense} will be to prove the following proposition, which shows that $AU^\pm$-orbit closures are the same as $W_2$-orbit closures.

\begin{proposition}\label{prop:AUorbits}
Fix any $g_0\in G$ and let $x_0=\pi^F_\Gamma(g_0)$, then $\overline{x_0AU^\pm}=\overline{x_0W_2}$ in $\Gamma\backslash G$.
\end{proposition}

Note that $AU^\pm$ is not generated by unipotents so we cannot apply Ratner's theorem to classify its orbits as in Lemmas \ref{lem:Uorbitclosure} and \ref{lem:Worbitclosure}.
Therefore the proof of Proposition \ref{prop:AUorbits} will require additional understanding of the failure of equidistribution of $U^\pm$-orbits inside $W_2$-orbits, which is classified in the work of Ratner \cites{RatnerMeasure,RatnerTop} and Dani--Margulis \cite{DM}.
Many of the arguments below are inspired by \cite{LeeOh}, which develops a topological approach to the study of orbit closures in $\SO_0(1,n+1)$; there are also some similarities with arguments found in \cite{BFMS}.
For completeness, we opt to give a complete account of the relevant details in what follows.

Momentarily assuming the proof of Proposition \ref{prop:AUorbits}, we first show how to deduce Theorem \ref{thm:orbit_closure} from Proposition \ref{prop:AUorbits}.

\begin{proof}[Proof of Theorem \ref{thm:orbit_closure} assuming Proposition \ref{prop:AUorbits}]
We argue the corresponding result for the frame bundle instead of the coframe bundle. 
Recall that, using the Riemannian metric, there is an isomorphism $F^*M$ to $FM$ which identifies $S^*M$ with $SM$ and is equivariant with respect to the flows discussed in \S \ref{subsection:hyperbolic_manifolds}.
Therefore, it suffices for our purposes to work in the frame bundle.
Moreover, as in \S \ref{subsection:TGsubmanifolds}, we equivariantly identify $FM$ with $\Gamma\backslash G$ and $SM$ with $\Gamma\backslash G/K_0$, hence it suffices to prove the corresponding result for orbit closures on these spaces under the usual group actions.

Fix $q\in F^*M$, which we identify with a point in $\Gamma\backslash G$ abusively also denoted $q$. We write $q=\Gamma\backslash\Gamma q_0$ for some $q_0\in G$.
After the above reductions, Lemma \ref{lem:Worbitclosure} and Proposition \ref{prop:AUorbits} show that
$$ \overline{\{\phi_t( e^{sU_1^\pm}(q)) : t, s \in \R\}}=\overline{qAU^\pm}=\overline{qW_2}=qH,$$
for some closed, connected subgroup $H$ such that
$$W_\ell\subseteq kH k^{-1}\subseteq N_G(W_\ell),$$
with $\ell\ge 2$ and $k\in K_0$.
Consequently if $\pi:FM\to M$ is the right quotient by $K$, then 
$$\Sigma_q \coloneqq \pi(qH)=\Gamma\backslash\Gamma q_0 H K/K=\Gamma\backslash\Gamma q_0k^{-1} W_\ell  K/K,$$
where the final equality follows from Corollary \ref{cor:WN} and the fact that $k\in  K_0\subset K$.
As $qH$ is closed in $\Gamma\backslash G$ and $\pi$ is proper, $\Sigma_q$ is a totally geodesic submanifold of $M$.
Note that this subset is indeed totally geodesic, as it lifts to the isometrically embedded $\bH^\ell\subseteq \bH^{n+1}$ corresponding to $q_0k^{-1} W_\ell  K/K\subseteq G/K$.

The similar computation shows that
$$\pi_S \overline{\{\phi_t( e^{sU_1^\pm}(q)) : t, s \in \R\}} = \pi_S(qH)=\Gamma\backslash\Gamma q_0 H K_0/K_0=\Gamma\backslash\Gamma q_0k^{-1} W_\ell  K_0/K_0,$$
and by the discussion in \S \ref{subsection:TGsubmanifolds}, this set is precisely $S\Sigma_q$.
The result then follows.
\end{proof}

\begin{remark}
The proof above makes it transparent why one needs to study $AU^\pm$-orbit closures as opposed to simply $U^\pm$-orbit closures.
The essential difference is the presence of the element $b\in N^\pm$ in Lemma \ref{lem:Uorbitclosure}. Upon insertion of this conjugating element into the calculations above, one no longer concludes that the corresponding subset lifts to a geodesic plane in $G/K$, merely that it lifts to a subset with $W_\ell$ conjugated by $b$.
This is sometimes referred to as being parallel to a geodesic plane in the literature.
The failure of $N^\pm$ to normalize $W_\ell$ makes this an essential problem and one cannot, in general, conclude that such subsets are geodesic planes. 
\end{remark}


\subsection{Finding good basepoints for orbits}\label{subsection:basepoints}


From the discussion in the previous subsection, we are reduced to showing that $\overline{x_0AU^\pm }=\overline{x_0 W_2}$ for every $x_0\in \Gamma\backslash G$.
Our goal in this section is to show that there always exists some point (in fact, many points) $y_0\in \overline{x_0AU^\pm}$ for which $\overline{y_0U^\pm}=\overline{x_0W_2}$ and hence to conclude that $\overline{y_0U^\pm}=\overline{x_0AU^\pm}=\overline{x_0 W_2}$.
For this, we must first recount important work of Dani--Margulis on equidistribution of unipotent flows.
We do this in a more general context than the present setting and then specialize to the current setting after giving the requisite background.

Let $U$, $W$ be subgroups of $G$ generated by unipotents such that $U\subset W$ is a proper subgroup.
Fix a point $\pi_\Gamma^F(g_0)=x_0\in\Gamma\backslash G$ and let $\overline{x_0 W}=x_0 H$ be the closure afforded by Ratner's theorem.
We call a point $y\in x_0H$ a \emph{singular point} if $\overline{y U}$ is proper in $x_0 H$.
The set of all such points is given by
$$\mathcal{S}_{x_0} \coloneqq \left\{y\in x_0H : \overline{y U}\subsetneq x_0 H\right\}\subset \Gamma\backslash G,$$
which we call the \emph{singular set}.

Note that any $y\in x_0H$ can be written as $y=\pi_\Gamma^F(g_0h_0)$ for some $h_0\in H$.
Moreover, again by Ratner's theorem, $\overline{y U}=yZ=\Gamma\backslash\Gamma g_0h_0Z$ for some closed, connected $Z\subseteq H$ such that $U\subseteq Z$.
If $yZ\subsetneq x_0H=yH$, it also follows that $Z^\dagger\subsetneq  H^\dagger$ is a proper subgroup (see for instance \cite[Lemmas 3.10, 3.11]{ADM24}) and in particular $Z\subsetneq H$.
As $yZ$ is closed, so is $yZ(g_0h_0)^{-1}=\Gamma\backslash\Gamma g_0h_0Z(g_0h_0)^{-1}$, and therefore the existence of such a $y$ is equivalent to the existence of a closed, connected subgroup $J=g_0h_0Z(g_0h_0)^{-1}$ for which $\Gamma\backslash \Gamma  J$ is closed in $\Gamma\backslash G$.
Note that by definition, $U\subseteq Z$ and therefore $g_0h_0U(g_0h_0)^{-1}\subseteq J$.
We collect elements of $G$ having this latter property in the sets 
$$X(J,U) \coloneqq \left\{g\in G : g U g^{-1}\subseteq J\right\}\subset G,$$
which in the literature are frequently referred to as \emph{tubes}.

We are only interested in tubes, $X(J,U)$, that correspond to orbit closures of elements in the singular set.
As shown above, these correspond precisely to the subgroups $J$ satisfying the following four conditions:
\begin{enumerate}
\item $J\subset  G$ is a proper, closed, connected subgroup,
\item $J$ contains a conjugate of $U$,
\item $J\cap \Gamma$ is a lattice in $J$ whose Zariski closure is precisely $J$,
\item $g_0^{-1} J g_0\subset  H$ is a proper subgroup of $H$.
\end{enumerate}
Let $\mathcal{H}_{(g_0,H)}$ denote the collection of all subgroups of $G$ satisfying conditions (1)--(4) and define the following subset of $G$
$$\mathcal{S}_{g_0H}\coloneqq g_0H\cap \left(\bigcup_{J\in\mathcal{H}_{(g_0,H)}}X(J,U)\right).$$
The collection $\mathcal{H}_{(g_0,H)}$ is a countable collection of subgroups of $G$ by work of Ratner \cite[Theorem 2]{RatnerMeasure} or, alternatively, Dani--Margulis \cite[Proposition 2.3]{DM}.
It follows from the discussion above that
\begin{equation}\label{eqn:singsetrelation}
\mathcal{S}_{x_0}=\pi_\Gamma^F\left(\mathcal{S}_{g_0H}\right).
\end{equation}
In particular, for a given choice of the triple $(x_0,U,W)$, we wish to show the existence of points $y$ in the complement of this set in $x_0 H$.

We now specialize to our setting.
For the remainder of the subsection, we focus on the case of the subgroups $U=U^+$ and $W=W_2$.
It will be transparent that the case where $U=U^-$ is entirely similar, with only cosmetic differences to the proofs.
In this setting, in the definition of $\mathcal{H}_{(g_0,H)}$, Condition (1) specializes to
\begin{enumerate}
\item[(1')] $J\subset  G$ is a proper, closed, connected reductive subgroup,
\end{enumerate}
which in particular implies that $J^\dagger$ is conjugate to some standard subgroup $W_{\ell}$.
As $g_0^{-1}J^\dagger g_0$ is a proper subgroup of $H^\dagger\cong W_{\ell'}$, it also follows that $\ell<\ell'$.

In the remainder of the subsection, we complete the proof of Proposition \ref{prop:AUorbits} modulo the following technical lemma, which says that the right $W_2$-saturation of a tube is nowhere dense in the orbit $x_0H$.

\begin{lemma}\label{lem:nowheredensetube}
Suppose that $x_0=\pi_\Gamma^F(g_0)$ and $\overline{x_0W_2}=x_0H$ for $H$ as in Lemma \ref{lem:Worbitclosure}.
Then for any $J\in \mathcal{H}_{(g_0,H)}$, the set $\left( g_0H\cap X(J,U)\right)W_2$ is nowhere dense in $g_0H$.
\end{lemma}

As an immediate consequence, we deduce the nowhere density of the right $W_2$-saturation of the singular set (see also \cite[Lemma 3.14]{ADM24}).
Strictly speaking, this is stronger than what we need to furnish the point $y$ above.

\begin{corollary}\label{cor:nowheredensesingset}
Suppose that $x_0=\pi_\Gamma^F(g_0)$ and $\overline{x_0W_2}=x_0H$ for $H$ as in Lemma \ref{lem:Worbitclosure}.
Then $\mathcal{S}_{x_0}W_2$ is nowhere dense in $x_0H$.
\end{corollary}
\begin{proof}
Assuming Lemma \ref{lem:nowheredensetube}, this is a simple application of the Baire category theorem combining \eqref{eqn:singsetrelation} with the facts that $\Gamma$, $\mathcal{H}_{(g_0,H)}$ are countable and that the image of a function distributes on unions. 
\end{proof}

We now conclude Proposition \ref{prop:AUorbits}, momentarily assuming Lemma \ref{lem:nowheredensetube}.

\begin{proof}[Proof of Proposition \ref{prop:AUorbits} assuming Lemma \ref{lem:nowheredensetube}]
Fix any $g_0\in G$ and let $x_0=\pi^F_\Gamma(g_0)$.
By Lemmas \ref{lem:Uorbitclosure} and \ref{lem:Worbitclosure}, we write $\overline{x_0 U}=x_0L$ and $\overline{x_0 W_2}=x_0H$.
Then 
$$\overline{x_0 U}=x_0L\subseteq \overline{x_0AU}\subseteq \overline{x_0 W_2}=x_0H.$$
The $KAN$-decomposition of $W_2$ is given by $W_2=K_2AU=AUK_2$, where $K_2\cong \SO(2)$ as in the beginning of \S \ref{subsection:TGsubmanifolds}, and the equalities of decompositions follow from taking inverse and the fact that $A$ normalizes $U$.
As $K_2$ is compact it follows that 
$$x_0H=\overline{x_0W_2}=\overline{x_0 AUK_2}=\overline{x_0 AU}K_2,$$
and consequently any $y\in x_0H$ can be written as $y=y_0 k$ for some $y_0\in\overline{x_0AU}$, $k\in K_2$.
From Corollary \ref{cor:nowheredensesingset}, there exists some $y\in x_0H$ such that $y\notin S_{x_0}W_2$.

For such $y$, it follows by definition of the singular set that
$$\overline{yU_0}=yH=x_0H.$$
Decomposing this point as $y=y_0k$ as above, Corollary \ref{cor:nowheredensesingset} shows that $y_0\notin S_{x_0}W_2$ as well, since the entire right $W_2$-orbit of $y$ has this property.
Therefore,
$$x_0H=\overline{y_0 U}\subseteq\overline{x_0 AU}\subseteq \overline{x_0 W_2}=x_0H,$$
and hence $\overline{x_0AU}=\overline{x_0 W_2}=x_0H$ as required.
\end{proof}


\subsection{The proof of Lemma \ref{lem:nowheredensetube}}\label{subsection:nowheredense}


To prove Lemma \ref{lem:nowheredensetube}, we proceed in three steps. 
First, in Lemma \ref{lem:standardtube} we give a computation of the tube $X(J,U)$ in the simplest possible case when $J=W_{\ell}$, which we refer to as the \emph{standard tube}. 
Next, we give a geometric argument in Lemma \ref{lem:tubegeometry} that will allow us to conclude nowhere density in the simplest possible case, that is, when $g_0$ is the identity and the tube in consideration is the standard one.
Finally, we give the proof of Proposition \ref{lem:nowheredensetube}, where one uses general properties of tubes and orbits to reduce nowhere density to this simple setting.

Before embarking upon this, we make a few observations about the behavior of tubes under various group operations.
These properties will be essential in the final step, when we reduce from the general case in the proof of Lemma \ref{lem:nowheredensetube}.
Two properties of tubes are immediate from the definition, namely that
\begin{align}
X(J,U)&=X(J^\dagger,U),\label{eqn:tubeproperty1}\\ 
gX(J,U)&=X(gJg^{-1},U),~~\forall g\in G,\label{eqn:tubeproperty2}
\end{align}
where in the former we are using that $U$ is generated by unipotents.
There is no analog of the latter property with the element on the righthand side for general $g\in G$, however, in the case that $g\in N_G(U)$ one moreover has that 
\begin{equation}\label{eqn:tubeproperty3}
X(J,U)g=X(J,U),~~\forall g\in N_G(U),
\end{equation}
which similarly follows from the definition.
We now compute the standard tube.

\begin{lemma}\label{lem:standardtube}
Let $\ell\ge 2$, $U=U^+$, and $N=N^+$, then
$$X(W_{\ell},U)=W_{\ell}N_G(U)=W_{\ell}K_UN,$$
with $K_U$ as in Lemma~\ref{lem:Unormalizer}.
\end{lemma}
\begin{proof}
We first prove the first equality.
It is clear that $W_{\ell}N_G(U)\subseteq  X(W_{\ell},U)$ by definition.
For the reverse inclusion, let $g\in X(W_{\ell},U)$ so that $gUg^{-1}\subset  W_{\ell}$. 
As $W_{\ell}$ acts transitively by conjugation on its one-parameter unipotent subgroups, there exists some $w\in W_{\ell}$ for which $wgU(wg)^{-1}=U$.
Consequently $wg\in N_G(U)$ and hence $g\in W_\ell N_G(U)$ as required.

For the second equality, using Lemma \ref{lem:Unormalizer}, we note that $N_G(U)=NAK_U=AK_UN$ by taking inverses and using the fact that $K_U\subset K_0$ and $K_0$ centralizes $A$. 
As $A\subset  W_{\ell}$, the result follows.
\end{proof}

\begin{lemma}\label{lem:tubegeometry}
For any $\ell'>\ell\ge 2$, the subset $(W_{\ell'}\cap X(W_\ell,U))W_2$ of $W_{\ell'}$ is nowhere dense.
\end{lemma}

\begin{proof}
Throughout we let $N=N^+$.
Taking inverses, we equivalently show that $W_2(W_{\ell'}\cap X(W_\ell,U))=W_2(W_{\ell'}\cap N K_UW_{\ell})$ is nowhere dense in $W_{\ell'}$, as the latter is compatible with our identification of $G/K$ with $\bH^{n+1}$.
We first show this on a particular quotient of $W_{\ell'}$ and then lift the result to $W_{\ell'}$ itself at the end of the proof.

Let $\mathcal{P}_{\ell'}(\ell)$ denote the set of all isometric embeddings of $\bH^\ell$ into $\bH_{\mathrm{std}}^{\ell'}$, where throughout the proof we use the notation defined at the beginning of \S \ref{subsection:TGsubmanifolds}. 
Using Lemma \ref{lem:Wnormalizer}, one verifies that $\mathrm{stab}_{W_{\ell'}}(\bH^\ell_{\mathrm{std}})=N_{W_{\ell'}}(W_{\ell})$.
As $W_{\ell'}$ acts transitively on geodesic $\ell$-planes in $\bH_{\mathrm{std}}^{\ell'}$, $\mathcal{P}_{\ell'}(\ell)$ has the structure of a homogeneous space and is identified with $W_{\ell'}/N_{W_{\ell'}}(W_{\ell})$ by mapping $wN_{W_{\ell'}}(W_{\ell})$ to $w\bH^\ell_{\mathrm{std}}$.

We claim that the image of $W_2(W_{\ell'}\cap N K_UW_{\ell})$ is nowhere dense in $ W_{\ell'}/N_{W_{\ell'}}(W_{\ell})\simeq \mathcal{P}_{\ell'}(\ell)$.
As $W_\ell\subset N_{W_{\ell'}}(W_\ell)$, it suffices to show that $W_2(W_{\ell'}\cap NK_U)$ is nowhere dense in $W_{\ell'}/N_{W_{\ell'}}(W_{\ell})$.
Using Equations \eqref{eqn:hyperbolicspaceboundary}, \eqref{eqn:Kequation}, and \eqref{eqn:horospherical}, one computes that $NK_U$ stabilizes the point $\infty=(1,1,0,\cdots,0)\in \partial \bH^\ell_{\mathrm{std}}\subset \partial \bH^{\ell'}_{\mathrm{std}}\subset \partial  \bH^{n+1}$.
In particular, the $(W_{\ell'}\cap NK_U)$-orbit of $\bH^\ell_{\mathrm{std}}$ is contained in
$$\mathcal{P}_\infty \coloneqq \{\bH^\ell\subset\bH_\mathrm{std}^{\ell'}: \infty\in\partial \bH^\ell\}\subset \mathcal{P}_{\ell'}(\ell).$$ 
Moreover, the $KAN$-decomposition of $W_2$ gives that $W_2= K_2 A U$ and a straightforward computation using \eqref{eqn:Aequation} shows that $AU\cdot\infty=\infty$.
Consequently, $W_2\cdot \mathcal{P}_\infty=K_2\cdot \mathcal{P}_\infty$, so we are reduced to computing this orbit.

A final computation shows that if $v_\theta=(1,\cos(\theta),\sin(\theta),0,\dots,0)$ then
$$K_2\cdot \infty=\{v_\theta: \theta\in[0,2\pi)\},$$
and therefore
$$W_2(W_{\ell'}\cap NK_U)\cdot \bH^\ell_{\mathrm{std}}\subset \mathcal{P}_\theta\coloneqq K_2\cdot \mathcal{P}_\infty=\{ \bH^\ell\subset\bH_\mathrm{std}^{\ell'}: v_\theta\in\partial \bH^\ell~\text{for some }\theta\in[0,2\pi)\}.$$ 
However, the set $\mathcal{P}_\theta$ is nowhere dense in $\mathcal{P}_{\ell'}(\ell)$ and therefore so is the image of $W_2(W_{\ell'}\cap N K_UW_{\ell})$ in $W_{\ell'}/N_{W_{\ell'}}(W_\ell)$.
The result then follows since quotient maps are open maps and hence taking pre-image preserves nowhere density.
\end{proof}

Combining the ingredients above, we are now in a position to prove Lemma \ref{lem:nowheredensetube} and thus complete the proof of Theorem \ref{thm:orbit_closure}.

\begin{proof}[Proof of Lemma \ref{lem:nowheredensetube}]
The entire proof is a sequence of reductions from the general computation stated in Lemma \ref{lem:nowheredensetube} to that of the one in Lemma \ref{lem:tubegeometry}.
Suppose that $x_0=\pi_\Gamma^F(g_0)$ and $\overline{x_0W_2}=x_0H$ for $H$ as in Lemma \ref{lem:Worbitclosure}, so that there exists $k\in K_0$ for which
$$W_{\ell'}\subseteq kHk^{-1}\subseteq N_G(W_{\ell'}),$$
for some $\ell'\ge 2$.
In the sequel, we will use the refined description of $k$ from Corollary \ref{cor:Wnormalizerdecomposition}.

We may immediately reduce to the case that $\ell'>2$ since otherwise $\mathcal{H}_{(g_0,H)}=\emptyset$.
Indeed, in that setting $H^\dagger\cong W_2$ and therefore a combination of properties (1') and (4) of $\mathcal{H}_{(g_0,H)}$ show that there are no subgroups $J\subset G$ for which $J^\dagger$ is both isomorphic to a standard subgroup and for which $g_0^{-1} J^\dagger g_0\subseteq H^\dagger$ is a proper subgroup. The reduction from $g_0^{-1} J g_0\subset H$ being proper to the statement on subgroups generated by unipotents follows, for instance, from \cite[Lemmas 3.10, 3.11]{ADM24}.

Now assume $\ell'>2$ and $J\in \mathcal{H}_{(g_0,H)}$.
Then, by left translation by $g_0^{-1}$, the condition that $\left( g_0H\cap X(J,U)\right)W_2$ is nowhere dense in $g_0H$ is equivalent to the condition that
$$\left( H\cap g_0^{-1}X(J,U)\right)W_2=\left( H\cap X(g_0^{-1}Jg_0,U)\right)W_2,$$
is nowhere dense in $H$.
This follows from \eqref{eqn:tubeproperty2}.

Note that there is a bijection from $\mathcal{H}_{(g_0,H)}$ to $\mathcal{H}_{(e,H)}$ by sending $J$ to $g_0^{-1}Jg_0$.
By applying \eqref{eqn:tubeproperty1}, we are therefore reduced to considering subsets of $H$ of the form
$$\left( H\cap X(J,U)\right)W_2=\left( H\cap X(J^\dagger,U)\right)W_2,$$
for any $J\in\mathcal{H}_{(e,H)}$.

We now reduce to the case where $W_{\ell'} \subset H \subset N_G(W_{\ell'})$.
By the explicit description of $k$ in Corollary \ref{cor:Wnormalizerdecomposition}, it follows that both $k\in K_U$ and also that $k$ normalizes $W_2$. Therefore
\begin{align*}
k\left( H\cap X(J^\dagger,U)\right)W_2k^{-1}&=\left( kHk^{-1}\cap kX(J^\dagger,U)k^{-1}\right)W_2\\
&=\left( kHk^{-1}\cap X(kJ^\dagger k^{-1},U)\right)W_2
\end{align*}
where we use Equations \eqref{eqn:tubeproperty2} and \eqref{eqn:tubeproperty3}.
As conjugation preserves nowhere density, we are therefore reduced to the case that $k=e$, since the previous equation implies the case for general $k$.
That is to say, we are reduced to showing that $\left( H\cap X(J^\dagger,U)\right)W_2$ is nowhere dense in $H$, where $W_{\ell'}\subseteq H\subseteq N_G(W_{\ell'})$ and $J\in \mathcal{H}_{(e,H)}$.

We next reduce to the case that $J^\dagger=W_{\ell}$ for some $2\le \ell< \ell'$. 
Indeed, if $J\in \mathcal{H}_{(e,H)}$ then from the discussion in \S \ref{subsection:basepoints}, $J^\dagger\cong W_{\ell}$ for some $2\le \ell< \ell'$ and, moreover, as all such subgroups of $W_{\ell'}$ are $W_{\ell'}$-conjugate, there exists $w\in W_{\ell'}\subseteq H$ for which $wJ^\dagger w^{-1}=W_{\ell}$.
As both the group $H$ and the property of nowhere density are invariant under left translation by $w$, nowhere density of the previous equation is equivalent to nowhere density of
$$w\left( H\cap X(J^\dagger,U)\right)W_2=\left( wH\cap wX(J^\dagger,U)\right)W_2=\left( H\cap X(W_{\ell},U)\right)W_2,$$
in $H$.

Finally, we reduce to the case that $H=W_{\ell'}$.
Since $W_{\ell'}\subseteq H\subseteq N_G(W_{\ell'})$, it follows from Corollary \ref{cor:Wnormalizerdecomposition} that $H=W_{\ell'} C$ for some $C\subset K_0$ for which $C\subseteq K_U$ and $C$ normalizes $W_2$.
In fact, this corollary shows that $C$ is a subgroup of $\langle  k_{\ell'}, \mathrm{O}(n-\ell'+1)\rangle$ and, as $\ell'>2$, this shows that $C$ in fact centralizes $W_2$.
By a final application of \eqref{eqn:tubeproperty3}, we have that
$$\left( H\cap X(W_{\ell},U)\right)W_2=\left( W_{\ell'}C\cap X(W_{\ell},U)\right)W_2=\left( W_{\ell'}\cap X(W_{\ell},U)\right)W_2C.$$
In particular, the right $C$-saturation of this set combined with the fact that $C$ and $W_2$ commute shows that nowhere density of the above set in $H$ is equivalent to nowhere density of
$\left( W_{\ell'}\cap X(W_{\ell},U)\right)W_2$ in $W_{\ell'}$. However, this is precisely what is proved by  Lemma~\ref{lem:tubegeometry}.
\end{proof}


\subsection{Geodesic submanifolds in hyperbolic manifolds}\label{section:TGexamples}


In this subsection, we briefly recount what is presently known about totally geodesic submanifolds of hyperbolic manifolds.
Hyperbolic manifolds bifurcate into two distinct types, arithmetic and non-arithmetic.
At present, we have an effectively complete understanding of the behavior of totally geodesic submanifolds in the former class, while the latter class remains more mysterious.
We describe what is known about each of these classes of manifolds below.

Informally, arithmetic manifolds are a classical construction of locally symmetric spaces which arise from integer structures on the isometry group $G$. 
In the hyperbolic setting, it is well-known that there are precisely three general constructions of arithmetic hyperbolic manifolds which are as follows: 
\begin{itemize}
\item Arithmetic manifolds of simplest type, which exist in all dimensions.
\item Arithmetic manifolds of second type, which exist in every odd dimension.
\item Arithmetic manifolds of triality type, which exist only in dimension $7$.
\end{itemize}
Additionally, in each dimension for which a given construction exists, one can produce infinitely many commensurability classes of that construction. 
Recall that two manifolds are \emph{commensurable} if they share a common finite sheeted cover.
This forms an equivalence relation on the set of finite-volume hyperbolic manifolds.

Arithmetic manifolds satisfy a strong submanifold dichotomy -- in every codimension there exists either $0$ or infinitely many totally geodesic submanifolds.
Moreover, for each of the above constructions, one can use the theory of algebraic groups over number fields to show the following regarding their totally geodesic submanifolds:
\begin{itemize}
\item Arithmetic manifolds of simplest type have infinitely many totally geodesic submanifolds of every codimension.
\item Arithmetic manifolds of second type always have infinitely many totally geodesic submanifolds of every even codimension but no totally geodesic submanifolds of codimension $1$.
\item Arithmetic manifolds of triality type have infinitely many totally geodesic submanifolds of dimension $3$.
\end{itemize}
The first two items of this list are straightforward to the reader well-versed in arithmetic manifolds.
For the triality type manifolds, we refer the interested reader to \cite[Theorem 1.5]{BBKS}.

For non-arithmetic manifolds, much less is currently known. 
Unlike in the arithmetic setting, it is known that non-arithmetic manifolds necessarily have only finitely many codimension $1$ totally geodesic submanifolds and more generally only finitely many maximal totally geodesic submanifolds \cite[Theorem 1.1]{BFMS}, where maximal means with respect to inclusion.

In dimensions $2$ and $3$, Teichm\"uller theory and work of Thurston \cite{Thurston}, culminating in Agol's proof of the virtual fibering conjecture \cite{Agol}, give a complete commensurability classification of all constructions of hyperbolic $2$- and $3$-manifolds.
Indeed, in dimension $2$ they are all hyperbolic structures on genus $g\ge 2$ surfaces, and in dimension $3$, all hyperbolic manifolds are mapping tori of pseudo-Anosovs (up to commensurability).
However, in dimensions $\ge 4$, our understanding of constructions of hyperbolic manifolds is still incomplete.
The following are presently the only known families of constructions of higher dimensional hyperbolic manifolds:
\begin{itemize}
\item Non-arithmetic manifolds arising from hyperbolic reflection groups. These are known to exist only in dimensions $\le 30$ when compact \cite{Vinberg} and $\le 997$ when one only assumes finite-volume \cite{Prok}.
\item Non-arithmetic manifolds arising from the gluing construction of Gromov--Piatetski-Shapiro \cite{GPS} and subsequent generalizations, e.g., by Raimbault \cite{Raimbault} and Gelander--Levit \cite{GelLev}. These exist in all dimensions.
\item Non-arithmetic manifolds arising from the gluing construction of Agol \cite{Agol}, generalized by Belolipetsky--Thomson \cite{BelThom} to higher dimensions, and subsequent generalizations. These exist in all dimensions.
\item Non-arithmetic manifolds arising from gluing constructions which hybridize the former two types. These exist in all dimensions.
\end{itemize}
Informally, the latter three gluing constructions are all built by taking an arithmetic manifold of simplest type which contains an embedded totally geodesic hypersurface, cutting the arithmetic manifold open on that hypersurface, and gluing this to another manifold built similarly.
In particular, these constructions always have codimension $1$ totally geodesic submanifolds coming from the boundary on which one is gluing. For different reasons, manifolds corresponding to hyperbolic reflection groups also always have such submanifolds.
As such, the following remains an extremely important open question for higher dimensional hyperbolic manifolds.

\begin{ques}\label{ques:noTGs}
Do there exist non-arithmetic manifolds of dimension $\ge 4$ which have no immersed totally geodesic, codimension $1$ submanifolds?
\end{ques}

We note that in dimension $2$, this question is vacuous as we only consider geodesic submanifolds of dimension $\ge2$. 
In dimension $3$, it is known that there are a plethora of non-arithmetic manifolds with no totally geodesic surfaces (e.g. by using the results of \cite[\S5.3.1]{MRBook} on manifolds in \cite[\S13.5]{MRBook} or on census manifolds in Snappy \cite{Snappy}), so Question \ref{ques:noTGs} is asking about genuinely higher dimensional phenomenon.


 
\section{Fractal uncertainty principle}\label{section:FUP}


A crucial part of the argument for Theorem \ref{thm:support} is the higher-dimensional fractal uncertainty principle, which we generalize to apply to hyperbolic manifolds. We begin by recalling the  following definitions. We denote the ball centered at $x$ of radius $r$ by $B_r(x)$. 
\begin{definition}\label{def:porous_on_balls}
A set $X \subset \R^n$ is \emph{$\nu$-porous on balls} from scales $\alpha_0$ to $\alpha_1$ if for every ball $B \subset \R^n$ of diameter $\alpha_0 < R< \alpha_1$, there is some $x \in B$ such that $B_{\nu R}(x) \cap X = \emptyset$.
\end{definition}

\begin{definition}\label{def:porous_on_lines}
A set $X \subset \R^n$ is \emph{$\nu$-porous on lines} from scales $\alpha_0$ to $\alpha_1$ if for all line segments $\tau \subset \R^n$ of length $\alpha_0 < R< \alpha_1$, there is some $x \in \tau$ such that $B_{\nu R}(x) \cap X = \emptyset$.
\end{definition}

Let $\cF_h$ be the unitary semiclassical Fourier transform defined by 
$$
\cF_h f(\xi) = h^{-\frac{n}{2}} \int_{\R^n} e^{-2 \pi i \lrang{x, \xi}/h} f(x) dx, \quad f \in L^2(\R^n).
$$

We now state the higher-dimensional fractal uncertainty principle. 
\begin{theorem}[~\cite{Co24}*{Theorem 1.1}]\label{thm:fractal_uncertainty}
Set $0 < \nu \leq \frac{1}{3}$. Let
\begin{itemize}
\item $X_- \subset [-1,1]^n$ be $\nu$-porous on balls from scales $h$ to 1;
\item $X_+ \subset [-1,1]^n$ be $\nu$-porous on lines from scales $h$ to 1.
\end{itemize}
Then there exist $\beta, C >0$, depending only on $\nu$ and $n$, such that
\begin{equation*}
\|\1_{X_-} \cF_h \1_{X_+}\|_{L^2 (\R^n) \rightarrow L^2(\R^n)} \leq Ch^{\beta}.
\end{equation*}
\end{theorem}

We spend the rest of this section generalizing Theorem  ~\ref{thm:fractal_uncertainty}. We do so by adapting the work of ~\cite{BD18}*{\S\S 2.2, 4} to higher dimensions.


\subsection{Porosity properties}\label{subsection:porosity_properties}

The first step in generalizing Theorem ~\ref{thm:fractal_uncertainty} is to adapt ~\cite{BD18}*{Lemmas 2.1--4} to higher dimensions to show that line and ball porosity are preserved under certain operations (although the porosity constant $\nu$ may decrease and the scales may change). Notably, ~\cite{BD18} used $\delta$-regularity instead of porosity and thus used different proofs to ours. 

First we show that porosity is preserved under dilations and shifts.
\begin{lemma}\label{lem:affine_ball}
Let $X$ be $\nu$-porous on balls from scales $\alpha_0$ to $\alpha_1$. Fix $\lambda >0$ and $y \in \R^n$. Then $\tilde{X} \coloneqq y + \lambda X$ is $\nu$-porous on balls from scales $\lambda \alpha_0$ to $\lambda \alpha_1$.
\end{lemma}

\begin{proof}
Let $B$ be a ball of diameter $R$, where $\lambda \alpha_0 < R < \lambda \alpha_1$. Set $\tilde{B} = \lambda^{-1}(B-y)$. As $\tilde{B}$ is a ball of diameter $\lambda^{-1} R$,  there exists $\tilde{x} \in \tilde{B}$ such that $B_{\nu \lambda^{-1} R}(\tilde{x}) \cap X = \emptyset$. Thus, $B_{\nu R}(\lambda \tilde{x} +y) \cap \tilde{X} = \emptyset$, where $\lambda \tilde{x} +y \in B$.
\end{proof}

\begin{lemma}\label{lem:affine_line}
Let $X$ be $\nu$-porous on lines from scales $\alpha_0$ to $\alpha_1$. Fix $\lambda >0$ and $y \in \R^n$. Then $\tilde{X} \coloneqq y + \lambda X$ is $\nu$-porous on lines from scales $\lambda \alpha_0$ to $\lambda \alpha_1$.
\end{lemma}

\begin{proof}
Let $\tau$ be a line segment of length $R$, where $\lambda \alpha_0 < R < \lambda \alpha_1$. Set $\tilde{\tau} = \lambda^{-1}(\tau-y)$. As $\tilde{\tau}$ is a line segment of length $\lambda^{-1} R$, there exists $\tilde{x} \in \tilde{\tau}$ such that $B_{\nu \lambda^{-1} R}(\tilde{x}) \cap X = \emptyset$. Thus, $B_{\nu R}(\lambda \tilde{x} +y) \cap \tilde{X} = \emptyset$, where $\lambda \tilde{x} +y \in \tau$.
\end{proof}

We now turn our attention to neighborhoods of porous sets. We use the following notation. 
\begin{notation}\label{notation:neighborhood}
For any set $S \subset \R^n$ and any $\delta>0$, define 
\begin{equation*}\label{eq:neighborhood_def}
S(\delta) \coloneqq S + B_\delta (0).
\end{equation*}
\end{notation}

We quote the following two lemmas on neighborhoods of porous sets.  
\begin{lemma}[~\cite{Ki24}, Lemma 2.17]\label{lem:gen_porous_balls}
Let $\nu \in (0,1 )$, $0 < \alpha_0 \leq \alpha_1$ and $0 < \alpha_2 \leq \frac{\nu}{2} \alpha_1$. Assume that $X \subset \R^n$ is $\nu$-porous on balls from scales $\alpha_0$ to $\alpha_1$. Then the neighborhood $X(\alpha_2)$ is $\frac{\nu}{2}$-porous on balls from scales $\max(\alpha_0, \frac{2}{\nu}\alpha_2)$ to $\alpha_1$.
\end{lemma}

\begin{lemma}[~\cite{Ki24}, Lemma 2.18]\label{lem:gen_porous_lines}
Let $\nu \in (0,1 )$, $0 < \alpha_0 \leq \alpha_1$ and $0 < \alpha_2 \leq \frac{\nu}{2} \alpha_1$. Assume that $X \subset \R^n$ is $\nu$-porous on lines on scales $\alpha_0$ to $\alpha_1$. Then the neighborhood $X(\alpha_2)$ is $\frac{\nu}{2}$-porous on lines from scales $\max(\alpha_0, \frac{2}{\nu}\alpha_2)$ to $\alpha_1$.
\end{lemma}

To generalize Theorem \ref{thm:fractal_uncertainty}. we will use changes of variables. To that end, we show that  ball porosity is preserved under diffeomorphisms. 

\begin{lemma}\label{lem:diffeomorphism_ball}
Let $U \subset \R^n$ be a ball and $\varkappa:U \rightarrow \varkappa(U)$ be a diffeomorphism with 
$$C_1^{-1}|x-y| \leq |\varkappa(x) - \varkappa(y)| \leq C_1 |x-y|,$$
for $x, y \in U$.
Let $X \subset U$ and suppose $\varkappa(X)$ is $\nu$-porous on balls from scales $\alpha_0$ to $\alpha_1$. Then $X$ is $\frac{\nu}{2C_1^2}$-porous on balls from scales $C_1 \alpha_0$ to $C_1 \alpha_1$.
\end{lemma}

\begin{proof}
Let $x \in \R^n$ and pick $C_1 \alpha_0 < R < C_1 \alpha_1$. We need to show that there exists $y \in B_{R/2}(x)$ such that 
$B_{\nu R/C_1^2}(y) \cap X = \emptyset$. Note that exists $\tilde{x} \in B_{R/2}(x)$ such that either $B_{R/4}(\tilde{x}) \subset U \cap B_{R/2}(x)$ or $B_{R/4}(\tilde{x}) \cap U = \emptyset$. Therefore, it suffices to assume that $B_{R/4}(x) \subset U$ and show there exists $y \in B_{R/4}(x)$ such that 
$B_{\nu R/2C_1^2}(y) \cap X = \emptyset$.

We know
$$B_{\frac{R}{4C_1}}(\varkappa(x)) \subset \varkappa\left(B_{\frac{R}{4}}(x)\right).$$
By the ball porosity of $\varkappa(X)$, there exists $\varkappa(y) \in B_{R/4C_1}(\varkappa(x))$ such that 
$B_{\nu R/2C_1}(\varkappa(y)) \cap \varkappa(X) = \emptyset.$
Thus,
$$\varkappa^{-1}\left(B_{\frac{\nu R}{2C_1}}(\varkappa(y)) \cap \varkappa(U) \right)\cap X = \emptyset.$$
Finally, since $B_{\nu R/2C_1^2}(y)\subset \varkappa^{-1}(B_{\nu R/2C_1}(\varkappa(y))\cap \varkappa(U)) \cup (\R^n \setminus U)$,  $B_{\nu R/2C_1^2}(y) \cap X = \emptyset$. The proof concludes by noting that $y \in B_{R/4}(x)$. 
\end{proof}

We also show that line porosity is preserved under diffeomorphisms. We do not use this fact in our paper. However, we still include it as it may be helpful for future results. 

\begin{lemma} \label{lem:diffeomorphism_line}
Let $\varkappa(x):\R^n \rightarrow \R^n$ be a diffeomorphism with 
\begin{equation}\label{eq:kappa_lipschitz}
C_1^{-1} |x-y| \leq |\varkappa(x) - \varkappa(y)| \leq C_1 |x-y|
\end{equation}
and for $1\leq i, j\leq n$, 
$$|\partial_i \partial_j \varkappa^{-1} (x)|\leq C_2.$$
Assume $0< \alpha_0 < \alpha_1 \leq \frac{\nu}{C_1 C_2 n}$.
Let $X \subset \R^n$ and suppose $\varkappa(X)$ is $\nu$-porous on lines from scales $\alpha_0$ to $\alpha_1$. Then $X$ is $\frac{\nu}{2C_1^2}$-porous on lines from scales $C_1 \alpha_0$ to $C_1\alpha_1$.
\end{lemma} 

\begin{proof}
Let $x, v \in \R^n$ with $|v| =1$ and pick $C_1 \alpha_0 < R < C_1 \alpha_1$. It suffices to show that there exists a $0 \leq t_0 \leq R$ such that $B_{\frac{\nu R}{2 C_1^2}}(x+ t_0 v) \cap X =\emptyset$.

From \eqref{eq:kappa_lipschitz}, $C_1^{-1} \leq |d\varkappa (x)v|$. Thus for $w \coloneqq \tfrac{d\varkappa(x) v}{C_1 |d \varkappa(x) v|}$, $\{\varkappa(x) + tw : 0 \leq t \leq R\} \subset \{\varkappa(x) + td\varkappa(x) v: 0 \leq t \leq R\}$ and $\alpha_0 \leq R |w| \leq \alpha_1$.
By the line porosity of $\varkappa(X)$, 
there exists $0 \leq t_0 \leq R$ such that 
$$B_{\frac{\nu R}{C_1}} (\varkappa(x) + t_0w) \cap \varkappa(X)  = \emptyset.$$

Thus,
$$\varkappa^{-1} \left(B_{\frac{\nu R}{C_1}} (\varkappa(x) + t_0w)\right) \cap X  = \emptyset.$$
Clearly,
$$B_{\frac{\nu R}{C_1^2}}(\varkappa^{-1}(\varkappa(x) + t_0w)) \subset \varkappa^{-1} \left(B_{\frac{vR}{C_1}} (\varkappa(x) + t_0w)\right).$$

From the Taylor expansion in $t$,
$$\varkappa^{-1}(\varkappa(x) + t_0w) \in B_{\frac{ C_2 n t_0^2}{2C_1^2}}\left(x + t_0d\varkappa^{-1}(\varkappa(x)) w\right) =B_{\frac{ C_2 n t_0^2}{2C_1^2}}\left( x + t_0 \frac{v}{C_1 |d \varkappa(x) v|}\right) .$$

As $\alpha_1 \leq \frac{\nu}{C_1 C_2 n}$, we know $\frac{C_2 n t_0^2}{2 C_1^2 } \leq \frac{\nu R}{2C_1^2}$. 

Therefore, 
$$B_{\frac{\nu R}{2C_1^2}}\left( x + t_0 v\right) \subset \varkappa^{-1} \left(B_{\frac{vR}{C_1}} (\varkappa(x) + t_0w)\right),$$
which gives $B_{\frac{\nu R}{2C_1^2}}\left( x + t_0 v\right) \cap X = \emptyset$.
\end{proof}


\subsection{Generalizations of the fractal uncertainty principle}\label{subsection:gen_FIO}


We first quote the following lemma which adapts Theorem ~\ref{thm:fractal_uncertainty} to  unbounded sets. 
\begin{lemma}[~\cite{Ki24}*{Proposition 2.19}] \label{lem:FUP_unbounded}
Set $0 < \nu \leq \frac{1}{3}$. Let
\begin{itemize}
\item $X_- \subset \R^n$ be $\nu$-porous on balls from scales $h$ to 1;
\item $X_+  \subset \R^n$ be $\nu$-porous on lines from scales $h$ to 1.
\end{itemize}
Then there exists $\beta, C >0$, depending only on $\nu$ and $n$, such that
\begin{equation*}
\|\1_{X_-} \cF_h \1_{X_+}\|_{L^2 (\R^n) \rightarrow L^2(\R^n)} \leq Ch^\beta.
\end{equation*}
\end{lemma}

We can extend Lemma ~\ref{lem:FUP_unbounded} by using a simple generalization of ~\cite{DJN22}*{Proposition 2.10}. Using the notation of ~\cite{DJN22}, we set $\gamma_0^\pm = \varrho$ and $\gamma_1^\pm =0$ to obtain the following.

\begin{lemma}
\label{lem:FUP_new_scales}
Set $0 < \nu \leq \frac{1}{3}$ and $\frac{1}{2} < \varrho \leq 1$. Let 
\begin{itemize}
\item $X_- \subset \R^n$ be $\nu$-porous on balls from scales $h^\varrho$ to 1;
\item $X_+ \subset \R^n$ be $\nu$-porous on lines from scales $h^\varrho$ to 1.
\end{itemize}
Then there exists $\beta, C >0$, where $C$ depends only on $\nu$ and $n$, while $\beta$ depends only on $\nu$, $n$, and $\varrho$ such that
\begin{equation*}
\|\1_{X_-} \cF_h \1_{X_+}\|_{L^2 (\R^n) \rightarrow L^2(\R^n)} \leq Ch^\beta.
\end{equation*}
\end{lemma}

We now follow the argument of ~\cite{BD18}*{\S \S 4.1--2}, generalizing Lemma ~\ref{lem:FUP_unbounded} first to operators with variable amplitude, then to operators with a general phase. 

Let $A= A(h): L^2(\R^n) \rightarrow L^2(\R^n)$ be of the form
$$Af(x) = h^{-\frac{n}{2}} \int_{\R^n} e^{-2 \pi i \lrang{x, \xi}/h} a(x, \xi) f(\xi) d\xi,$$
where $a(x, \xi) \in C^\infty_c(\R^{2n})$ satisfies for each multi-index $\alpha$ and constants $C_\alpha$, $C_a$,
\begin{equation}\label{eq:a_bounds}
\sup |\partial_x^{\alpha} a | \leq C_\alpha, \quad \diam \supp a \leq C_a.
\end{equation}

We begin by showing that Lemma ~\ref{lem:FUP_new_scales} holds when $\cF_h$ is replaced by  $A(h)$. We adapt the argument of ~\cite{BD18}*{Proposition 4.2} to higher dimensions. 
\begin{lemma}\label{lem:FUP_A}
Set $0 < \nu \leq \frac{1}{3}$ and $\frac{1}{2} < \rho, \varrho \leq 1$. Let 
\begin{itemize}
\item $X_- \subset \R^n$ be $\nu$-porous on balls from scales $h^\varrho$ to $1$;
\item $X_+ \subset \R^n$ be $\nu$-porous on lines from scales $h^\varrho$ to $1$.
\end{itemize}
Assume ~\eqref{eq:a_bounds} holds. There exists $\beta>0$ depending only on $\nu$, $\rho$, $\varrho$, and $n$  and $C>0$ depending only on $\nu$, $\rho$, $n$, $\{C_\alpha\}$, and $C_a$ such that for $h$ sufficiently small,
$$
\|\1_{X_-(h^\rho)} A(h) \1_{X_+(h^\rho)}\|_{L^2(\R^n) \rightarrow L^2(\R^n)} \leq Ch^{\beta}.
$$
\end{lemma}

\begin{proof}
In this proof, the constant $C$ varies, but always depends on $\nu$, $n$, $\rho$, $\{C_\alpha\}$, and $C_a$.  

We first claim 
\begin{equation}\label{eq:A_L2_bound}
\|A\|_{L^2(\R^n) \rightarrow L^2(\R^n)} \leq C.
\end{equation}

To show ~\eqref{eq:A_L2_bound}, we first compute the integral kernel of $A^* A$:
$$
\cK_{A^*A}(\xi, \eta) = h^{-n} \int_{\R^n} e^{2 \pi i\lrang{x, \xi -\eta}/h} \overline{a(x, \xi)} a(x, \eta) dx.
$$
Using ~\eqref{eq:a_bounds} and repeated integration by parts in $x$, for each $N \in \N$, there exists $C_N>0$ such that
$$ \left|\cK_{A^* A}(\xi, \eta)\right| \leq C_N h^{-n} \lrang{\frac{\xi-\eta}{h}}^{-N}.
$$
Thus by Schur's inequality (see ~\cite{Zw12}*{Theorem 4.21}), we know $\|A^* A\|_{L^2 \rightarrow L^2} \leq C$, which proves ~\eqref{eq:A_L2_bound}. 

Now note that
$$
\1_{X_-(h^{\rho})}A \1_{X_+(h^{\rho})}= \1_{X_-(h^\rho)} \cF_h^* A_1 + A_2 \cF_h A \1_{X_+(h^\rho)},
$$
where
$$
A_1 \coloneqq \1_{\R^n \setminus X_+(2h^\rho)} \cF_h A \1_{X_+(h^\rho)}, \quad A_2 \coloneqq \1_{X_-(h^\rho)}\cF_h^* \1_{X_+(2h^{\rho})}.
$$

From ~\eqref{eq:A_L2_bound}, we know 
$$
\|\1_{X_-(h^{\rho})}A \1_{X_+(h^{\rho})}\|_{L^2 \rightarrow L^2} \leq \|A_1\|_{L^2 \rightarrow L^2} + C \|A_2\|_{L^2 \rightarrow L^2}.
$$
We first prove the decay of $A_1$. We compute the integral kernel of $A_1$: 
$$
\cK_{A_1}(\xi, \eta) = h^{-n} \1_{\R^n \setminus X_+(2h^\rho)}(\xi) \1_{X_+(h^{\rho})}(\eta) \int_{\R^n} e^{2 \pi i \lrang{x, \eta-\xi}/h} a(x, \eta) dx.
$$

On $\supp \cK_{A_1}$, $|\xi - \eta|> h^{\rho}$. From ~\eqref{eq:a_bounds} and repeated integration by parts in $x$, for each $M \in \N$, there exists $C_M >0$ such that
$$|\cK_{A_1}(\xi, \eta)| \leq C_M h^{-n} \lrang{\frac{\xi - \eta}{h}}^{-M}.$$
Choosing $M \geq \frac{1 +n}{1 - \rho}$, via Schur's inequality we see
\begin{equation}\label{eq:A_1_bound}
\|A_1\|_{L^2 \rightarrow L^2} \leq Ch.
\end{equation}

We now estimate $\|A_2\|_{L^2 \rightarrow L^2}$. From Lemma ~\ref{lem:gen_porous_balls}, $X_-(h^{\rho})$ is $\frac{\nu}{2}$-porous on balls from scales \linebreak $\max(h^\varrho, \frac{2}{\nu}h^{\rho})$ to $1$ and from Lemma ~\ref{lem:gen_porous_lines}, $X_+(2h^{\rho})$ is $\frac{\nu}{2}$-porous on lines from scales $\max(h^\varrho, \frac{4}{\nu}h^{\rho})$ to $1$. 
For $\rho' \in (\tfrac{1}{2}, \min(\varrho, \rho))$ and sufficiently small $h$, $X_-(h^{\rho})$ is $\frac{\nu}{2}$-porous on balls from scales $h^{\rho'}$ to $1$ and $X_+(2h^{\rho})$ is $\frac{\nu}{2}$-porous on lines from scales $h^{\rho'}$ to $1$.

Then from Lemma ~\ref{lem:FUP_new_scales}, 
\begin{equation}\label{eq:A_2_bound}
\|A_2\|_{L^2 \rightarrow L^2} \leq Ch^{\beta}.
\end{equation}
From ~\eqref{eq:A_1_bound} and ~\eqref{eq:A_2_bound}, we finish the proof.
\end{proof}

We now generalize the fractal uncertainty principle to operators with general phase.

Let $B= B(h):L^2(\R^n) \rightarrow L^2(\R^n)$ be of the form 
\begin{equation}\label{eq:B(h)_def}
Bf(x) = h^{-\frac{n}{2}} \int_{\R^n} e^{i \Phi(x,y)/h} b(x,y) f(y) dy,
\end{equation}
where for some open set $U \subset \R^{2n}$,
\begin{equation}\label{eq:B(h)_conditions}
\Phi(x, y) \in C^\infty(U; \R), \quad b \in C^\infty_c(U), \quad \det \partial_{xy}^2\Phi \neq 0 \text{ on } U.
\end{equation}

From the condition $\det \partial_{xy}^2\Phi \neq 0$, locally we can write the graph of the twisted gradient of $\Phi$ in terms of some symplectomorphism $\kappa$ of open subsets of $T^* \R^n$:
\begin{equation*}\label{eq:kappa_generating_function}
(x, \xi) = \kappa(y, \eta) \quad \Leftrightarrow \quad \xi = \partial_x \Phi(x,y), \quad \eta = -\partial_y \Phi(x,y).
\end{equation*}
We see that $B$ is a semiclassical Fourier integral operator associated to $\kappa$.

\begin{proposition} \label{prop:FUP_gen_phase}
Set $0 < \nu \leq \frac{1}{3}$ and $\frac{3}{4} < \varrho, \rho < 1$. Let 
\begin{itemize}
\item $X_- \subset \R^n$ be $\nu$-porous on balls from scales $h^\varrho$ to 1;
\item $X_+ \subset \R^n$ be $\nu$-porous on lines from scales $h^\varrho$ to 1.
\end{itemize}
Assume ~\eqref{eq:B(h)_conditions} holds. 
Then there exists constants $\beta, C >0$ depending only on $\nu$, $\rho$, $\varrho$, $n$, $\Phi$, and $b$  such that for $h$ sufficiently small,
\begin{equation*}\label{eq:FUP_B}
\|\1_{X_-(h^{\rho})} B(h) \1_{X_+(h^{\rho})}\|_{L^2(\R^n) \rightarrow L^2(\R^n)} \leq Ch^{\beta/2}.
\end{equation*}
\end{proposition}

In the 1-dimensional version of Proposition~\ref{prop:FUP_gen_phase}, \cite{BD18}*{Proposition 4.3}, $\beta$ is independent of $\Phi$ and $b$. This is due to an argument that gives a  bound similar to~\eqref{eq:bilipschitz}, but with constants independent of $\Phi$ and $b$. This argument falls apart in higher dimensions due to the fact that $\det \partial^2_{xy} \Phi \neq \partial^2_{xy} \Phi$.

We adapt the proof of ~\cite{BD18}*{Proposition 4.3} to higher dimensions, starting with  the following lemma, generalized from ~\cite{BD18}*{Lemma 4.4}.

\begin{lemma}\label{lem:FUP_gen_phase}
Suppose the assumptions of Proposition~\ref{prop:FUP_gen_phase} hold. Then there exists $\beta, C>0$ depending only on $\nu$, $\rho$, $\varrho$, $n$, $\Phi$, and $b$ such for all balls $J$ of diameter $h^{1/2}$ and $h$ sufficiently small,
$$\|\1_{X_-(h^{\rho/2})} B(h) \1_{X_+(h^\rho) \cap J}\|_{L^2(\R^n) \rightarrow L^2(\R^n)} \leq Ch^{\beta/2}.$$
\end{lemma}

\begin{proof}
Breaking the symbol $b$ into pieces using a partition of unity, we may assume that
$\supp b \subset D_- \times D_+ \subset U$,
where $D_-,  D_+$ are balls .
We assume that $J \subset D_+$, else for $h$ sufficiently small,
$\1_{X_-(h^{\rho/2})} B(h) \1_{X_+(h^\rho) \cap J} =0$.

Let $y_0$ be  the center of $J$. Choose $x_0 \in D_-$.
By~ \eqref{eq:B(h)_conditions}, there exists some neighborhood $U_{x_0}$ of $x_0$ with $U_{x_0} \subset D_-$
where $\partial_y \Phi (\cdot, y_0)$ is locally invertible. Therefore, for
any $x_1, x_2 \in U_{x_0}$,
$$\frac{|\partial_y \Phi(x_1, y_0) - \partial_y \Phi(x_2, y_0)|}{|x_1 - x_2|} \leq \sup_{x \in U_{x_0}} \| \partial^2_{xy} \Phi(x, y_0)\|,$$
and
$$\frac{|x_1 - x_2|}{|\partial_y \Phi(x_1, y_0) - \partial_y \Phi(x_2, y_0)|} \leq \sup_{x \in U_{x_0}} \| \partial_x (\partial_y \Phi(x, y_0))^{-1}\| =\sup_{x \in [\partial_y \Phi(U_{x_0}, y_0)]^{-1}} \|(\partial_{xy}^2 \Phi(x, y_0))^{-1}\|.$$

From~ \eqref{eq:B(h)_conditions}, we see $\sup_{x \in U_{x_0}} \| \partial^2_{xy} \Phi(x, y_0)\|$ and  $\sup_{x \in [\partial_y \Phi(U_{x_0}, y_0)]^{-1}} \|(\partial_{xy}^2 \Phi(x, y_0))^{-1}\|$ are nonzero and depend continuously on $x_0$ and $y_0$. Therefore, there exists $C=C_{x_0, y_0}>0$ and a neighborhood $U_{y_0}$ of $y_0$  with $U_{x_0} \times U_{y_0} \subset U$ such that for  $y \in U_{y_0}$,
\begin{equation}\label{eq:bilipschitz}
C^{-1} |x_1 - x_2| \leq |\partial_y \Phi(x_1, y) - \partial_y \Phi(x_2, y)| \leq C |x_1 - x_2|.
\end{equation}
Thus, using another partition of unity to further decrease the support of $b$ and shrinking $D_-$, we may assume that
$$\supp b \subset D_- \times D'_+ \subset D_- \times D_+ \subset U,$$
where $D'_+$ is a ball with $D_+' \Subset D_+$
and for $(x_1, y), (x_2, y) \in D_- \times D'_+$, ~\eqref{eq:bilipschitz} holds.

Now define the function 
$$\phi: D_- \rightarrow \R^n, \quad \phi(x) = \frac{1}{2 \pi} (\partial_{y_1} \Phi(x, y_0), \ldots, \partial_{y_n} \Phi(x, y_0)).$$
From~\eqref{eq:B(h)_conditions}, $\phi : D_- \rightarrow \phi(D_-)$ is a diffeomorphism.
Now define $\Psi = \sum_{|\alpha| =2} \Psi_\alpha(x, y) (y-y_0)^\alpha \in C^\infty(D_- \times D_+)$ to be the remainder in the following Taylor expansion of $\Phi$:
$$\Phi(x,y) = \Phi(x, y_0) + 2 \pi \lrang{\phi(x), y-y_0} + \Psi(x,y), \quad x \in D_-, y \in D_+.$$

Define the isometries $W_-: L^2(D_-) \rightarrow L^2(\phi(D_-))$ and $W_+: L^2(\R^n) \rightarrow L^2(\R^n)$ by 
$$W_-f(x) = e^{-i \Phi(\phi^{-1}(x), y_0)/h} |\det \partial_x (\phi^{-1})(x)|^{1/2} f(\phi^{-1}(x)), \quad W_+f(y)  = h^{-\frac{n}{4}} f\left(\frac{y-y_0}{h^{1/2}}\right).$$
Let $\chi \in C^\infty_c(B_1(0); [0,1])$ be a cutoff function such that $\chi =1$ near $B_{\frac{1}{2}}(0)$. Then set
$$\chi_J(y) \coloneqq \chi\left( \frac{y-y_0}{h^{1/2}}\right).$$
Clearly $\chi_J = 1$ on $J$. 

Set $A= A(h) \coloneqq W_- B(h) \chi_J W_+$. Then 
$$Af(x) = \tilde{h}^{-\frac{n}{2}} \int_{\R^n} e^{2 \pi i \lrang{x, \xi}/\tilde{h}} a(x, \xi; \tilde{h}) f(\xi) d\xi,$$
where $\tilde{h} \coloneqq h^{\frac{1}{2}}$ and 
$$a(x, \xi; \tilde{h}) = |\det \partial_x (\phi^{-1})(x)|^{1/2} e^{i \sum \Psi_\alpha (\phi^{-1}(x), y_0 + \tilde{h}\xi) \xi^\alpha} b(\phi^{-1}(x), y_0 + \tilde{h}\xi) \chi(\xi).$$

Note that $a$ satisfies ~\eqref{eq:a_bounds}, where $C_\alpha$, $C_a$ depend only on $\Phi$ and $b$. For $h$ sufficiently small,
\begin{align*}
\|\1_{X_-(h^{\rho/2})} B \1_{X_+(h^\rho) \cap J}\|_{L^2 \rightarrow L^2}&\leq \|W_- \1_{X_-(h^{\rho/2}) \cap D_-} B \chi_J \1_{X_+(h^\rho)} W_+\|_{L^2 \rightarrow L^2}\\
&\leq \left\|\1_{\tilde{X}_-(C \tilde{h}^\rho)}A \1_{\tilde{X}_+(\tilde{h}^{2 \rho -1})} \right\|_{L^2 \rightarrow L^2} \\
&\leq \left\|\1_{\tilde{X}_-(\tilde{h}^{2\rho-1})}A \1_{\tilde{X}_+(\tilde{h}^{2 \rho -1})} \right\|_{L^2 \rightarrow L^2},
\end{align*}
where $\tilde{X}_- \coloneqq \phi(X_- \cap D_-)$, $\tilde{X}_+ \coloneqq h^{-1/2}(X_+-y_0)$.

Using~\eqref{eq:bilipschitz}, from Lemma ~\ref{lem:diffeomorphism_ball}, we see $\tilde{X}_-$ is $\frac{\nu}{2C^2}$-porous on balls from scales $C h^\varrho$ to $1$. Thus, $\tilde{X}_-$ is $\frac{\nu}{2C^2}$-porous on balls from scales $\tilde{h}^{2 \varrho -1}$ to $1$.
From Lemma ~\ref{lem:affine_line}, $\tilde{X}_+$ is $\nu$-porous on lines from scales $\tilde{h}^{2 \varrho -1}$ to $1$. 

Then by Lemma ~\ref{lem:FUP_A}, 
$$\left\|\1_{\tilde{X}_-(\tilde{h}^{2\rho-1})}A \1_{\tilde{X}_+(\tilde{h}^{2 \rho -1})} \right\|_{L^2 \rightarrow L^2} \leq C h^{\beta/2},$$
which completes the proof.
\end{proof}

\begin{proof}[Proof of Proposition ~\ref{prop:FUP_gen_phase}]

In the following, $C$ is a constant that can vary but only depends on $\nu, \rho, \varrho, n, \Phi$, and $b$. Using a similar argument to the one in Lemma~\ref{lem:FUP_gen_phase}, since $\Phi \in C^\infty(U; \R)$ and $\det \partial^2_{x,y} \Phi \neq 0$ on $U$, after using a partition of unity for $b$ and shrinking $U$, we may assume that 
\begin{equation}\label{eq:Phi_bound}
C^{-1}|y-y'| \leq |\partial_x \Phi(x,y) - \partial_x \Phi(x, y')| \quad \text{for all } (x, y), (x, y') \in U.  
\end{equation}
From ~\cite{DZ16}*{Lemma 3.3}, there exists $\psi = \psi(x;h) \in C^\infty(\R^n; [0,1])$ such that for some global constants $C_{\alpha,\psi}$,
\begin{equation*}\label{eq:psi_support}
\psi =1 \text{ on } X_-(h^\rho), \quad \supp \psi \subset X_-(h^{\rho/2});
\end{equation*}
\begin{equation}\label{eq:psi_derivative_bounds}
\sup |\partial^\alpha_x \psi| \leq C_{\alpha, \psi}h^{-\rho |\alpha|/2}.
\end{equation}

Take the smallest ball $D_+$ such that $\supp b \subset \R^n \times D_+$. Take a maximal set of $\frac{1}{2}h^{1/2}$-separated points
$$y_1, \ldots, y_N \in X_+(h^\rho) \cap D_+, \quad N \leq Ch^{-\frac{n}{2}}$$
and let $J_k$ be the ball of diameter $h^{1/2}$ centered at $y_k$. Define the operators
$$B_k \coloneqq \sqrt{\psi} B \1_{X_+(h^\rho) \cap J_k}, \quad k=1, \ldots, N.$$

Noting that $X_+(h^\rho) \cap D_+ \subset \bigcup_k (X_+(h^\rho) \cap J_k)$, we see 
\begin{equation}\label{eq:cotlar_stein_setup}
\left\|\1_{X_-(h^\rho)}B \1_{X_+(h^\rho)}\right\|_{L^2 \rightarrow L^2} \leq \left\|\sqrt{\psi} B \1_{X_+(h^\rho) \cap D_+}\right\|_{L^2 \rightarrow L^2} \leq \left\|\sum_{k=1}^N B_k \right\|_{L^2 \rightarrow L^2}.
\end{equation}

We estimate the right-hand side of ~\eqref{eq:cotlar_stein_setup} using the Cotlar--Stein Theorem ~\cite{Zw12}*{Theorem C.5}. We say that two points $y_k, y_m$ are \emph{close} if $|y_k -y_m| \leq 10h^{1/2}$. Else, $y_k, y_m$ are \emph{far}. Each point is close to at most $100^n$ points.

Suppose $y_k, y_m$ are far. Thus, $J_k \cap J_m = \emptyset$, which implies
\begin{equation}\label{eq:far_estimate1}
B_k B_m^* =0.
\end{equation}
We claim that 
\begin{equation}\label{eq:far_estimate2}
\|B_k^* B_m\|_{L^2 \rightarrow L^2} \leq Ch^{2n}.
\end{equation}
To show ~\eqref{eq:far_estimate2}, we first compute the integral kernel of $B_k^*B_m$:
$$\cK_{B_k^*B_m}(y,y') = h^{-n} \1_{X_+(h^\rho) \cap J_k}(y) \1_{X_+(h^\rho) \cap J_m}(y') \int_{\R^n} e^{i(\Phi(x,y')-\Phi(x,y))/h} b(x,y') \overline{b(x,y)}\psi(x) dx.$$
Note that if $(y,y') \in \supp \cK_{B_k^*B_m}$, then $|y-y'| \geq h^{1/2}$. We integrate by parts in $x$, using 
$$L=\frac{h}{i|\partial_x(\Phi(x,y') - \Phi(x, y))|^2} \lrang{\partial_x(\Phi(x,y') - \Phi(x, y)), \partial_x}.$$
From ~\eqref{eq:Phi_bound} and ~\eqref{eq:psi_derivative_bounds}, we gain $h^{(1-\rho)/2}$ after each integration of integration by parts. As $\rho<1$, after finitely many steps, we conclude ~\eqref{eq:far_estimate2} by Schur's inequality (see ~\cite{Zw12}*{Theorem 4.21}).

To handle $y_k , y_m$ close, from Lemma ~\ref{lem:FUP_gen_phase}, uniformly in $k$,
\begin{equation}\label{eq:Bn_bound}
\|B_k\|_{L^2 \rightarrow L^2} \leq \|\1_{X_-(h^{\rho/2})}B \1_{X_+(h^\rho) \cap J_k}\|_{L^2 \rightarrow L^2} \leq Ch^{\beta/2}.  
\end{equation}

Using ~\eqref{eq:far_estimate1}, ~\eqref{eq:far_estimate2}, and ~\eqref{eq:Bn_bound} to apply the Cotlar-Stein theorem, we know 
$$\left\|\sum_{k=1}^N B_k\right\|_{L^2 \rightarrow L^2} \leq Ch^{\beta/2}.$$
From ~\eqref{eq:cotlar_stein_setup}, we conclude Proposition ~\ref{prop:FUP_gen_phase}.
\end{proof}

Finally, we adapt Proposition ~\ref{prop:FUP_gen_phase} to manifolds. 
Suppose $M, \tilde{M}$ are $n$-dimensional compact manifolds. 
Let $B= B(h):L^2(M) \rightarrow L^2(\tilde{M})$ be of the form 
\begin{equation}\label{eq:B(h)_def_manifold}
Bf(x) = h^{-\frac{n}{2}} \int_{M} e^{i \Phi(x,y)/h} b(x,y) f(y) dy,
\end{equation}
where for some open set $U \subset M \times \tilde{M}$,
\begin{equation}\label{eq:B(h)_conditions_manifold}
\Phi \in C^\infty(U; \R), \quad b \in C^\infty_c(U), \quad \det \partial_{xy}^2\Phi \neq 0 \text{ on } U.
\end{equation}
Note that $\det \partial_{xy}^2 \Phi \neq 0$ is a coordinate-invariant property. 

We pick coordinate charts $\psi_j: M_j \rightarrow X_j$, $\tilde{\psi}_j: \tilde{M}_j \rightarrow \tilde{X}_j$, where $M= \bigcup M_j$, $\tilde{M} = \bigcup \tilde{M}_j$, $X_j, \tilde{X}_j \subset \R^n$ for open sets $M_j$, $\tilde{M}_j$, $X_j$, $\tilde{X}_j$ and $\psi_j$, $\tilde{\psi}_j$ are diffeomorphisms  such that for some $C_0>0$, 
\begin{equation}\label{eq:det_bounds}
|\det \partial \psi_j|, |\det \partial \tilde{\psi}_j| \leq C_0.
\end{equation}
As $M, \tilde{M}$ are compact, we can assume that there are finitely many $\psi_j$, $\tilde{\psi}_j$.

\begin{proposition} \label{prop:FUP_manifolds}
Set $0 < \nu \leq \frac{1}{3}$ and $\frac{3}{4} < \varrho, \rho <1$. Let $X_- \subset \tilde{M}$ and $X_+ \subset M$ such that 
\begin{itemize}
\item for all $k$, $\tilde{\psi}_k(X_- \cap \tilde{M}_k)$ is $\nu$-porous on balls from scales $h^\varrho$ to 1;
\item for all $j$, $\psi_j(X_+ \cap M_j)$ is $\nu$-porous on lines from scales $h^\varrho$ to 1.
\end{itemize}
Assume ~\eqref{eq:B(h)_conditions_manifold} holds. 
Then there exists constants $C, \beta >0$ depending on  $\nu$, $\rho$, $n$, $\Phi$, $M$, $\tilde{M}$, $\psi_j$, $\tilde{\psi}_j$, and $b$ such that for $h$ sufficiently small,
\begin{equation*}\label{eq:FUP_B_manifold}
\|\1_{X_-(h^{\rho})} B(h) \1_{X_+(h^{\rho})}\|_{L^2(M) \rightarrow L^2(\tilde{M})} \leq Ch^{\beta/2}.
\end{equation*}
\end{proposition}

\begin{proof}

Let $\sum \chi_j =1$ be a partition of unity on $M$ subordinate to $M_j$ and $\sum \tilde{\chi_j}=1$ be a partition of unity on $\tilde{M}$ subordinate to $\tilde{M}_j$.  We have
$$B= \sum_{j,k} \tilde{\chi}_k B \chi_j = \sum_{j,k} \tilde{\psi}_k^* B_{jk} \psi_j^{-*} \chi_j,$$
where $B_{jk}: L^2(X_j) \rightarrow L^2(\tilde{X}_k)$ is given by $B_{jk} = \tilde{\psi}_k^{-*} \tilde{\chi}_k B \chi_j \psi_j^*$. More specifically,
\begin{align*}
&B_{jk} f(x)\\
&= h^{-\frac{n}{2}} \int_{M} e^{i \Phi(\tilde{\psi}^{-1}_k (x), y)} \tilde{\chi}_k (\tilde{\psi}_k^{-1}(x)) \chi_j(y) b(\tilde{\psi}_k^{-1}(x), y) f(\psi_j(y)) dy  \\
&=h^{-\frac{n}{2}} \int_{\R^n} e^{i \Phi(\tilde{\psi}^{-1}_k (x), \psi_j^{-1}(y))} \tilde{\chi}_k (\tilde{\psi}_k^{-1}(x)) \chi_j(\psi_j^{-1}(y)) b(\tilde{\psi}_k^{-1}(x), \psi_j^{-1} (y)) f(y) |\det \partial \psi_j^{-1}(y)| dy,   
\end{align*}
for some $g_j >0$.
We see that $B_{jk}$ is of the form ~\eqref{eq:B(h)_def}.
From \eqref{eq:det_bounds}, for $h$ sufficiently small,
\begin{align*}
\|\1_{X_- (h^\rho)} B \1_{X_+(h^\rho)}\|_{L^2(M) \rightarrow L^2(\tilde{M})} & \leq \sum_{j, k} \|\1_{X_-(h^\rho)}  \tilde{\psi}_k^* \tilde{\chi}_k B_{jk} \psi_j^{-*} \chi_j \1_{X_+(h^\rho)}\|_{L^2(M ) \rightarrow L^2(\tilde{M})}\\
& \leq  \sum_{j, k} \|\1_{\tilde{\psi_k}(X_-(h^\rho) \cap \tilde{M}_k)} B_{jk} \1_{\psi_j(X_+(h^\rho) \cap M_j)} \|_{L^2(\R^n) \rightarrow L^2(\R^n)}\\
& \leq C \sum_{j, k} \|\1_{\tilde{\psi_k}(X_- \cap \tilde{M}_k)(h^\rho)} B_{jk} \1_{\psi_j(X_+ \cap M_j)(h^\rho)} \|_{L^2(\R^n) \rightarrow L^2(\R^n)}.
\end{align*}
Thus, by our porosity assumptions on $\tilde{\psi}_k(X_- \cap \tilde{M}_k)$ and $\psi_j(X_+ \cap M_j)$, by Proposition ~\ref{prop:FUP_gen_phase}, we conclude that there exists $C, \beta>0$ such that  
$\|\1_{X_-} B \1_{X_+}\|_{L^2(M) \rightarrow L^2(\tilde{M})} \leq C h^{\beta/2}$.
\end{proof}


\section{Proof of Theorem ~\ref{thm:support} up to Lemma~\ref{lem:A_w_decay}}


\label{section:reduction}
Let $(M,g)$ be a compact hyperbolic $(n+1)$-dimensional manifold.  Let $u_j$ be a sequence of $L^2$-normalized eigenfunctions of $-\Delta$ with eigenvalues $h_j^{-2}$ that converges semiclassically to $\mu$.
Set 
$$U_{\mu} \coloneqq S^*M \setminus \supp \mu$$ and suppose that $U_\mu$ is $U_1^-$-dense.
Recalling the invariance of semiclassical measures under the geodesic flow, to prove Theorem~\ref{thm:support}, it suffices to find a contradiction. 

\subsection{Partition of unity}\label{subsection:partition_of_unity}
Using the partition of unity constructed in Lemma~\ref{lem:a1_a2_safe}, we build a microlocal partition of unity.  
We follow ~\cite{DJ18}*{\S 3.1} and ~\cite{ADM24}*{Lemma 4.4}.

\begin{lemma}\label{lem:partition_of_unity}
There exists a pseudodifferential partition of unity
\begin{equation}\label{eq:partition_of_unity}
I=A_0+A_1+A_2,\quad
A_0\in\Psi^0_h(M),\quad
A_1,A_2\in\Psi^{\comp}_h(M)
\end{equation}
endowed with the following properties. 
\begin{itemize}
\item The wavefront set of $A_0$ is bounded away from the cosphere bundle $S^*M$. In particular,
\begin{equation}\label{eq:WF(A0)}
\WF_h(A_0)\cap \{\tfrac{1}{2}\leq |\xi|_g\leq 2\}=\emptyset,\quad
\WF_h(I-A_0)\subset \{\tfrac{1}{4}<|\xi|_g<4\}.
\end{equation}
\item For $j=1,2$ and $a_j \coloneqq \sigma_h(A_j)$, there exists $U_1^-$-dense $U_j \subset S^*M$ such that 
\begin{equation}\label{eq:supp_a1_a2}
U_j \Subset S^*M \setminus \supp a_j.
\end{equation}
\item For $j= 1,2$,
\begin{equation}\label{eq:abchi}
a_j = b \chi_j, 
\end{equation}
where $\chi_j$ is a homogeneous function of order 0 and $b$ depends only on $|\xi|_g$ with $\{\tfrac{1}{2} \leq |\xi|_g \leq 2\} \subset \supp b \subset \{\tfrac{1}{4} \leq |\xi|_g \leq 4\}$.

\item $A_1$ is controlled by $\mu$,
that is
\begin{equation}
  \label{eq:A-1-a}
\WF_h(A_1)\cap \supp \mu = \emptyset.
\end{equation}
\end{itemize}
 
\end{lemma}

\begin{proof}
Set $A_0:=\psi_0(-h^2\Delta)$, where
$\psi_0\in C^\infty(\mathbb R;[0,1])$ satisfies
$$
\supp\psi_0\cap [1/4,4]=\emptyset,\quad
\supp(1-\psi_0)\subset (1/16, 16).
$$
Then ~\eqref{eq:WF(A0)} follows from the fact that $\sigma_h(\psi_0(-h^2 \Delta)) = \psi_0(|\xi|_g^2)$.

We now construct $A_1, A_2$. From Lemma ~\ref{lem:a1_a2_safe}, there exists $\chi_1, \chi_2 \in C^\infty(S^*M ;[0,1])$ such that $\chi_1 + \chi_2 =1$, $\supp \chi_1 \subset U_\mu$, and $S^*M \setminus \supp \chi_1$, $S^*M \setminus \supp \chi_2$ are $U_1^-$-dense. By Lemma~\ref{lem:safe_compact}, there exists compact $U_1^-$-dense sets $U_1, U_2 \subset S^*M$ such that  $U_1 \Subset S^*M \setminus \supp \chi_1$ and $U_2 \Subset S^*M \setminus \supp \chi_2$. We extend $\chi_1$ and $\chi_2$ to be homogeneous functions of order $0$ on $T^*M \setminus 0$.

We write $I -A_0 = \op_h(b) + R$, where $R= \cO(h^\infty)_{\Psi_h^{\comp}}$ and $b(x, \xi) = \tilde{\psi}(|\xi|^2_g)$ for $\tilde{\psi} \in C^\infty_c (\R, [0,1])$ with $\supp \tilde{\psi} \subset [\tfrac{1}{16}, 16]$.

Then set
$$a_1 \coloneqq \chi_1 b, \quad a_2 \coloneqq \chi_2 b, \quad A_1 \coloneqq \op_h(a_1) + R, \quad A_2 \coloneqq \op_h(a_2).$$ 
The statements ~\eqref{eq:partition_of_unity}, \eqref{eq:supp_a1_a2}, ~\eqref{eq:abchi}, and ~\eqref{eq:A-1-a} follow immediately.
\end{proof}

\subsection{Dynamical refinement of partition of unity}\label{subsection:refine}
For $T \in \N$, define the set of words 
$$\cW(T) \coloneqq \{1, 2\}^T = \{\w =w_0 \ldots w_{T-1} : w_0, \ldots, w_{T-1} \in \{1,2\} \} .$$
Let $A: L^2(M) \rightarrow L^2(M)$ be a bounded operator and recall the notation $A(t) \coloneqq U(-t)AU(t)$ from ~\eqref{eq:T_notation}. Then for $\w =w_0 \ldots w_{T-1} \in \cW(T)$, set 
$$A_{\w} \coloneqq A_{w_{T-1}}(T-1) \ldots A_{w_1}(1)A_{w_0}(0).$$

We now carefully pick the length of our words, i.e., the values of $T$ we will use. Fix
\begin{equation}\label{eq:rho_def}
\rho \in (3/4,1).   
\end{equation}
Define
\begin{equation}\label{eq:T_def}
T_0 \coloneqq \lrceil{\frac{\rho}{4} \log h^{-1}}, \quad T_1 \coloneqq 4T_0 \approx \rho \log h^{-1}.
\end{equation}

We claim that $T_0$ and $T_1$ are chosen so that for $\w \in \cW(T_0)$ or $\w \in \cW(T_1)$, up to an error term, $A_\w$ can be  respectively written as the quantization of a symbol in $S^{\comp}_{L_s, \rho/4}$ or $S^{\comp}_{L_s, \rho}$. Specifically, for 
$$a_\w \coloneq \prod_{j=0}^{T-1} (a_{w_j} \circ \phi_j), \quad \w \in \cW(T),$$
we have the following lemma. 

\begin{lemma}[~\cite{DJ18}*{Lemma 3.2}]\label{lem:a_w_symbol_class}
For each $\w \in \cW(T_0)$, we have (with bounds independent of $\w$)
$$a_\w \in S^{\comp}_{L_s, \rho/4}(T^*M \setminus 0), \quad A_\w = \op_h^{L_s}(a_\w) + \cO(h^{3/4})_{L^2 \rightarrow L^2}$$
and for $\w \in \cW(T_1)$, we have
$$a_\w \in S_{L_s, \rho}^{\comp}(T^*M \setminus 0), \quad A_\w = \op_h^{L_s}(a_\w) + \cO(h^{1- \rho -})_{L^2 \rightarrow L^2},$$
where the constants in $\cO(\cdot)$ are uniform in $\w$. 
\end{lemma}

The $h^{3/4}$-remainder for $T_0$  is necessary for Lemma \ref{lem:ac_bound} to hold. Although the choice of $T_1$ gives a larger  error term, a propagation time close to $\rho \log h^{-1}$ is needed for a later application of the fractal uncertainty principle in Lemma \ref{lem:a_hyperbolic_ball_porosity} and Lemma \ref{lem:a_hyperbolic_line_porosity}.

We outline the proof of Lemma~\ref{lem:a_w_symbol_class} to provide intuition, but defer the full proof to ~\cite{DJ18}. 
Recall that uniformly in $t \in [0, \rho \log h^{-1}]$, from  \eqref{eq:log_propagate}, we know $a \circ \phi_t \in S^{\comp}_{L_s, \rho, 0} (T^*M \setminus 0)$  and from \eqref{eq:egorov_log}, we have an Egorov's theorem.
Thus, we can use the following result, which combines the statements of  \cite{DJ18}*{Lemma A.1} and \cite{DJ18}*{Lemma A.6}.

\begin{lemma}[\cite{DJ18}*{Lemmas A.1, A.6}]
  \label{lem:many_product}
Let $C$ be an arbitrary fixed constant and assume that $a_1,\dots,a_N\in S^{\comp}_{L,\rho,\rho'}(U)$, $1\leq N\leq C\log h^{-1}$
are such that $\sup |a_j|\leq 1$, and
each $S^{\comp}_{L,\rho,\rho'}(U)$ seminorm of $a_j$ is bounded uniformly in $j$.
Then for all small $\varepsilon>0$,
the product $a_1\ldots a_N$ lies in $S^{\comp}_{L,\rho+\varepsilon,\rho'+\varepsilon}(U)$.
For $A_1,\dots,A_N : L^2(M) \rightarrow L^2(M)$ such that
$A_j=\op_h^L(a_j)+\mathcal O(h^{1-\rho-\rho'-})_{L^2\to L^2}$ where the constants in $\mathcal O(\cdot)$
are independent of $j$, 
$$
A_1\ldots A_N=\op_h^L(a_1\ldots a_N)+\mathcal O(h^{1-\rho-\rho'-})_{L^2\to L^2}.
$$
\end{lemma}

Lemma~\ref{lem:a_w_symbol_class} then follows.

We now take weighted sums of the operators $A_\w$. For a function $c: \cW(T) \rightarrow \C$, define the operator 
\begin{equation*}
A_c \coloneqq \sum_{\w \in \cW(T)} c(\w) A_\w,
\end{equation*}
with symbol 
\begin{equation*}
a_c \coloneqq \sum_{\w \in \cW(T)} c(\w) a_\w.
\end{equation*}  If $c=\1_E$ for $E \subset \cW(T)$, we use the notation $A_E$ to denote $A_{\1_E}$.

The following lemma shows that, modulo a small remainder, $A_c$ is pseudodifferential.
\begin{lemma}[~\cite{DJ18}*{Lemma 4.4}]\label{lem:ac_bound}
Assume $\sup |c| \leq 1$. Then for $T= T_0$,
$$a_c \in S^{\comp}_{L_s, \frac{1}{2}, \frac{1}{4}}(T^*M \setminus 0), \quad A_c = \op_h^{L_s}(a_c) + \cO(h^{1/2}),$$
where the $S^{\comp}_{L_s, \frac{1}{2}, \frac{1}{4}}$-seminorms of $a_c$ and the constant in $\cO(\cdot)$ are independent of $c$.
\end{lemma}

We also quote the following lemma which shows an “almost monotonicity” property for norms of the operators $A_c$.
\begin{lemma}[~\cite{DJ18}*{Lemma 4.5}]\label{lem:ac_ad_bound}
Assume $c, d: \cW(T_0) \rightarrow \R$ and $|c(w)| \leq d(w) \leq 1$ for all $\w \in \cW(T_0)$. Then for all $u \in L^2(M)$, we have
$$\|A_c u\|_{L^2} \leq \|A_d u\|_{L^2} + Ch^{1/8} \|u\|_{L^2},$$
where the constant $C$ is independent of $c, d$.
\end{lemma}

Define the following function $F:\cW(T_0) \rightarrow [0,1]$, which gives the proportion of the digit 1 in a word $\w = w_0 \ldots w_{T_0-1}$:
\begin{equation}\label{eq:F_def}
F(\w) \coloneqq \frac{|\{k \in \{0,\ldots, T_0 -1\}: w_k=1\}|}{T_0}.
\end{equation}
For $\alpha \in (0, \frac{1}{2})$, which we later select to be sufficiently small in ~\eqref{eq:alpha_def}, set 
\begin{equation*}
\cZ = \{ \w \in \cW(T_0): F(\w) >\alpha\}.
\end{equation*}
We call $\cZ$ the set of \emph{controlled words}. This is due to the fact that
for $\w \in \cZ$, if $(x, \xi) \in \supp a_\w$, then at least $\alpha T_0$ of the points $\phi_0(x, \xi), \phi_1(x, \xi), \ldots, \phi_{T_0-1}(x, \xi)$ lie in $\supp a_1$, which from ~\eqref{eq:A-1-a} is controlled by $\mu$.

We use $\cZ$ to split $\cW(2T_1)$ into the following two disjoint sets, where a word in $\cW(2T_1)$  is now written as a concatenation of $8$ elements of $\cW(T_0)$. 
Specifically, set
\begin{equation}\label{eq:def_X}
\begin{split}
\cY&\coloneqq \{\w^{(1)} \ldots\w^{(8)} : \w^{(k)} \in \cZ \text{ for some } 1 \leq k \leq 8 \},\\
\cX&\coloneqq \cW(2T_1) \setminus \cY=\{\w^{(1)} \ldots \w^{(8)} :  \w^{(k)} \in \cW(2T_1) \setminus \cZ \text{ for all } 1 \leq k \leq 8\}.
\end{split}
\end{equation}
We call elements of $\cX$ \emph{uncontrolled long logarithmic words} and elements of $\cY$ \emph{controlled long logarithmic words}.

Since $P$ and $A_1 + A_2 =I - A_0$ are both functions of $\Delta$, $A_1 + A_2$ and $P$ commute. Therefore, $A_1 + A_2$ and $U(t)$ commute. We see
\begin{equation} \label{eq:A_1A_2}
A_{\cW(T)} = (A_1 + A_2)^T.
\end{equation}

From ~\cite{DJ18}*{Lemma 3.1}, we know for all $T \geq 0$ and $u \in H^2(M)$,
\begin{equation}\label{eq:A_1A_2_upperbound}
\left\|u-(A_1 + A_2)^T u\right\|_{L^2} \leq C \left\|(-h^2 \Delta -I)u\right\|_{L^2}.
\end{equation}
The proof in~\cite{DJ18} exploits the fact that $A_{\cW(T)}$ is equal to $I$ microlocally near $S^*M$. 

Recall that $u_j$ is our sequence of $L^2$-normalized eigenfunctions of $-\Delta$ with eigenvalues $h_j^{-2}$ that converges semiclassically to $\mu$.
Then,
$$\|u_j\|_{L^2} \leq  \|u_j - (A_1 + A_2)^{2T_1} u_j\|_{L^2}+ \|A_\cX u_j\|_{L^2} + \|A_\cY u_j\|_{L^2}  =   \|A_\cX u_j\|_{L^2} + \|A_\cY u_j\|_{L^2}.$$

Therefore Theorem~\ref{thm:support} follows from the next two propositions.  

\begin{proposition}\label{prop:AY_decay}
As $j \rightarrow \infty$, we have
$$\|A_\cY u_j\|_{L^2} \rightarrow 0.$$
\end{proposition}
We prove this proposition in \S \ref{subsection:decay_of_Y}, relying on the fact that $a_1$ is controlled by $\mu$.

\begin{proposition}\label{prop:AX_decay}
There exists some $C, \beta >0$, depending only on $M$, $a_1$, $a_2$, $\rho$ such that 
$$\|A_{\cX} \|_{L^2 \rightarrow L^2} \leq Ch^{\beta/10}.$$
\end{proposition}
We prove this proposition in \S\ref{subsection:decay_of_X} and \S\ref{subsection:proof_Aw_decay}, using that  $U_\mu$ is $U_1^-$-dense.

\subsection{Proof of Proposition ~\ref{prop:AY_decay}}\label{subsection:decay_of_Y}

We adapt ~\cite{DJ18}*{\S 4}, which in turn used many of the tools from ~\cite{An08}*{\S 2}. The following argument could be easily modified to hold for $o(h/\log h^{-1})$ quasimodes.

We first control the behavior of propagated operators, following the proof of~\cite{DJ18}*{Lemma 4.2}. 
\begin{lemma}\label{lem:A_bound}
Let $A: L^2(M) \rightarrow L^2(M)$ be uniformly bounded in $h$. Then,
$$\|A(t) u_j\|_{L^2} \leq \|A u_j\|_{L^2}.$$
\end{lemma}
    
\begin{proof}
We know
$$\partial_t \left(e^{i t/h} U(t)\right) = -\frac{i}{h} e^{it/h} U(t)(P-I).$$
Integrating the above equality from $0$ to $t$, we know 
$$\left\|U(t)u_j - e^{-it/h}u_j\right\|_{L^2} = \left\|e^{it/h} U(t)u_j -u_j\right\|_{L^2} \leq \frac{|t|}{h} \|(P-I)u_j\|_{L^2}.$$

Then since $A : L^2(M) \rightarrow L^2(M)$ is uniformly bounded in $h$, 
$$\|A(t) u_j\|_{L^2} = \|A U(t)u_j\|_{L^2} \leq \|A u_j\|_{L^2} + \frac{C|t|}{h} \|(P-I) u_j\|_{L^2}.$$

For $\psi_E (\lambda) \coloneqq (\psi_P(\lambda) -1)/(\lambda -1)$, by \eqref{eq:P_def},
$$P-I = \psi_E(-h_j^2 \Delta) (-h_j^2\Delta -I).$$
Thus, 
$$\|(P-I)u_j\|_{L^2} \leq C\|(-h_j^2 \Delta -I)u_j\|_{L^2} =0,$$
which finishes the proof.
\end{proof}

Now note that $$\|A_1 u_j\|^2_{L^2} = \|\op_h(a_1)u_j\|_{L^2}^2 \rightarrow \int_{T^*M} |a_1|^2 d \mu =0.$$
Therefore,
\begin{equation}\label{eq:A1_decay}
\|A_1(t) u_j\|_{L^2} \leq \|A_1 u_j\|_{L^2} \rightarrow 0.
\end{equation}

We next follow the proof argument of \cite{DJ18}*{Lemma 4.6} to prove the decay of $A_\cZ$.
\begin{lemma}\label{lem:A_Z_decay}
We have
$$\|A_\cZ u_j\|_{L^2} \rightarrow 0.$$
\end{lemma}

\begin{proof}
Since $0 \leq \alpha \1_\cZ(\w) \leq F(\w) \leq 1$ for all $\w \in \cW(T_0)$, by Lemma ~\ref{lem:ac_ad_bound},
\begin{equation}\label{eq:A_Z_upper_bound}
\alpha \|A_\cZ u_j\|_{L^2} \leq \|A_F u_j\|_{L^2} + \cO(h^{1/8}).
\end{equation}
By ~\eqref{eq:F_def} and~\eqref{eq:A_1A_2},
$$A_F = \frac{1}{T_0} \sum_{k=0}^{T_0 -1} \sum_{\substack{\w \in \cW(T_0)\\ w_k=1}} A_\w = \frac{1}{T_0} \sum_{k=0}^{T_0 -1} (A_1 + A_2)^{T_0 - 1 - k} A_1(k) (A_1 + A_2)^k.$$
As $A_1 +A_2 = I - \psi_0(-h_j^2 \Delta)$, we see that $\|A_1 + A_2\|_{L^2 \rightarrow L^2} \leq 1$.
Thus,
$$\|A_F u_j \|_{L^2}  \leq \max_{0 \leq k < T_0} \|A_1(k) (A_1 +A_2)^k u_j\|_{L^2}.$$
From Lemma \ref{lem:A_bound}, $\|A_1(k)\|_{L^2 \rightarrow L^2} = \|A_1\|_{L^2 \rightarrow L^2} \leq C$. Thus using ~\eqref{eq:A_1A_2_upperbound},
$$\|A_1(k) u_j - A_1(k) (A_1 + A_2)^k u_j\|_{L^2} \leq C \|u_j - (A_1 + A_2)^k u_j\|_{L^2} =0.$$
Therefore by ~\eqref{eq:A1_decay},
\begin{equation}\label{eq:A_F_decay}
\|A_F u_j\|_{L^2} \leq \max_{0 \leq k < T_0} \|A_1(k) u_j \|_{L^2} \leq \|A_1 u_j\|_{L^2} \rightarrow 0.
\end{equation}
The proof then finishes by combining ~\eqref{eq:A_Z_upper_bound} and ~\eqref{eq:A_F_decay}.
\end{proof}

\begin{proof}[Proof of Proposition ~\ref{prop:AY_decay}]

From ~\eqref{eq:def_X} and ~\eqref{eq:A_1A_2}, we see 
$$A_\cY = \sum_{k=1}^8 U(7T_0) (A_{\cW(T_0) \setminus \cZ} U(T_0))^{8-k} A_\cZ((k-1)T_0) (A_1+ A_2)^{(k-1)T_0}.$$
By Lemma ~\ref{lem:ac_bound} and ~\eqref{eq:bounded_property}, we know $\|A_{\cW(T_0) \setminus \cZ}\|_{L^2 \rightarrow L^2}$, $\|A_\cZ\|_{L^2 \rightarrow L^2} \leq C$.
Therefore,
$$\|A_\cY u_j\|_{L^2} \leq C \sum_{k=1}^8 \|A_\cZ((k-1)T_0) (A_1 + A_2)^{(k-1)T_0} u_j\|_{L^2}$$
and 
by ~\eqref{eq:A_1A_2_upperbound}, 
$$\|A_\cZ((k-1)T_0) u_j - A_\cZ((k-1)T_0) (A_1 + A_2)^{(k-1)T_0} u_j\|_{L^2} \leq C\|u_j - (A_1 + A_2)^{(k-1)T_0} u_j\|_{L^2} =0.$$
Therefore, 
\begin{align*}
\|A_\cY u_j\|_{L^2} &\leq C \sum_{k=1}^8 \|A_\cZ ((k-1)T_0) (A_1 + A_2)^{(k-1)T_0}u_j\|_{L^2}\\
&\leq C\|A_\cZ ((k-1)T_0) u_j\|_{L^2} \leq C\|A_\cZ u_j\|_{L^2}.   
\end{align*}
Thus, by Lemma ~\ref{lem:A_Z_decay}, we conclude $\|A_\cY u_j\|_{L^2} \rightarrow 0$. 
\end{proof}

\subsection{Reduction of 
Proposition ~\ref{prop:AX_decay} to Lemma~\ref{lem:A_w_decay}}
\label{subsection:decay_of_X}

We begin by quoting the following lemma, which estimates the size of $\cX$.

\begin{lemma}[~\cite{DJ18}*{Lemma 3.3}]\label{lem:size_X_bound}
The number of elements in $\cX$ is bounded by 
\begin{equation*}\label{eq:X_size}
\# \cX \leq Ch^{-4 \sqrt{\alpha}},
\end{equation*}
where $C$ may depend on $\alpha$.
\end{lemma}

We claim that we can bound $\|A_\w\|_{L^2 \rightarrow L^2}$ in the following way.

\begin{lemma}\label{lem:A_w_decay}
There exist $C, \beta>0$, depending only on $M$, $a_1$, $a_2$, and $\rho$ such that
$$\sup_{\w \in \cW(2T_1)} \|A_\w\|_{L^2 \rightarrow L^2} \leq Ch^{\beta/2}.$$
\end{lemma}

We defer the proof of this lemma until \S\ref{subsection:proof_Aw_decay}. 
Assuming that this lemma holds, we can prove Proposition ~\ref{prop:AX_decay}.

\begin{proof}[Proof of Proposition ~\ref{prop:AX_decay}]
From Lemma ~\ref{lem:size_X_bound} and Lemma ~\ref{lem:A_w_decay}, we have 
$$\|A_\cX\|_{L^2 \rightarrow L^2} \leq \# \cX \left(\sup_{\w \in \cW(2T_1)} \|A_{\w}\|_{L^2 \rightarrow L^2} \right) \leq C h^{\frac{\beta}{2}-4 \sqrt{\alpha}},$$
where $\beta>0$ depends only on $M, a_1, a_2, \rho$ and $C>0$ depends on $M, a_1, a_2, \rho, \alpha$. 
Setting
\begin{equation}\label{eq:alpha_def}
\alpha = \frac{\beta^2}{100},
\end{equation}
we conclude $\|A_\cX\|_{L^2 \rightarrow L^2} \leq Ch^{\frac{\beta}{10}}$.
\end{proof}


\section{Proof of Lemma~\ref{lem:A_w_decay}}\label{section:apply_FUP}


We will apply the higher-dimensional fractal uncertainty principle to prove Lemma~\ref{lem:A_w_decay}. Here we briefly outline our argument.

In \S \ref{subsection:hyperbolic_porosity}, we define a notion of ball and line porosity contingent on the hyperbolic structure of the manifold. We show that, in some sense, the support of $a_\w$ satisfies this porosity. The  porosity of $a_\w$ comes from the propagation of $a_1$ and $a_2$ by $\phi_t$, an Anosov geodesic flow. In \S\ref{subsection:symplectomorphisms}, we construct two symplectomorphisms to ``straighten out" the stable and unstable Lagrangian foliations. In \S\ref{subsection:hyperbolic_FUP}, we employ these symplectomorphisms to show the our hyperbolic notion of porosity gives a fractal uncertainty principle. Finally, in \S\ref{subsection:proof_Aw_decay}, we apply this fractal uncertainty principle to conclude Lemma~\ref{lem:A_w_decay}.

\subsection{Hyperbolic porosity}\label{subsection:hyperbolic_porosity}

For $u =(u_1, \ldots, u_n) \in \R^n$, define the vector field on $F^*M$
$$\cU^\pm u \coloneqq u_1 U_1^\pm + \cdots + u_n U_n^\pm.$$

Similarly, for $v = (v_1, \ldots, v_{n+1}) \in \R^{n+1}$, define
$$\cV^\pm v  \coloneqq v_1 U_1^\pm + \cdots + v_{n} U_{n}^\pm + v_{n+1} X.$$

Using these vector fields, we define analogues on hyperbolic manifolds of ball and line porosity given in Definitions \ref{def:porous_on_balls} and \ref{def:porous_on_lines}. See also Figure~\ref{fig:porous}.

\begin{definition}\label{def:hyperbolic_ball_porosity}
We say that $\Omega \subset S^*M$ is \emph{hyperbolic $(\nu^\pm,  \varepsilon^\mp)$-porous on balls from scales $\alpha_0$ to $\alpha_1$} if for all  $q \in F^*M$ and $\alpha \in [\alpha_0, \alpha_1]$, there exists $u_0 \in \R^n$ with $|u_0| \leq \alpha$ such that 
$$\pi_S\{e^{\cV^{\mp} v} e^{\cU^\pm (u+u_0)} q: u \in \R^n, |u| \leq \nu \alpha, v \in \R^{n+1}, |v| \leq \varepsilon \}\cap \Omega = \emptyset.$$
\end{definition}

\begin{definition}\label{def:hyperbolic_line_porosity}
We say that $\Omega \subset S^*M$ is \emph{hyperbolic $(\nu^\pm, \varepsilon^\mp)$-porous on lines from scales $\alpha_0$ to $\alpha_1$} if for all  $q \in F^*M$ and $\alpha \in [\alpha_0, \alpha_1]$, there exists $t_0 \in [-\alpha, \alpha]$ such that 
$$\pi_S\{(e^{\cV^{\mp} v} e^{\cU^{\pm} u + U_1^\pm t_0} q): u \in \R^n, |u| \leq \nu \alpha, v \in \R^{n+1}, |v| \leq \varepsilon \}\cap \Omega = \emptyset.$$
\end{definition}

\begin{figure}
\centering
\resizebox{6.5in}{2.2in}{
\begin{tikzpicture} 
\begin{scope}[yshift=-2in, xshift=-2in]
\filldraw[color=blue!3!white]  (0, 0) ellipse (4 and 2);
\filldraw[color=blue!10!white] (-2.25, 3.2) rectangle (.25, -1.2);
\filldraw[color=blue!25!white]  (-1, -1.2) ellipse (1.25 and .625);
\filldraw[color=blue!25!white]  (-1, 1) ellipse (1.25 and .625);
\filldraw[color=blue!25!white] (-2.25, 1) rectangle (.25, -1.2);
\filldraw[color=blue!20!white]  (-1, 3.2) ellipse (1.25 and .625);
\draw (0, 0) ellipse (4 and 2);
\filldraw (0,0) circle (0.05cm) node[anchor=east, scale=1.2]{$q$};
\filldraw (-1, 1) circle (0.05cm) node[anchor=east, scale=1.2]{$q_0$};
\draw (0,0) --  (2.89, -1.39);
\draw node[scale=1.2] at (1.5, -.5) {$\alpha$};
\draw (-1, 1) ellipse (1.25 and .625);
\draw (-2.25, 1) -- (-2.25, 3.2);
\draw[<->] (-2.5, 1) -- (-2.5, 3.2);
\draw node[scale=1.5] at (-2.7, 2) {$\varepsilon$};
\draw (.25, 1) -- (.25, 3.2);
\draw[dashed] (-2.25, 1) -- (-2.25, -1.2);
\draw[dashed] (.25, 1) -- (.25, -1.2);
\draw (-1, 3.2) ellipse (1.25 and .625);
\draw[dashed] (-1, -1.2) ellipse (1.25 and .625);
\draw (-1,1) -- (-.1, .56);
\draw node[scale=1.2] at (-.4, 1) {$\nu \alpha$};
\draw[->] (-3, 4.7) -- (6, 2);
\draw node[scale=1.2] at (6.4, 2) {$\cU^+$};
\draw[->] (-3, 4.7) -- (-8, 0);
\draw node[scale=1.2] at (-8.3, -.1) {$\cU^+$};
\draw[->] (-3, 4.7) -- (-3, 7);
\draw node[scale=1.2] at (-3, 7.2) {$\mathcal{V}^-$};

\begin{scope}[xshift=6.3in]

\filldraw[color=blue!10!white] (-2.25, 3.2) rectangle (.25, -1.2);
\filldraw[color=blue!10!white]  (-1, -1.2) ellipse (1.25 and .625);
\filldraw[color=blue!20!white]  (-1, 3.2) ellipse (1.25 and .625);
\draw[color=blue!70!white, ultra thick] (0,0) arc (30:70:2.8);
\draw[color=blue!70!white, ultra thick] (0,0) arc (30:0:5);
\draw node[scale=1.2] at (2.5,-1.5) {$\{e^{U_1^+ s} q: |s| \leq \alpha \}$};
\filldraw (0,0) circle (0.05cm) node[anchor=east, scale=1.2]{$q$};
\filldraw (-1, 1) circle (0.05cm);
\draw node[scale=1.2] at (-1.3, .9) {$q_1$};
\draw (-1, 1) ellipse (1.25 and .625);
\draw (-2.25, 1) -- (-2.25, 3.2);
\draw[<->] (-2.5, 1) -- (-2.5, 3.2);
\draw node[scale=1.5] at (-2.7, 2) {$\varepsilon$};
\draw (.25, 1) -- (.25, 3.2);
\draw (-2.25, 1) -- (-2.25, -1.2);
\draw (.25, 1) -- (.25, -1.2);
\draw (-1, 3.2) ellipse (1.25 and .625);
\draw[dashed] (-1, -1.2) ellipse (1.25 and .625);
\draw (-2.25, -1.2) arc (-180:0:1.25 and .625);
\draw (-1,1) -- (.2, .8);
\draw node[scale=1.2] at (-.4, 1.1) {$\nu \alpha$};
\draw[->] (-3, 4.7) -- (6, 2);
\draw node[scale=1.2] at (6.4, 2) {$\cU^+$};
\draw[->] (-3, 4.7) -- (-8, 0);
\draw node[scale=1.2] at (-8.3, -.1) {$\cU^+$};
\draw[->] (-3, 4.7) -- (-3, 7);
\draw node[scale=1.2] at (-3, 7.2) {$\mathcal{V}^-$};
\end{scope}
\end{scope}
\end{tikzpicture}}
    \caption{The left image represents the set from Definition~\ref{def:hyperbolic_ball_porosity}: $\{e^{\cV^{-} v} e^{\cU^+ u} q_0: u \in \R^n, |u| \leq \nu \alpha, v \in \R^{n+1}, |v| \leq \varepsilon \}$, where $q_0 = e^{\cU^+ u_0} q$. The right image represents the set from Definition~\ref{def:hyperbolic_line_porosity}: $\{e^{\cV^{-} v} e^{\cU^+ u} q_1: u \in \R^n, |u| \leq \nu \alpha, v \in \R^{n+1}, |v| \leq \varepsilon \}$, where $q_1 = e^{U^+_1 t_0} q$. The picture is to give intuition only; as $U^+_i$, $U_i^-$, $X$ do not commute, they do not give a coordinate system.} 
    \label{fig:porous}
\end{figure}

Let $\w \in \cW(2T_1)$.
We decompose the word $\w$ into two words of length $T_1$ in order to construct two functions. The support of these functions will be either hyperbolic porous on lines or hyperbolic porous on balls. Write 
$$\w =\w_+ \w_-, \quad \w_\pm \in \cW(T_1).$$ 
Relabel $\w_+$ and $\w_-$ as
$$\w_+ = w^+_{T_1} \ldots w_1^+ \quad \text{and} \quad \w_- = w_0^- \ldots w_{T_1-1}^-.$$
Now set 
\begin{equation}\label{eq:a_pm_def}
a_+ \coloneqq \prod_{k=1}^{T_1} a_{w_k^+} \circ \phi_{-k} \quad \text{and} \quad a_- \coloneqq \prod_{k=0}^{T_1-1} a_{w_k^-} \circ \phi_k.
\end{equation}

From \eqref{eq:abchi}, recall that $a_j = b\chi_j$, where $\chi_j$ are homogeneous functions of degree 0 and $b$ depends only on $|\xi|_g$ with 
$\{\tfrac{1}{2} \leq |\xi|_g \leq 2\} \subset \supp b \subset \{\tfrac{1}{4} \leq |\xi|_g \leq 4\}$.
By  the homogeneity of $\phi_t$, we have 
$$a_+ = b^{T_1-1}\prod_{k=1}^{T_1} \left(\chi_{w_k^+} \circ \phi_{-k}\right) \quad \text{and} \quad a_- = b^{T_1-1} \prod_{k=0}^{T_1-1} \left(\chi_{w_k^-} \circ \phi_k\right),$$

Therefore 
\begin{equation}\label{eq:mathcal_A_def}
\begin{split}
\supp a_+ &\subset \cA_+ \coloneqq \{\tfrac{1}{4} \leq |\xi|_g \leq 4\} \cap \bigcap_{k=1}^{T_1} \phi_k (\supp \chi_{w_k^+})\\
\supp a_- &\subset \cA_- \coloneqq \{\tfrac{1}{4} \leq |\xi|_g \leq 4\} \cap \bigcap_{k=0}^{T_1-1} \phi_{-k} (\supp \chi_{w_k^-}).
\end{split}
\end{equation}
Clearly, $\cA_\pm$ are invariant under scalings of $\xi$ within the compact set $\{\tfrac{1}{4} \leq |\xi|_g \leq 4\}$. We will use this fact and use $\cA_\pm$ in lieu of $\supp a_\pm$  in our argument.

Recalling the definitions of $\rho$ and $T_1$ from ~\eqref{eq:rho_def} and ~\eqref{eq:T_def}, we also know the following.

\begin{lemma}\label{lem:a_pm_symbol_classes}
We have $a_+ \in S_{L_u, \rho}^{\comp}(T^*M \setminus 0)$ and   $a_- \in S_{L_s, \rho}^{\comp}(T^*M \setminus 0)$ with bounds on the semi-norms that do not depend on $\w$. 
Moreover,
\begin{align*}\label{eq:b+-}
\begin{split}
A_{\w_+}(-T_1) &= \op_h^{L_u}(a_+) + \cO\left(h^{1- \rho -}\right)_{L^2 \rightarrow L^2},\\
A_{\w_-} &= \op_h^{L_s}(a_-) + \cO\left(h^{1- \rho - }\right)_{L^2 \rightarrow L^2},
\end{split}
\end{align*}
where the constants in $\cO(\cdot)$ are uniform in $\w$.
\end{lemma}

\begin{proof}
From Lemma ~\ref{lem:a_w_symbol_class}, this lemma holds for $a_-$ and $A_{\w_-}$. For $a_+$ and $A_{\w_+}$, we reverse the flow $\phi_t$, which exchanges the stable and unstable foliations.
\end{proof}

\subsubsection{Ball porosity} \label{subsection:ball_porous}
In this subsection, we show that $\supp a_+ \cap S^*M$ is hyperbolic porous on balls. This result exploits the ergodicity of the geodesic flow.

Utilizing the Mautner phenomenon \cite{Mautner}, Moore \cite{Moore} showed that the geodesic flow on $S^*M$ is \emph{strongly mixing}, i.e. for $f, g \in C^\infty(S^*M)$, 
\begin{equation}\label{eq:mixing}
\lim_{t \rightarrow \infty} \int_{S^*M} (f \circ \phi_{t}) g  d \mu_L = \left(\int_{S^* M} f d \mu_L\right) \left( \int_{S^*M} g d\mu_L \right),
\end{equation}
where $\mu_L$ is the Liouville measure.
This was also proved by Anosov in~\cite{An67} in a more general setting. 

We now prove density of translates of horocyclic segments. 

\begin{lemma}\label{lem:ergodicity}
For each open nonempty set $U \subset S^*M$, there exists some $T>0$ such that for all $q \in F^*M$ and $t \geq T$, we have
$$\phi_{-t} \pi_S \{e^{\cU^+ u}q :u \in \R^n, |u| \leq 1 \} \cap U \neq \emptyset.$$
\end{lemma}

\begin{proof}
Let $U \subset S^*M$ be an open set and fix $q_0 \in F^*M$.

Let $\chi \in C^\infty(S^*M)$ be a cutoff function such that $\supp \chi \subset U$ and $\int_{S^*M} \chi d\mu_L >0$. Pick $\varepsilon>0$ sufficiently small so that if $q \in \pi_S^{-1}(\supp \chi)$, then 
$\pi_S\{e^{\cV^- v} q : v \in \R^{n+1}, |v| \leq \varepsilon\} \subset U$. 

Let $f=f_{q_0} \in C^\infty(S^*M)$ such that $\int_{S^*M} f d\mu_L >0$ and
$$\supp f \subset \pi_S \left\{e^{\cV^- v} e^{\cU^+ u} q_0:  u \in \R^n, |u| \leq 1,  v \in \R^{n+1}, |v| \leq \varepsilon \right\}.$$

By ~\eqref{eq:mixing}, 
$\lim_{t \rightarrow \infty} \int_{S^*M} (f \circ \phi_{t}) \chi d \mu_L = ( \int_{S^*M} f d\mu_L ) ( \int_{S^*M} \chi d\mu_L )>0.$
Therefore, there exists $T_{q_0}>0$, depending on $q_0$ and $U$, such that for all $t \geq T_{q_0}$, $\phi_{-t}(\supp f) \cap \supp \chi \neq \emptyset$.  Fix $t \geq T_{q_0}$.

We now examine $\phi_{-t}(\supp f)$. From ~\eqref{eq:commutator}, we know 
\begin{equation}\label{eq:commute_constant}
\phi_{-t}( e^{U_i^\pm}q) = e^{e^{ \pm t} U_i^\pm}\phi_{-t} (q).
\end{equation}

Thus,
\begin{align*}
\phi_{-t}(\supp f) & \subset \phi_{-t}\pi_S \left\{e^{\cV^- v} e^{\cU^+ u} q_0: u \in \R^n,  |u| \leq 1,  v \in \R^{n+1}, |v| \leq \varepsilon \right\}\\
& =\pi_S \left\{ \phi_{-t} (e^{\cV^- v} e^{\cU^+ u} q_0): u \in \R^n, |u| \leq 1,  v \in \R^{n+1}, |v| \leq \varepsilon  \right\}\\
& \subset \pi_S \left\{ e^{\cV^- v} e^{\cU^+ u} \phi_{-t} (q_0):  u \in \R^n, |u| \leq e^t,  v \in \R^{n+1}, |v| \leq \varepsilon \right\}.
\end{align*}

Therefore, for some $|u_0|  \leq e^t$ and $|v_0| \leq \varepsilon$, 
$e^{\cV^- v_0} e^{\cU^+u_0} \phi_{-t}( q_0 )\in \pi_S^{-1} \supp \chi$. By the choice of $\varepsilon$, $\pi_S (e^{\cU^+u_0} \phi_{-t}( q_0)) \in U$.

From \eqref{eq:commute_constant}, $\pi_S(e^{\cU^+u_0} \phi_{-t}( q_0)) = \phi_{-t} (\pi_S( e^{\cU^+ e^{-t} u_0} q_0))$. Clearly, $|e^{-t} u_0| \leq 1$.

Finally, note that by the compactness of $F^*M$, there exists some $T>0$ which depends on $U$ such that $T \geq T_{q}$ for all $q\in F^*M$.  
\end{proof}

\begin{remark}
One can make stronger statements about the subset
$$\phi_{-t} \pi_S \{e^{\cU^+ u}q :u \in \R^n, |u| \leq 1 \},$$
of $S^*M$ appearing in Lemma \ref{lem:ergodicity}.
Indeed as $T\to\infty$, this subset equidistributes to $\mu_L$ by work of Shah \cite{ShahExpanding}.
This is moreover known with an effective rate, as shown by Kleinbock and Margulis \cite{KleinbockMargulis}.
\end{remark}

We now prove the hyperbolic ball porosity of $\cA_+ \cap S^*M$. The following proof takes inspiration from ~\cite{DJ18}*{Lemma 5.10}.

\begin{lemma}\label{lem:a_hyperbolic_ball_porosity}
There exists $\nu_1, \varepsilon_0, K_1>0$, where  $\nu_1 = e^{-T-1}\varepsilon_0$ for some $T>1$, depending only on $M, a_1, a_2$, such that $\cA_+ \cap S^*M$ is 
hyperbolic $(\nu_1^+, \varepsilon_0^-)$-porous on balls from scales $K_1 h^\rho$ to $1$
in the sense of Definition ~\ref{def:hyperbolic_ball_porosity}.
\end{lemma}

\begin{proof}
From \eqref{eq:supp_a1_a2}, there exist $U_1, U_2 \subset S^*M$ such that for $w=1,2$, $U_w \Subset S^*M \setminus \supp a_w$. Therefore, $U_w \Subset S^*M \setminus  \supp \chi_w$.

From Lemma ~\ref{lem:ergodicity}, there exists $T>1$ such that for each $q  \in F^*M$ and  some $u_w = u_w(q)$ with $|u_w| \leq 1$, 
\begin{equation}\label{eq:propagate_forward}
\pi_S(\phi_{-T}( e^{\cU^+ u_w} (q))) \in U_w.
\end{equation}

Set $K_1 \coloneqq e^{2+ T}$. Fix $q_0 \in F^*M$ and $\alpha \in [K_1 h^\rho, 1]$. Let $t$ be the unique integer such that $e^{-t} < \alpha \leq e^{-t +1}$. From ~\eqref{eq:T_def}, $1 < T+t \leq T_1 -1$. Set $w= w_{t+T}^+ \in \{1,2\}$ and note that $\cA_+ \cap \phi_{t+T} (S^*M \setminus \supp \chi_w) = \emptyset$.

Fix $u_0 \coloneqq u_w(\phi_{-t}(q_0))$. From ~\eqref{eq:commutator} and ~\eqref{eq:propagate_forward},
\begin{equation}\label{eq:propagate_into_U}
\pi_S (\phi_{-T-t}(e^{ \cU^+ e^{-t} u_0}q_0) )=  \pi_S (\phi_{-T}(e^{\cU^+ u_0}\phi_{-t}(q_0))) \in U_w.
\end{equation}

Choose $\varepsilon_0 >0$ sufficiently small so that for $w=1,2$, if $u \in  \R^n$, $v \in \R^{n+1}$ satisfy $|u|, |v| \leq \varepsilon_0$ and $q \in \pi^{-1}_S(U_w)$, then 
$$\pi_S (e^{\cV^- v} e^{\cU^+u} q )\subset S^*M \setminus \supp \chi_w.$$

Then from ~\eqref{eq:propagate_into_U},  we know 
\begin{equation}\label{eq:ball_epsilon}
\{e^{\cV^- v} e^{\cU^+ u}( \phi_{-T -t}( e^{\cU^+ e^{-t} u_0} q_0)) : |u|, |v| \leq \varepsilon_0\} \subset \pi_S^{-1}(S^*M \setminus \supp \chi_w).
\end{equation}

From  ~\eqref{eq:commutator}, for any $q \in F^*M$,
\begin{equation}\label{eq:ball_propagate}
e^{\cV^- v} e^{\cU^+  u} \phi_t(  q)=\phi_t (e^{\cV^- v'} e^{\cU^+ e^t u} q) ,
\end{equation}
where $v=(v_1, \ldots, v_{n+1})$ and $v'= (e^{-t} v_1, \ldots, e^{-t} v_n, v_{n+1})$. 

Thus from \eqref{eq:ball_epsilon} and \eqref{eq:ball_propagate},
$$\pi_S \left\{\phi_{-t -T}(e^{\cV^-v} e^{\cU^+( e^{-T-t} u +e^{-t} u_0)} q_0) : |u|, |v| \leq \varepsilon_0 \right\} \subset S^*M \setminus \supp \chi_w.$$
Setting $\nu_1 \coloneqq e^{-T-1}\varepsilon_0$, we see that
$$\pi_S \left\{e^{\cV^- v} e^{\cU^+ (u + e^{-t} u_0)} q_0: |u| \leq \nu_1 \alpha, |v| \leq \varepsilon_0 \right\}  \subset \phi_{t+T} \left(S^*M \setminus \supp a_{w} \right) \subset S^*M \setminus \cA_+,$$
with $|e^{-t} u_0| \leq \alpha$.
\end{proof}

\subsubsection{Line porosity} \label{subsection:line porous}
We show that $\cA_- \cap S^*M$ is hyperbolic porous on lines. This result uses our assumption that $U_\mu$ is $U_1^-$ dense. Similarly to Lemma~\ref{lem:a_hyperbolic_ball_porosity}, the proof is adapted from ~\cite{DJ18}*{Lemma 5.10}.

\begin{lemma}\label{lem:a_hyperbolic_line_porosity}
There exists $\nu_1, \varepsilon_0, K_1 >0$, with $\nu_1 = \varepsilon_0/T$ for some $T>1$, depending only on $M, a_1, a_2$ such that $\cA_- \cap S^*M$ is hyperbolic $(\nu_1^-, \varepsilon_0^+)$-porous on lines from scales $K_1h^\rho$ to $1$ in the sense of Definition ~\ref{def:hyperbolic_line_porosity}.
\end{lemma}

\begin{proof}
For $w =1, 2$, from \eqref{eq:supp_a1_a2}, there exists a $U_1^-$-dense set $U_w \subset S^*M$ such that 
$U_w \Subset S^*M \setminus \supp a_w$. Therefore, $U_w \Subset S^*M \setminus  \supp \chi_w$.

Using Lemma ~\ref{lem:finite_safe}, there exists $T>1$ depending only on $M, a_1$, and  $a_2$ such that for each $q \in F^*M$ and some $t_w = t_w(q) \in [-T, T]$,  
\begin{equation}\label{eq:safe_consequence}
\pi_S(e^{t_w U_1^-}q) \in U_w.
\end{equation}

Set $K_1 \coloneqq 3T$. Fix $q_0 \in F^*M$ and $\alpha \in [K_1 h^\rho, 1]$.  Let $j$ be the unique integer such that $e^{j-1} \alpha < T \leq e^j \alpha$. From ~\eqref{eq:T_def},  $1 \leq j \leq T_1-1$.

Set $t_0 \coloneqq e^{-j} t_w(\phi_{j}(q_0))$. As $e^j \alpha \geq T$, we know $t_0 \in [-\alpha, \alpha]$.

By ~\eqref{eq:commutator} and ~\eqref{eq:safe_consequence}, 
\begin{equation}\label{eq:propogate_forward_ball}
\pi_S (\phi_j( e^{t_0 U_1^-}q_0)) =  \pi_S (e^{e^j t_0 U_1^-}\phi_j( q_0))\in U_w.
\end{equation}

Choose $\varepsilon_0 >0$ sufficiently small so that for $w=1,2$, if $|u|, |v| \leq \varepsilon_0$ and $q \in \pi_S^{-1}(U_w)$, then
$$\pi_S (e^{\cV^+ v} e^{\cU^- u} q) \subset S^*M \setminus \supp \chi_w.$$
Then from ~\eqref{eq:propogate_forward_ball},we know
$$\pi_S \left\{e^{\cV^+ v} e^{\cU^- u} \phi_j( e^{t_0 U_1^-} q_0 ): |u|, |v| \leq \varepsilon_0\right\} \subset S^*M \setminus \supp \chi_w.$$
From ~\eqref{eq:commutator}, for all $q \in F^*M$
$$e^{\cV^+ v}e^{\cU^- u} \phi_j( q)= \phi_j( e^{\cV^+ v'} e^{\cU^- e^{-j}u}q),$$
where $v = (v_1, \ldots, v_{n+1})$ and $v' = (e^{j}v_1, \ldots, e^{j} v_n, v_{n+1})$.

Setting $\nu_1 \coloneqq \varepsilon_0/T$, we see that
$$\pi_S\{e^{\cV^+ v} e^{\cU^- u} e^{t_0 U_1^-} q_0 : |u| \leq \nu_1 \alpha, |v| < \varepsilon_0\} \subset \phi_{-j} ( S^*M \setminus \supp a_w) \subset S^*M \setminus \cA_-,$$
which completes the proof.
\end{proof}

\subsection{Symplectomorphisms on hyperbolic space}\label{subsection:symplectomorphisms}

We follow the exposition of ~\cite{DZ16}*{\S 4.4}.
Define the maps
\begin{equation} \label{eq:B_def}
B_{\pm}: T^* \bH^{n+1} \backslash 0 \rightarrow \bS^{n}
\end{equation}
as follows: for $(x, \xi) \in T^* \bH^{n+1}, B_\pm(x, \xi)$ is the limit of the projection to the ball model of $\bH^{n+1}$ of the geodesic $e^{tX}(x, \xi)$ as $t \rightarrow \pm \infty$. 
For a more in-depth exposition, see ~\cite{DFG15}*{\S 3.4}.

Then the lifts of the weak stable/unstable Lagrangian foliations $L_s/L_u$ (defined in  \eqref{eq:LsLu_def}) to $T^*\bH^{n+1} \setminus 0$ are given by
\begin{equation*}\label{eq:L_lifts}
\begin{split}
\pi^*_\Gamma L_s (x, \xi) &= \ker dB_+(x, \xi) \cap \ker dp(x, \xi),\\
\pi^*_\Gamma L_u (x, \xi) &= \ker dB_-(x, \xi) \cap \ker dp(x, \xi),
\end{split}
\end{equation*}
where we recall $p$ from \eqref{eq:p_def}.

We construct the symplectomorphisms 
\begin{equation*}\label{eq:kappa_range_domain}
\kappa^\pm : T^*\bH^{n+1}\setminus 0 \rightarrow T^*(\R^+_w \times \bS^{n}_y)
\end{equation*}
which map $L_s$ and $L_u$, respectively, to the vertical foliation on $T^*(\R^{+} \times \bS^{n})$:
\begin{equation*}\label{eq:kappa_def}
(\kappa^+)_* L_u = (\kappa^-)_*L_s =L_V \coloneqq \ker (dw) \cap \ker(dy).
\end{equation*}

The use of two symplectomorphisms is necessary; we cannot use a single diffeomorphism to simultaneously straighten out both $L_s$ and $L_u$.

Set 
\begin{equation*}\label{eq:G_def}
\cG(y, y') = \frac{y' - (y\cdot y') y}{1- y \cdot y'} \in \R^{n+1}, \quad y, y' \in \bS^{n} \subset \R^{n+1}, \quad y \neq y',
\end{equation*}
which is half the stereographic projection of $y'$ with the base point $y$. We have $\cG(y, y') \perp y$, thus we think of $\cG(y, y')$ as a vector in $T_y \bS^{n}$. Using the round metric on the sphere, we can also think of  $\cG(y, y')$ as a vector in $T_y^* \bS^{n}$.
For $(x, \xi) \in T^* \bH^{n+1} \setminus 0$, set
$$G_{ \pm}(x, \xi)= p(x, \xi) \cG(B_{\pm}(x, \xi), B_{\mp}(x, \xi)) \in T_{B_{\pm}(x, \xi)}^* \bS^{n}.$$
Denote by $\cP(x, y)$ the following function defined on the ball model of $\bH^{n+1}$ by 
\begin{equation*}\label{eq:Poisson_kernel_def}
\cP(x, y)=\frac{1-|x|^2}{|x-y|^2}, \quad x \in \bH^{n+1}, \quad y \in \bS^{n}.
\end{equation*}

We then construct the symplectomorphisms $\kappa^\pm$ in the following way:

\begin{lemma}[~\cite{DZ16}*{Lemma 4.7}]\label{lem:kappa_pm_def}
Let $\kappa^{\pm}: T^*\bH^{n+1} \setminus 0 \rightarrow T^*(\R^+ \times \bS^{n})$ be given by
\begin{equation}\label{eq:kappa_explicit_def}
\kappa^{ \pm}: (x, \xi) \mapsto \left(p(x, \xi), B_{\mp}(x, \xi), \pm \log \cP(x, B_{\mp}(x, \xi)), \pm G_{\mp}(x, \xi)\right). 
\end{equation}
Then $\kappa^{\pm}$ are exact symplectomorphisms. 
\end{lemma}

To explain, for $\kappa^\pm(x, \xi)=(w, y, \theta, \eta)$:
\begin{itemize}
\item $y, \eta$ determine the geodesic $\gamma(t)=e^{t X}(x, \xi)$ up to shifting $t$ and rescaling $\xi$. In particular, $y$ gives the limit of the geodesic $\gamma(t)$ as $t \rightarrow \mp \infty$;
\item $w$ is the length of $\xi$, corresponding to the energy of the geodesic $\gamma(t)$;
\item $\theta$ satisfies $\theta(\gamma(t))=\theta(\gamma(0))-t$ and thus determines the position of $(x, \xi)$ on the geodesic $\gamma(t)$.
\end{itemize}

We will later use the symplectomorphism 
\begin{equation*} \label{eq:kappa_hat}
\hat{\kappa} \coloneqq \kappa^+ \circ (\kappa^-)^{-1} : T^*(\R^+ \times \bS^{n}) \rightarrow T^*(\R^+ \times \bS^{n}).
\end{equation*}

We cite the following lemma, which characterizes the Fourier integral operators associated to $\hat{\kappa}^{-1}$.

\begin{lemma}~\cite{DZ16}*{Lemma 4.9}\label{lem:FIO_to_PsiDO}
Assume that $B \in I^{\comp}_h(\hat{\kappa}^{-1})$. Then we have
$$B = A \widetilde{B}_\chi + \cO(h^\infty)_{L^2 \rightarrow L^2}$$
for some $A \in \Psi_h^{\comp}(\R^{+} \times \bS^{n})$, $\chi \in C_c^{\infty}(\bS_{\Delta}^{n})$, and
$$\widetilde{B}_\chi v(w, y)=(2 \pi h)^{-\frac{n}{2}} \int_{\bS^{n}}\left|\frac{y-y'}{2}\right|^{2 i \omega / h} \chi(y, y') v(w, y') d y',$$
where $\bS_\Delta^{n}=\{(y, y') \in \bS^{n} \times \bS^{n}: y \neq y'\}$, $|y-y'|$ denotes the Euclidean distance, and $dy'$ is the standard volume form on the sphere.
\end{lemma}

\subsection{Hyperbolic fractal uncertainty principle}\label{subsection:hyperbolic_FUP}
The goal of this subsection is to prove Proposition \ref{prop:hyperbolic_FUP}, stated following this paragraph. Intuitively, this proposition tell us that the hyperbolic  $(\nu_1^+, \varepsilon_0^-)$-porosity on balls of $\cA_+$ from Lemma~\ref{lem:a_hyperbolic_ball_porosity} and the hyperbolic  $(\nu_1^-, \varepsilon_0^+)$-porosity on lines of $\cA_-$ from Lemma~\ref{lem:a_hyperbolic_line_porosity} give a fractal uncertainty principle.
The statement and the proof of Proposition \ref{prop:hyperbolic_FUP} are adapted to higher dimensions and to manifolds from ~\cite{DJ18}*{Proposition 5.7}.  Note that  Lemma~\ref{lem:a_hyperbolic_ball_porosity} and  Lemma~\ref{lem:a_hyperbolic_line_porosity} give two sets of values for $\nu_1$, $\varepsilon_0$, and $K_1$. From now on, we assume that  $\nu_1$ and $\varepsilon_0$ are each equal to the minimum of their two values and $K_1$ is equal to the maximum of its two values. We can write $\varepsilon_0 = C \nu_1$, where $C$ depends only on $M, a_1, a_2$.

\begin{proposition}\label{prop:hyperbolic_FUP}
Recall $a_\pm$ from \eqref{eq:a_pm_def}. There exists $\beta>0$  depending only on $M$, $a_1$, $a_2$, and $\rho$ such that 
for all $Q \in \Psi_h^0(M)$
\begin{equation}\label{eq:Q_bound}
\left\|\op_h^{L_s}(a_-) Q \op_h^{L_u}(a_+)\right\|_{L^2(M) \rightarrow L^2(M)} \leq Ch^{\beta/2},
\end{equation}
where $C$ depends only on $M, a_1, a_2, Q$, and $\rho$.
\end{proposition}

For some $(x_0, \xi_0) \in S^*M$ and $C_1$ selected to be sufficiently small in  Lemma ~\ref{lem:a_tilde_porous_ball} and Lemma ~\ref{lem:a_tilde_porous_line}, define
\begin{equation}\label{eq:V_diam}
V \coloneqq \left\{(x, \xi) \in T^*M \setminus 0: \tfrac{1}{4} \leq |\xi|_g \leq 4,  (x, \xi) \in B_{\tfrac{\nu_1}{C_1^2}}(x_0, |\xi|_g\xi_0)\right\}.
\end{equation}
By a microlocal partition of unity and since $a_\pm$ is supported in $\{\tfrac{1}{4} \leq |\xi|_g \leq 4\}$, we assume $\WF_h(Q) \subset V$.

Now define the following neighborhoods of $V$
$$V' \coloneqq \left\{(x, \xi) \in T^*M \setminus 0: \tfrac{1}{4}- \tfrac{\nu_1}{C_1^2} \leq |\xi|_g \leq 4 +\tfrac{\nu_1}{C_1^2}, (x, \xi) \in B_{\tfrac{2\nu_1}{C_1^2}}(x_0, |\xi|_g\xi_0) \right\},$$
$$V'' \coloneqq \left\{(x, \xi) \in T^*M \setminus 0: \tfrac{1}{4}- \tfrac{2\nu_1}{C_1^2} \leq |\xi|_g \leq 4 +\tfrac{2\nu_1}{C_1^2}, (x, \xi) \in B_{\tfrac{3\nu_1}{C_1^2}}(x_0, |\xi|_g\xi_0) \right\}.$$

We compose the maps $\kappa^\pm$ (given in ~\eqref{eq:kappa_explicit_def}) with a local inverse of the covering map $\pi_\Gamma : T^* \bH^{n+1} \rightarrow T^*M$ defined in ~\eqref{eq:pi_gamma_def} to obtain exact symplectomorphisms
$$\kappa_0^\pm : V \rightarrow T^*(\R^+_w \times \bS^n_y).$$

We can assume that $\kappa_0^\pm(V)$ is contained in a compact subset of $T^*(\R^+_w \times \bS^{n}_y)$ that depends only on $M$.

We choose operators
$$\cB_\pm \in I^{\comp}_h(\kappa^\pm_0), \quad \cB'_\pm \in I^{\comp}_h((\kappa_0^\pm)^{-1})$$
which quantize $\kappa_0^\pm$ near $\kappa^\pm_0(\WF_h(Q)) \times \WF_h(Q))$ in the sense of ~\eqref{eq:quantize_kappa}. We  use these operators to conjugate $\op_h^{L_s}(a_-)$ and $\op_h^{L_u}(a_+)$ to operators on $\R^+ \times \bS^n$. Define
$$A_- \coloneqq \cB_- \op_h^{L_s}(a_-) \cB'_-, \quad A_+ \coloneqq \cB_+ Q \op_h^{L_u}(a_+) \cB'_+, \quad B \coloneqq \cB_-\cB_+'.$$
Note that $B \in I_h^{\comp}(\kappa^- \circ (\kappa^+)^{-1})$.
We have
$$\op_h^{L_s}(a_-) Q \op_h^{L_u}(a_+) = \cB'_- A_-B A_+\cB_+ + \cO(h^\infty)_{L^2 \rightarrow L^2}.$$

By ~\eqref{eq:FIO_conjugation}, there exists $\tilde{a}_\pm \in S^{\comp}_{L_V, \rho}(T^*(\R^+_w \times \bS^n_y))$ such that
\begin{equation*}\label{eq:supp_a_tilde}
A_\pm  = \op_h^{L_V}(\tilde{a}_\pm) + \cO(h^\infty)_{L^2 \rightarrow L^2}, \quad \supp \tilde{a}_\pm \subset \kappa_0^\pm(V \cap \supp a_\pm).
\end{equation*}

Then ~\eqref{eq:Q_bound} follows from showing
\begin{equation}\label{eq:B_bound}
\|\op_h^{L_V}(\tilde{a}_-) B \op_h^{L_V}(\tilde{a}_+)\|_{L^2(\R^+ \times \bS^n) \rightarrow L^2(\R^+ \times \bS^n)} \leq Ch^{\beta/2}. 
\end{equation}

Recall the definition of $\cA_\pm$ from \eqref{eq:mathcal_A_def}. Then for 
\begin{equation}\label{eq:A_tilde}
\tilde{\cA}_\pm \coloneqq \kappa_0^\pm(V \cap \cA_\pm),
\end{equation}
we know  $\supp \tilde{a}_\pm \subset \tilde{\cA}_\pm$.

Recall the definition of  $B_-(x, \xi)$ from ~\eqref{eq:B_def}. Then since $U_i^\pm$ generate one-parameter unipotent flows,
\begin{equation}\label{eq:B-_change}
d(B_\pm \circ \pi_S)\cdot U_i^\pm =0, \quad d(B_\pm \circ \pi_S) \cdot X =0.
\end{equation}
We also define the family of functions for $\lambda>0$: 
\begin{equation}\label{eq:f_lambda}
f_\lambda: T^*M\setminus 0 \rightarrow T^*M \setminus 0, \quad f_\lambda(x, \xi) = (x, \lambda \xi).
\end{equation}
For the remainder of the paper, balls and distances on $\bS^n$ are given by geodesics under the round metric, unless otherwise noted.

We adapt following lemma  from ~\cite{DJ18}*{Lemma 5.8}.

\begin{lemma}\label{lem:a_tilde_porous_ball}
There exists a constant $C_1>0$ depending only on $M$, $a_1$, $a_2$ such that the following holds. Define the projection of $\tilde{\cA}_+$ onto the $y$-variables
$$\Omega_+ \coloneqq \left\{y \in \bS^n : \exists w, \theta, \eta \text{ such that } (w, y, \theta, \eta) \in \tilde{\cA}_+\right\} \subset \bS^n,$$
and set 
$\nu \coloneqq \nu_1^2/ C_1^3$. 
Then for all balls $B \subset \bS^n$ of diameter $R \in [C_1 K_1 h^\rho/ \nu_1, 1]$, there exists $y_1 \in B$ such that $B_{\nu R}(y_1) \cap \Omega_+ =\emptyset$.
\end{lemma}

\begin{proof}
Set $C_1>0$ to be sufficiently large, as specified later in the proof. $C_1$ will depend only on $M$, $a_1$, $a_2$.
Define $W \coloneqq \kappa_0^+(V)$. We lift $V''$ to $T^*\bH^{n+1} \setminus 0$ and use $\kappa^+$ to extend $\kappa_0^+$ to a symplectomorphism
$$\kappa_0^+: V' \rightarrow W', \quad V'' \rightarrow W''$$
for open sets $W', W'' \subset T^*(\R^+_w \times \bS_y^n)$. Define 
$$W''_\bS \coloneqq \kappa_0^+(V'' \cap S^*M).$$

For $C_1$ to be sufficiently large, $\diam W''_\bS \leq \frac{C_1}{10}\diam(V'' \cap S^*M)$. Then by ~\eqref{eq:V_diam},
\begin{equation}\label{eq:diam_W''_+}
\diam(W''_\bS) \leq \frac{\nu_1}{C_1}.
\end{equation}
Again, select $C_1$ to be sufficiently large so that the $\nu_1/C_1^3$-neighborhoods of $W, W'$ are contained respectively in $W', W''$. 

Let $B \subset \bS^n$ be a ball of diameter $R \in [C_1 K_1 h^\rho/ \nu_1, 1]$, centered at some $y_0 \in \bS^n$. 

Assume first that the $y$-projection of $W'$ does not contain $y_0$. By ~\eqref{eq:A_tilde}, $\tilde{\cA}_+ \subset W$. Then the distance between $y_0$ and $\Omega_+$ is at least $\nu_1/C_1^3$. Therefore, the ball of radius $\nu R$ centered at $y_0$ does not intersect $\Omega_+$.

Now assume the $y$-projection of $W'$ does contain $y_0$. Therefore, we can choose $w_0, \theta_0, \eta_0$ such that $(w_0, y_0, \theta_0, \eta_0) \in W'$ and 
$$(x_0, \xi_0) \coloneqq (\kappa_0^+)^{-1}(w_0, y_0, \theta_0, \eta_0) \in S^*M.$$
Choose $\Xi_0 \coloneqq (\xi_2, \ldots, \xi_{n+1})$ such that $(x_0, \xi_0, \Xi_0) \in F^*M$. Define
$$\alpha \coloneqq \frac{\nu_1 R}{C_1},$$
and note $K_1 h^\rho \leq \alpha \leq  1$.

By Lemma \ref{lem:a_hyperbolic_ball_porosity}, we know $\cA_+ \cap S^*M$ is hyperbolic $(\nu_1^+, \varepsilon_0^-)$-porous on balls from scales $K_1 h^\rho$ to $1$. Thus, there exists $u_0 \in \R^n$, $|u_0| \leq \alpha$ such that
\begin{equation*}
\{\pi_S(e^{\cV^-v} e^{\cU^+(u +u_0)}(x_0, \xi_0, \Xi_0)) : |u| \leq \nu_1 \alpha, |v| \leq \varepsilon_0\} \cap \cA_+ = \emptyset.
\end{equation*}
As  $\cA_+$ is homogeneous inside $\{\tfrac{1}{4} \leq |\xi|_g \leq 4\}$,
\begin{equation}\label{eq:a+_control}
\{f_\lambda(\pi_S(e^{\cV^-v} e^{\cU^+(u +u_0)}(x_0, \xi_0, \Xi_0)) ): |u| \leq \nu_1 \alpha, |v| \leq \varepsilon_0, \tfrac{1}{4} \leq \lambda \leq 4\} \cap \cA_+ = \emptyset,
\end{equation}
where we recall $f_\lambda$ from \eqref{eq:f_lambda}.
Set
$$(x_1, \xi_1, \Xi_1) \coloneqq e^{\cU^+ u_0} (x_0, \xi_0, \Xi_0), \quad  (x_1, \xi_1) \in V''.$$

By \eqref{eq:tangent_space_decomp} and \eqref{eq:preimage_Es_Eu}, for $C_1$ sufficiently large, we have a diffeomorphism 
$$\Theta: \tilde{U} \times \left[-\tfrac{1}{4} -\tfrac{2\nu_1}{C_1^2}, 4 + \tfrac{2\nu_1}{C_1^2}\right] \rightarrow W'', \quad (u, v, \lambda) \mapsto \kappa_0^+(f_{\lambda}(\pi_S( e^{\cV^-v} e^{ \cU^+ u} (x_0, \xi_0, \Xi_0)))),$$
where $\tilde{U}$ is a neighborhood of $(0, 0) \in \R^n \times  \R^{n+1}$. 
From ~\eqref{eq:B-_change},
the value of $y$ does not change if we change $v$. As $B_-$ is invariant under rescaling $\xi$, the value of $y$ does not change with $\lambda$. Thus, the $y$ component of $\Theta$ is equal to a diffeomorphism $\Theta_1(u)$, defined on a subset of $\R^n$.

We apply $\kappa_0^+$ to ~\eqref{eq:a+_control} and use ~\eqref{eq:A_tilde} to know that 
\begin{equation}\label{eq:line_Theta_intersection_+} 
\left\{\Theta(u,v, \lambda) : (u,v) \in \tilde{U}, |u| \leq \nu_1 \alpha, |v| \leq \varepsilon_0, \tfrac{1}{4} \leq \lambda \leq 4\right\} \cap \tilde{\cA}_+ = \emptyset.
\end{equation}
By \eqref{eq:diam_W''_+}, since $\varepsilon_0$ is a constant multiple of $\nu_1$,  $\diam (\tilde{U}) \leq \sqrt{C_1} \diam(W''_\bS) \leq \varepsilon_0$.
We also know that $\tilde{\cA}_+ \subset \kappa_0^+ (\{\tfrac{1}{4} \leq |\xi|_g \leq 4\})$. Thus, we can remove the conditions $ |v| \leq \varepsilon_0$ and $\tfrac{1}{4}\leq \lambda \leq 4$ from \eqref{eq:line_Theta_intersection_+}.

Therefore,
\begin{equation}\label{eq:Theta_control_+}
B_{\nu_1 \alpha} (0) \cap \Theta_1^{-1}(\Omega_+) = \emptyset.
\end{equation}

We label 
$$(w_1, y_1, \theta_1, \eta_1) \coloneqq \kappa_0^+ (\pi_S (x_1, \xi_1, \Xi_1)) = \kappa_0^+(x_1, \xi_1) \in W''.$$

Consider the ball $B_{\nu R} (y_1) \subset \bS^n$. We know that $|y_0 - y_1 | \leq C_1 |u_0| \leq R/2$. Therefore, $y_1 \in B$. We also know $\Theta_1(0) = y_1$ and $\diam(\Theta^{-1}_1(B_{\nu R} (y_1) ) \leq 2C_1 \nu R \leq 2\nu_1 \alpha$, which by ~\eqref{eq:Theta_control_+} gives $B_{\nu R} (y_1) \cap \Omega_+ = \emptyset$.
\end{proof}

We now show that the hyperbolic $(\nu_1^-, \varepsilon_0^+)$-porosity on lines of $\cA_-$ from Lemma \ref{lem:a_hyperbolic_line_porosity} implies that $\tilde{\cA}_-$ has a property similar to line porosity in Definition~\ref{def:porous_on_lines}.
Similarly to Lemma~\ref{lem:a_tilde_porous_ball}, the following lemma is adapted from ~\cite{DJ18}*{Lemma 5.8}.

\begin{lemma}\label{lem:a_tilde_porous_line}
There exists a constant $C_1>0$ depending only on $M$, $a_1$, $a_2$, $\rho$ such that the following holds. Define the projection of $\tilde{\cA}_-$ onto the $y$-variables
$$\Omega_- \coloneqq \left\{y \in \bS^n : \exists w, \theta, \eta \text{ such that } (w, y, \theta, \eta) \in \tilde{\cA}_-\right\} \subset \bS^n,$$
and set $\nu \coloneqq \nu_1^2/2C_1^4$.
Then for all geodesics $I \subset \bS^n$ of length $|I| \in [C_1^3 K_1 h^\rho/\nu_1, 1]$, there exists $y_1 \in I$ such that $B_{\nu|I|} (y_1) \cap \Omega_- = \emptyset$.
\end{lemma}

\begin{proof}
Set $C_1>0$ to be sufficiently large, as specified later in the proof. $C_1$ will depend only on $M$, $a_1$, $a_2$, and $\rho$.
Define $W \coloneqq \kappa_0^-(V)$. We lift $V''$ to $T^*\bH^{n+1} \setminus 0$ and use $\kappa^-$ to extend $\kappa_0^-$ to a symplectomorphism
$$\kappa_0^-: V' \rightarrow W', \quad V'' \rightarrow W''$$
for open sets $W', W'' \subset T^*(\R^+_w \times \bS_y^n)$. Define
$$W''_\bS \coloneqq \kappa_0^- (V'' \cap S^*M).$$

For $C_1$  sufficiently large, $\diam W''_\bS \leq \frac{C_1}{10}\diam(V'' \cap S^*M)$. Then by ~\eqref{eq:V_diam},
\begin{equation}\label{eq:diam_W''}
\diam(W''_\bS) \leq \frac{\nu_1}{C_1}.
\end{equation}
Again, select $C_1$ to be sufficiently large so that the $\nu_1/C_1^3$-neighborhoods of $W, W'$ are contained respectively in $W', W''$.

Let $I =\{y(t) : |t| \leq |I|/2\} \subset \bS^n$ be a geodesic  with $C_1^3 K_1 h^\rho/\nu_1 \leq |I| \leq 1$ centered at some $y(0) \in \bS^n$. 

Assume first that the $y$-projection of $W'$ does not contain $y(0)$. By ~\eqref{eq:A_tilde}, $\tilde{\cA}_- \subset W$. Then the distance between $y(0)$ and $\Omega_-$ is at least $\nu_1/C_1^3$. Therefore, the ball of radius $\nu |I|$ centered at $y(0)$ does not intersect $\Omega_-$.

Now assume the $y$-projection of $W'$ does contain $y(0)$. 
Therefore, there exists a geodesic $I' = \{y(t): |t| \leq \nu_1/C_1^3\} \subset I$ of length $2 \nu_1/ C_1^3$ centered at $y(0)$ such that  $I'$ is contained in the $y$-projection of $W''$. 

By ~\eqref{eq:tangent_space_decomp}, we have $T_{(x, \xi)} (S^*M) = \R X \oplus E_s(x, \xi) \oplus E_u(x, \xi)$. Therefore, by ~\eqref{eq:preimage_Es_Eu} and ~\eqref{eq:B-_change},  we can find  $w(t), \theta(t), \eta(t)$ defined smoothly in $|t| \leq \nu_1/C_1^3$ such that 
$(w(t), y(t), \theta(t), \eta(t)) \in W''$,  $(w(0), y(0), \theta(0), \eta(0)) \in W'$,  
$$(x(t), \xi(t)) \coloneqq (\kappa_0^-)^{-1}(w(t), y(t), \theta(t), \eta(t)) \in S^*M,$$ 
and $(\dot{x}(0), \dot{\xi}(0)) \in E_u(x(0), \xi(0))$. 
Clearly, $(x(0), \xi(0)) \in V'$. 

From the definition of $E_u$ in ~\eqref{eq:Es_Eu_def}, we can write $(\dot{x}(0), \dot{\xi}(0))= (-\xi_2, -\xi_2 )$, where $\xi_2$ is orthogonal to both $x(0)$ and $\xi(0)$.

Now pick $\xi_3, \ldots, \xi_{n+1}$ such that for $\Xi_0 \coloneqq (\xi_2, \ldots, \xi_{n+1})$, $(x(0), \xi(0), \Xi_0) \in F^*M$. An explicit calculation shows 
\begin{equation}\label{eq:2nd_derivative}
\partial_t \pi_S(e^{t U_1^-}(x(0), \xi(0), \Xi_0))|_{t=0} = (-\xi_2, -\xi_2) = (\dot{x}(0), \dot{\xi}(0)).
\end{equation}

Define
$$\alpha \coloneqq \frac{|I'|}{2} = \frac{\nu_1 |I|}{2C_1^3},$$
and note  $K_1 h^\rho \leq \alpha \leq 1$.

By Lemma \ref{lem:a_hyperbolic_line_porosity},  $\cA_- \cap S^*M$ is hyperbolic $(\nu_1^-,  \varepsilon_0^+)$-porous on lines from scales $K_1 h^\rho$ to $1$. Thus, there exists $|t_0| \leq  \alpha$ such that
\begin{equation*}
\left\{\pi_S(e^{\cV^+ v} e^{\cU^-u + t_0 U_1^-}(x(0), \xi(0), \Xi_0)) :  |u| \leq \nu_1 \alpha,  |v| \leq \varepsilon_0 \right\} \cap \cA_- = \emptyset.
\end{equation*}
As $\cA_-$ is homogeneous inside $\{\tfrac{1}{4} \leq |\xi|_g \leq 4\}$, 
\begin{equation}\label{eq:a-_control}
\left\{f_\lambda(\pi_S(e^{\cV^+ v} e^{\cU^-u + t_0 U_1^-}(x(0), \xi(0), \Xi_0))) :  |u| \leq \nu_1 \alpha,  |v| \leq \varepsilon_0, \tfrac{1}{4} \leq \lambda \leq 4\right\} \cap \cA_- = \emptyset,
\end{equation}
where we recall $f_\lambda$ from \eqref{eq:f_lambda}.

From  \eqref{eq:2nd_derivative}, for $C_1$ sufficiently large,
\begin{equation*}
|\pi_S(e^{t_0 U_1^- } (x(0), \xi(0), \Xi_0)) - (x(t_0), \xi(t_0))| \leq C_1 |t_0|^2.
\end{equation*}
Therefore,
\begin{equation}\label{eq:s0_approx}
|f_{|\xi(t_0)|_g}(\pi_S(e^{t_0 U_1^- } (x(0), \xi(0), \Xi_0))) - (x(t_0), \xi(t_0))| \leq C_1 |t_0|^2.
\end{equation}

Set
$$(\tilde{x}, \tilde{\xi}, \tilde{\Xi}) \coloneqq e^{t_0U_1^-} (x(0), \xi(0), \Xi_0), \quad (\tilde{x}, \tilde{\xi}) \in V''.$$

By \eqref{eq:tangent_space_decomp} and \eqref{eq:preimage_Es_Eu}, for $C_1$ sufficiently large, we have a diffeomorphism 
$$\Theta: \tilde{U} \times \left[\tfrac{1}{4} - \tfrac{2 \nu_1}{C_1^2}, 4 +\tfrac{2 \nu_1}{C_1^2}\right] \rightarrow W'', \quad (u, v, \lambda) \mapsto \kappa_0^-(f_{\lambda}(\pi_S( e^{\cV^+v} e^{ \cU^- u} (\tilde{x}, \tilde{\xi}, \tilde{\Xi})))),$$
where $\tilde{U}$ is a neighborhood of $(0, 0) \in \R^n \times  \R^{n+1}$. 
From ~\eqref{eq:B-_change},
the value of $y$ does not change if we change $v$. As $B_+$ is invariant under rescaling $\xi$, the value of $y$ also does not change with $\lambda$. Thus, the $y$ component of $\Theta$ is equal to a diffeomorphism $\Theta_1(u)$, defined on a subset of $\R^n$.

We apply $\kappa_0^-$ to ~\eqref{eq:a-_control} and use ~\eqref{eq:A_tilde} to know that
\begin{equation}\label{eq:line_Theta_intersection}
\left\{\Theta(u, v, \lambda) : (u,v) \in \tilde{U}, |u| \leq \nu_1 \alpha, |v| \leq \varepsilon_0, \tfrac{1}{4} \leq \lambda \leq 4 \right\} \cap \tilde{\cA}_- = \emptyset.
\end{equation}
By \eqref{eq:diam_W''}, since $\varepsilon_0$ is a constant multiple of $\nu_1$, $\diam (\tilde{U}) \leq \sqrt{C_1} \diam(W'') \leq \varepsilon_0$.
We also know that $\tilde{\cA}_- \subset \kappa_0^- (\{\tfrac{1}{4} \leq |\xi|_g \leq 4\})$. Thus, we can remove the conditions $ |v| \leq \varepsilon_0$ and $\tfrac{1}{4} \leq \lambda \leq 4$ from \eqref{eq:line_Theta_intersection}.

Therefore,
\begin{equation}\label{eq:Theta_control}
B_{\nu_1 \alpha} (0) \cap \Theta_1^{-1}(\Omega_-) = \emptyset.
\end{equation}

We label 
$$(\tilde{w}, \tilde{y}, \tilde{\theta}, \tilde{\eta}) \coloneqq \kappa_0^- (\pi_S (\tilde{x}, \tilde{\xi}, \tilde{\Xi})) = \kappa_0^-(\tilde{x}, \tilde{\xi}) \in W''.$$

Consider the ball $B_{2\nu |I|} (\tilde{y}) \subset \bS^n$. We  know $\Theta_1(0) = \tilde{y}$ and $\diam(\Theta^{-1}_1(B_{2 \nu |I|} (\tilde{y}) )) \leq 2C_1 \nu |I| \leq 2 \nu_1 \alpha$, which by ~\eqref{eq:Theta_control} gives $B_{2\nu |I|} (\tilde{y}) \cap \Omega_- = \emptyset$.

By ~\eqref{eq:s0_approx}, we have that $|\tilde{y} - y(t_0)| \leq C_1^2 |t_0|^2 \leq \nu |I|$. Therefore, $B_{\nu |I|}(y(t_0)) \subset B_{2\nu |I|} (\tilde{y})$. This implies $B_{\nu |I|}(y(t_0)) 
 \cap \Omega_- = \emptyset$. Clearly, $y(t_0) \in I$.
\end{proof}

Recall that we have reduced Proposition ~\ref{prop:hyperbolic_FUP} to showing ~\eqref{eq:B_bound}. 

As $B = \cB_- \cB_+' \in I_h^{\comp}(\kappa^- \circ (\kappa^+)^{-1})$, by Lemma ~\ref{lem:FIO_to_PsiDO}, there exists $A \in \Psi_h^{\comp}(\R^+ \times \bS^n)$ such that
$$B = A \tilde{B}_\chi + \cO(h^\infty)_{L^2 \rightarrow L^2},$$
where $\chi \in C_c^\infty(\bS^n_\Delta)$ and $\tilde{B}_\chi : L^2(\R^+ \times \bS^n) \rightarrow L^2(\R^+ \times \bS^n)$ is given by $\tilde{B}_\chi v(w, y) = B_{\chi, w}(v(w, \cdot))(y)$, with $w>0$ and
$$B_{\chi,w } v(y) = (2 \pi h)^{-\frac{n}{2}} \int_{\bS^n} \left|\frac{y-y'}{2}\right|^{2 i w /h} \chi(y,y') v(y') dy'.$$
In the above equation, $|y-y'|$ denotes the Euclidean distance between $y, y' \in \bS^n \subset \R^{n+1}$.

Set $a'_- \coloneqq \tilde{a}_- \# \sigma_h(A)$ and $a'_+ \coloneqq \tilde{a}_+$. Then,
\begin{equation}\label{eq:B_to_B_chi}
\op_h^{L_V} (\tilde{a}_-) B \op_h^{L_V}(\tilde{a}_+) = \op_h^{L_V} (a'_-) \tilde{B}_\chi \op_h^{L_V}(a'_+) + \cO(h^{\infty})_{L^2 \rightarrow L^2}.
\end{equation}

Recall that $\supp \tilde{a}_\pm \subset \tilde{\cA}_\pm$. Then, $a'_\pm \in S^{\comp}_{L_V, \rho} (T^* (\R^+ \times \bS^n))$ with 
\begin{equation}\label{eq:a'_support}
\supp a'_\pm \subset \{1/4 \leq w \leq4, y \in \Omega_\pm\}.    
\end{equation}

By ~\cite{DZ16}*{Lemma 3.3}, there exists $\chi_\pm(y; h) \in C_c^\infty(\bS^n; [0,1])$ such that
$$|\partial^\alpha_y \chi_\pm| \leq C_\alpha h^{-\rho |\alpha|}, \quad \supp(1- \chi_\pm) \cap \Omega_\pm = \emptyset, \quad \supp \chi_\pm \subset \Omega_\pm(h^{\rho}).$$
Choose $\chi_w(w) \in C_c^\infty((1/8, 8))$ such that $\chi_w =1$ near $[1/4, 4]$. 
Then from ~\eqref{eq:B_to_B_chi} and ~\eqref{eq:a'_support}, we have
$$\op_h^{L_V}(\tilde{a}_-) B \op_h^{L_V}(\tilde{a}_+) = \op_h^{L_V}(a'_-) \chi_w \chi_- \tilde{B}_\chi \chi_+ \op_h^{L_V}(a'_+) + \cO(h^\infty)_{L^2 \rightarrow L^2}.$$

Thus to show ~\eqref{eq:B_bound}, it suffices to show
$$\|\chi_w \chi_- \tilde{B}_\chi \chi_+\|_{L^2(\R^+ \times \bS^n) \rightarrow L^2(\R^+ \times \bS^n)} \leq C h^{\beta/2},$$
which follows from showing that
\begin{equation}\label{eq:sup_bound}
\sup_{w \in [1/8, 8]} \|\1_{\Omega_-(h^\rho)} B_{\chi, w} \1_{\Omega_+ (h^\rho)}\|_{L^2(\bS^n) \rightarrow L^2(\bS^n)} \leq C h^{\beta/2}.
\end{equation}

Let $\{M_k\}$ be a finite covering of $\bS^n$ by balls of fixed radius $1/2$, each centered at $y_k \in \bS^n$. We view $\bS^n$ as a subset of $\R^{n+1}$ and identify each hyperplane tangent to $y_k$ with $\R^n$. 
Then let $\psi_k:M_k \rightarrow \R^n$ be gnomonic projection onto the hyperplane tangent to  $y_k$. More specifically, for $y \in M_k$, $\psi_k(y)$ is the intersection of the the hyperplane tangent to  $y_k$ and the line going through the center of $\bS^n$ and $y$.
We  know 
$$\psi_k(M_k) = B,$$ where $B \subset \R^n$ is a ball. There exists some $C_2>0$ such that for $x, y \in B$ and all $k$,
\begin{equation}\label{eq:psi_k_bounds}
C_2^{-1}|x-y|_{\R^n} \leq |\psi_k^{-1}(x) - \psi_k^{-1}(y)|_{\bS^n} \leq  |x-y|_{\R^n},
\end{equation}
where the respective metrics are the intrinsic metrics.
We further assume that $C_2\geq 2$.

Since $\psi_k$ are gnomonic projections, they preserve geodesics. This fact simplifies, but is not strictly necessary for the proof of line porosity of $\psi_k(\Omega_- \cap M_k)$. If we chose  different $\psi_k$, then a method similar to Lemma ~\ref{lem:diffeomorphism_line} could be used instead. 

For the following lemma, we take $\nu$ to be the minimum of its two values from Lemma \ref{lem:a_tilde_porous_ball} and Lemma \ref{lem:a_tilde_porous_line}.
\begin{lemma}\label{lem:mapped_porosity}
Fix $\varrho \in (3/4, \rho)$. For all $k$, $\psi_k(\Omega_+ \cap M_k)$ is $\nu/2C_2$-porous on balls from scales $h^\varrho$ to $1$ and $\psi_k(\Omega_- \cap M_k)$ is $\nu/2C_2$-porous on lines from scales $h^\varrho$ to $1$.
\end{lemma}

\begin{proof}
We begin by showing the ball porosity of $\psi_k(\Omega_+ \cap M_k)$.
Let $R \in [h^\varrho, 1]$. We examine $B_{R/2}(r_0)$ for some $r_0 \in \R^n$. 
Either $B_{R/2}(r_0)$ contains a ball of radius  $R/4$ contained inside $B$ or $B_{R/2}(r_0)$ contains a ball of radius $R/4$ that does not intersect $B$. Since $\nu/2C_2 \leq 1/4$, it suffices to examine   $B_{R/4}(r_0)$ contained in $B$.

 From ~\eqref{eq:psi_k_bounds}, we know that 
$$B_{R/4 C_2 }(\psi_k^{-1}(r_0)) \subset \psi^{-1}_k\left(B_{R/4}(r_0) \right) \subset \bS^n.$$
Recall $C_1, \nu_1, K_1$ from Lemma~\ref{lem:a_tilde_porous_ball}. For $h$ sufficiently small, $C_1 K_1 h^\rho /2\nu_1 < h^\varrho/2 C_2 \leq R/2C_2 \leq 1$.
Therefore by Lemma ~\ref{lem:a_tilde_porous_ball}, there exist  $y_1 \in B_{R/4 C_2}(\psi_k^{-1}(r_0))$ such that  $B_{\nu R/2 C_2}(y_1) \cap \Omega_+ = \emptyset$. Clearly, $\psi_k(y_1) \in B_{R/4}(r_0)$. From ~\eqref{eq:psi_k_bounds}, we have 
$$B_{\nu R/2C_2} (\psi_k(y_1)) \cap B  \subset \psi_k \left(B_{\nu R/2 C_2}(y_1) \cap M_k \right).$$ 
Therefore, $B_{\nu R/2C_2} (\psi_k(y_1)) \cap \psi_k(\Omega_+ \cap M_k)=\emptyset$. We conclude that $\psi_k(\Omega_+ \cap M_k)$ is $\nu/2C_2$-porous on balls from scales $h^\varrho$ to $1$

We now show the line porosity of $\psi_k(\Omega_- \cap M_k)$. 
Let $I$ be a line segment of length $R \in [h^\varrho, 1]$. Either $I \cap B$ contains a line segment of length $R/2$  or $I \setminus B$ contains a line segment of length $R/4$. 
Since $\nu/2C_2 \leq 1/4$, it suffices to assume $I$ is a line segment of length $R/2$, contained in $B$. 

By ~\eqref{eq:psi_k_bounds}, $\psi_k^{-1}(I)$ is a segment of a geodesic on $\bS^n$ of length at least $R/2C_2$. Recall $C_1, \nu_1, K_1$ from Lemma \ref{lem:a_tilde_porous_line}. For $h$ sufficiently small, $C_1 K_1 h^\rho /\nu_1 < h^\varrho/2 C_2 \leq R/2C_2 \leq 1$. Therefore by Lemma ~\ref{lem:a_tilde_porous_line}, there exists some $y_1 \in \psi_k^{-1}(I)$ such that $B_{\nu R/2C_2}(y_1) \cap \Omega_- = \emptyset$. Clearly, $\psi_k(y_1) \in I$. 
From ~\eqref{eq:psi_k_bounds}, we know $B_{\nu R/2C_2}(\psi_k(y_1)) \cap B \subset \psi_k(B_{\nu R/2C_2}(y_1) \cap M_k)$.
Therefore, 
$B_{\nu R/2C_2}(\psi_k(y_1)) \cap \psi_k(\Omega_- \cap M_k) = \emptyset$. We conclude that $\psi_k(\Omega_- \cap M_k)$ is $\nu/2C_2$-porous on lines from scales $h^\varrho$ to $1$.
\end{proof}

We return to showing ~\eqref{eq:sup_bound}. We see that $B_{\chi, w}$ is of the form ~\eqref{eq:B(h)_def_manifold} with
$M= \tilde{M} = \bS^n$, $U= \bS^n_\Delta$, and
$\Phi(y, y') = 2 w \log|y - y'| - w \log 4$. Again, $|y-y'|$ denotes the Euclidean distance between $y, y' \in \R^{n+1}$. We have $\Phi(y,y') = 2 w \log (\sum_{i=1}^{n+1} (y_i -y_i')^2)^{\frac{1}{2}} - w \log 4$. 
To apply Proposition \ref{prop:FUP_manifolds}, it remains to show that $\Phi$ satisfies \eqref{eq:B(h)_conditions_manifold}.

\begin{lemma}
For $y \neq y'$, $\det \partial_{y y'}^2 \Phi(y,y') \neq 0$.
\end{lemma}

\begin{proof}
For $i \neq j$, we calculate
$\partial_{y_i' y_j}^2 = 4w |y-y'|^{-4} (y_j -y_j')(y_i -y_i')$
and 
$\partial_{y_i' y_i}^2 = 4w |y-y'|^{-4} (y_i -y_i')^2  -2w |y-y'|^{-2}$.
Thus, it suffices to show $\det A \neq 0$, where
$$A \coloneqq \begin{bmatrix}
(y_1-y_1')^2 -\frac{1}{2}|y-y'|^2 & (y_1-y_1') (y_2-y_2') & \cdots & (y_1-y_1') (y_{n+1}-y_{n+1}') \\
(y_1-y_1') (y_2-y_2') & (y_2-y_2')^2-\frac{1}{2}|y-y'|^2 & \cdots & (y_2-y_2') (y_{n+1}-y_{n+1}')\\
\vdots & \vdots & \ddots & \vdots\\
(y_1-y_1') (y_{n+1}-y_{n+1}') & (y_2-y_2') (y_{n+1}-y_{n+1}') & \cdots & (y_{n+1}-y_{n+1}')^2 -\frac{1}{2}|y-y'|^2
\end{bmatrix}.$$

For $v = \begin{bmatrix}
(y_1-y_1'), &\cdots, & (y_{n+1}-y_{n+1}') 
\end{bmatrix}$ and $B=  \operatorname{diag}\left(-\frac{1}{2}|y-y'|^2, \ldots, -\frac{1}{2}|y-y'|^2\right)$, $A =v^T v  +B$.   Since $y \neq y'$, $B$ is invertible. 
Thus by the matrix determinant lemma, $\det A =  (1+v B^{-1} v^T) \det B$. 

Since $(1+v B^{-1} v^T) =1$, 
$$\det A = \left(\frac{-1}{2}\right)^n |y-y'|^{2n} \neq 0,$$
which concludes the proof.
\end{proof}

Therefore, using Lemma ~\ref{lem:mapped_porosity}, we can apply Proposition ~\ref{prop:FUP_manifolds} to conclude ~\eqref{eq:sup_bound}. This finishes the proof of Proposition \ref{prop:hyperbolic_FUP}.

\subsection{Proof of Lemma ~\ref{lem:A_w_decay}}\label{subsection:proof_Aw_decay}
Recall that we reduced the proof of Theorem \ref{thm:support} to showing  Lemma ~\ref{lem:A_w_decay}.
Note that
\begin{equation*}\label{eq:A_w_decomp}
A_\w = U(-T_1) A_{\w_-} A_{\w_+}(-T_1) U(T_1).
\end{equation*}

Thus, from Lemma ~\ref{lem:a_pm_symbol_classes}, Lemma ~\ref{lem:A_w_decay} follows from proving
\begin{equation}\label{eq:a_pm_FUP}
\|\op_h^{L_s}(a_-) \op_h^{L_u}(a_+)\|_{L^2 \rightarrow L^2} \leq C h^{\beta/2},
\end{equation}
for $C, \beta>0$ independent of $\w$.

Setting $Q=I$ in Proposition ~\ref{prop:hyperbolic_FUP}, we conclude ~\eqref{eq:a_pm_FUP}.

\bibliography{refs,rhyp}{}

\begin{bibdiv}
\begin{biblist}

\bib{ADM24}{article}{
      author={Athreya, J.},
      author={Dyatlov, S.},
      author={Miller, N.},
       title={Semiclassical measures for complex hyperbolic quotients},
        date={2024},
     journal={arXiv:2402.06477},
}

\bib{Agol}{misc}{
      author={Agol, Ian},
       title={Systoles of hyperbolic 4-manifolds},
        date={2006},
        note={arXiv{0612290}},
}

\bib{AKN09}{book}{
      author={Anantharaman, N.},
      author={Koch, H.},
      author={Nonnenmacher, S.},
      editor={Sidoravi\v{c}ius, Vladas},
       title={Entropy of eigenfunctions},
   publisher={Springer Netherlands},
        date={2009},
}

\bib{AN07}{article}{
      author={Anantharaman, N.},
      author={Nonnenmacher, S.},
       title={Half-delocalization of eigenfunctions for the {L}aplacian on an {A}nosov manifold},
        date={2007},
     journal={Ann. Inst. Fourier},
      volume={57},
      number={7},
       pages={2465\ndash 2523},
}

\bib{An08}{article}{
      author={Anantharaman, N.},
       title={Entropy and the localization of eigenfunctions},
        date={2008},
     journal={Ann. of Math.},
      volume={168},
      number={2},
       pages={435\ndash 475},
}

\bib{An67}{article}{
      author={Anosov, D.V.},
       title={Geodesic flows on closed {R}iemannian manifolds with negative curvature},
        date={1967},
     journal={Proc. Steklov Inst. Math.},
      number={90},
}

\bib{AS13}{article}{
      author={Anantharaman, N.},
      author={Silberman, L.},
       title={A haar component for quantum limits on locally symmetric spaces},
        date={2013},
     journal={Isr. J. Math.},
      volume={195},
      number={1},
       pages={393\ndash 447},
}

\bib{BBKS}{misc}{
      author={Belolipetsky, Mikhail},
      author={Bogachev, Nikolay},
      author={Kolpakov, Alexander},
      author={Slavich, Leone},
       title={Subspace stabilisers in hyperbolic lattices},
        date={2021},
        note={arXiv{2105.06897}},
}

\bib{BD18}{article}{
      author={Bourgain, J.},
      author={Dyatlov, S.},
       title={Spectral gaps without the pressure condition},
        date={2018},
     journal={Ann. of Math.},
      volume={187},
      number={3},
       pages={825\ndash 867},
}

\bib{BFMS}{article}{
      author={Bader, Uri},
      author={Fisher, David},
      author={Miller, Nicholas},
      author={Stover, Matthew},
       title={Arithmeticity, superrigidity, and totally geodesic submanifolds},
        date={2021},
        ISSN={0003-486X},
     journal={Ann. of Math. (2)},
      volume={193},
      number={3},
       pages={837\ndash 861},
         url={https://doi.org/10.4007/annals.2021.193.3.4},
      review={\MR{4250391}},
}

\bib{BelThom}{article}{
      author={Belolipetsky, Mikhail~V.},
      author={Thomson, Scott~A.},
       title={Systoles of hyperbolic manifolds},
        date={2011},
        ISSN={1472-2747},
     journal={Algebr. Geom. Topol.},
      volume={11},
      number={3},
       pages={1455\ndash 1469},
         url={https://doi-org.ezproxy.lib.ou.edu/10.2140/agt.2011.11.1455},
      review={\MR{2821431}},
}

\bib{Snappy}{misc}{
      author={Culler, Marc},
      author={Dunfield, Nathan~M.},
      author={Goerner, Matthias},
      author={Weeks, Jeffrey~R.},
       title={Snap{P}y, a computer program for studying the geometry and topology of $3$-manifolds},
         how={Available at \url{http://snappy.computop.org} (DD/MM/YYYY)},
}

\bib{Co24}{article}{
      author={Cohen, A.},
       title={Fractal uncertainty in higher dimensions},
        date={2023},
     journal={to appear in Ann. of Math., arXiv:2305.05022},
}

\bib{DFG15}{article}{
      author={Dyatlov, S.},
      author={Faure, F.},
      author={Guillarmou, C.},
       title={Power spectrum of the geodesic flow on hyperbolic manifolds},
        date={2015},
     journal={Anal. PDE},
      volume={8},
       pages={923\ndash 1000},
}

\bib{DJ18}{article}{
      author={Dyatlov, S.},
      author={Jin, L.},
       title={Semiclassical measures on hyperbolic surfaces have full support},
        date={2018},
     journal={Acta Math.},
      volume={220},
      number={2},
       pages={297\ndash 339},
}

\bib{DJ23}{article}{
      author={Dyatlov, S.},
      author={J{\'e}z{\'e}quel, M.},
       title={Semiclassical measures for higher dimensional quantum cat maps},
        date={2023},
     journal={Ann. Henri Poincar{\'e}},
}

\bib{DJN22}{article}{
      author={Dyatlov, S.},
      author={Jin, L.},
      author={Nonnenmacher, S.},
       title={Control of eigenfunctions on surfaces of variable curvature},
        date={2022},
     journal={J. Amer. Math.},
      volume={35},
      number={2},
       pages={361\ndash 465},
}

\bib{DM}{incollection}{
      author={Dani, S.~G.},
      author={Margulis, G.~A.},
       title={Limit distributions of orbits of unipotent flows and values of quadratic forms},
        date={1993},
   booktitle={I. {M}. {G}elfand {S}eminar},
      series={Adv. Soviet Math.},
      volume={16},
   publisher={Amer. Math. Soc., Providence, RI},
       pages={91\ndash 137},
      review={\MR{1237827}},
}

\bib{Do03}{article}{
      author={Donnelly, H.},
       title={Quantum unique ergodicity},
        date={2003},
     journal={Proc. Amer. Math. Soc.},
      volume={131},
       pages={2945\ndash 2951},
}

\bib{CdV85}{article}{
      author={de~Verdi\`{e}re, Y.~Colin},
       title={Ergodicit\'{e} et fonctions propres du laplacien},
        date={1985},
     journal={Commun. Math. Phys.},
      volume={102},
       pages={497},
}

\bib{Dy23}{article}{
      author={Dyatlov, S.},
       title={Macroscopic limits of chaotic eigenfunctions},
        date={2023},
     journal={Proceedings of {ICM} 2022},
      volume={V},
       pages={3704\ndash 3723},
}

\bib{DZ16}{article}{
      author={Dyatlov, S.},
      author={Zahl, J.},
       title={Spectral gaps, additive energy, and a fractal uncertainty principle},
        date={2016},
     journal={Geom. Funct. Anal.},
      volume={26},
       pages={1011\ndash 1094},
}

\bib{DZ19}{book}{
      author={Dyatlov, S.},
      author={Zworski, M.},
       title={Mathematical theory of scattering resonances},
      series={Graduate Studies in Mathematics},
   publisher={American Mathematical Society},
     address={Providence, RI},
        date={2019},
      volume={200},
}

\bib{EinLind}{article}{
      author={Einsiedler, Manfred},
      author={Lindenstrauss, Elon},
       title={On measures invariant under diagonalizable actions: the rank-one case and the general low-entropy method},
        date={2008},
        ISSN={1930-5311},
     journal={J. Mod. Dyn.},
      volume={2},
      number={1},
       pages={83\ndash 128},
         url={https://doi.org/10.3934/jmd.2008.2.83},
      review={\MR{2366231}},
}

\bib{GelLev}{article}{
      author={Gelander, Tsachik},
      author={Levit, Arie},
       title={Counting commensurability classes of hyperbolic manifolds},
        date={2014},
        ISSN={1016-443X},
     journal={Geom. Funct. Anal.},
      volume={24},
      number={5},
       pages={1431\ndash 1447},
         url={https://doi-org.ezproxy.lib.ou.edu/10.1007/s00039-014-0294-3},
      review={\MR{3261631}},
}

\bib{GPS}{article}{
      author={Gromov, M.},
      author={Piatetski-Shapiro, I.},
       title={Nonarithmetic groups in {L}obachevsky spaces},
        date={1988},
        ISSN={0073-8301},
     journal={Inst. Hautes \'{E}tudes Sci. Publ. Math.},
      number={66},
       pages={93\ndash 103},
         url={http://www.numdam.org/numdam-bin/item?id=PMIHES_1987__66__93_0},
      review={\MR{932135}},
}

\bib{GZ21}{article}{
      author={Galkowski, J.},
      author={Zelditch, S.},
       title={Lower bounds for {C}auchy data on curves in a negatively curved surface},
        date={2021},
     journal={Isr. J. Math.},
      volume={244},
       pages={971\ndash 1000},
}

\bib{Ha10}{article}{
      author={Hassell, A.},
       title={Ergodic billiards that are not quantum unique ergodic, with an appendix by {A}. {H}assell and {L}. {H}illairet},
        date={2010},
     journal={Ann. of Math.},
      volume={171},
       pages={605\ndash 618},
}

\bib{HS20}{article}{
      author={Han, R.},
      author={Schlag, W.},
       title={A higher-dimensional {B}ourgain--{D}yatlov fractal uncertainty principle},
        date={2020},
     journal={Anal. PDE},
      volume={13},
      number={3},
       pages={813\ndash 863},
}

\bib{Ji18}{article}{
      author={Jin, L.},
       title={Control for {S}chrodinger equation on hyperbolic surfaces},
        date={2018},
     journal={Math. Res. Lett.},
      volume={24},
      number={6},
       pages={1865\ndash 1877},
}

\bib{Ji20}{article}{
      author={Jin, L.},
       title={Damped wave equations on compact hyperbolic surfaces},
        date={2020},
     journal={Comm. in Math. Phys.},
      volume={373},
      number={3},
       pages={815\ndash 879},
}

\bib{Ki24}{article}{
      author={Kim, E.},
       title={Characterizing the support of semiclassical measures for higher-dimensional cat maps, with an appendix by {T}heresa {C}. {A}nderson and {R}obert {J}. {L}emke {O}liver},
        date={2024},
     journal={arXiv:2410.13449},
}

\bib{KazMarg}{article}{
      author={Ka\v{z}dan, D.~A.},
      author={Margulis, G.~A.},
       title={A proof of {S}elberg's hypothesis},
        date={1968},
        ISSN={0368-8666},
     journal={Mat. Sb. (N.S.)},
      volume={75(117)},
       pages={163\ndash 168},
      review={\MR{223487}},
}

\bib{KleinbockMargulis}{incollection}{
      author={Kleinbock, D.~Y.},
      author={Margulis, G.~A.},
       title={Bounded orbits of nonquasiunipotent flows on homogeneous spaces},
        date={1996},
   booktitle={Sina\u{\i}'s {M}oscow {S}eminar on {D}ynamical {S}ystems},
      series={Amer. Math. Soc. Transl. Ser. 2},
      volume={171},
   publisher={Amer. Math. Soc., Providence, RI},
       pages={141\ndash 172},
         url={https://doi-org.ezproxy.lib.ou.edu/10.1090/trans2/171/11},
      review={\MR{1359098}},
}

\bib{Li06}{article}{
      author={Lindenstrauss, E.},
       title={Invariant measures and arithmetic quantum unique ergodicity},
        date={2006},
     journal={Ann. of Math.},
      volume={163},
       pages={165\ndash 219},
}

\bib{LeeOh}{article}{
      author={Lee, Minju},
      author={Oh, Hee},
       title={Orbit closures of unipotent flows for hyperbolic manifolds with {F}uchsian ends},
        date={2024},
        ISSN={1465-3060},
     journal={Geom. Topol.},
      volume={28},
      number={7},
       pages={3373\ndash 3473},
         url={https://doi-org.ezproxy.lib.ou.edu/10.2140/gt.2024.28.3373},
      review={\MR{4833684}},
}

\bib{Mautner}{article}{
      author={Mautner, F.~I.},
       title={Geodesic flows on symmetric {R}iemann spaces},
        date={1957},
        ISSN={0003-486X},
     journal={Ann. of Math. (2)},
      volume={65},
       pages={416\ndash 431},
         url={https://doi-org.ezproxy.lib.ou.edu/10.2307/1970054},
      review={\MR{84823}},
}

\bib{Moore}{article}{
      author={Moore, Calvin~C.},
       title={Ergodicity of flows on homogeneous spaces},
        date={1966},
        ISSN={0002-9327},
     journal={Amer. J. Math.},
      volume={88},
       pages={154\ndash 178},
         url={https://doi-org.ezproxy.lib.ou.edu/10.2307/2373052},
      review={\MR{193188}},
}

\bib{MRBook}{book}{
      author={Maclachlan, Colin},
      author={Reid, Alan~W.},
       title={The arithmetic of hyperbolic 3-manifolds},
      series={Graduate Texts in Mathematics},
   publisher={Springer-Verlag, New York},
        date={2003},
      volume={219},
        ISBN={0-387-98386-4},
         url={https://doi-org.ezproxy.lib.ou.edu/10.1007/978-1-4757-6720-9},
      review={\MR{1937957}},
}

\bib{Prok}{article}{
      author={Prokhorov, M.~N.},
       title={Absence of discrete groups of reflections with a noncompact fundamental polyhedron of finite volume in a {L}obachevski\u{\i} space of high dimension},
        date={1986},
        ISSN={0373-2436},
     journal={Izv. Akad. Nauk SSSR Ser. Mat.},
      volume={50},
      number={2},
       pages={413\ndash 424},
      review={\MR{842588}},
}

\bib{Raimbault}{article}{
      author={Raimbault, Jean},
       title={A note on maximal lattice growth in {${\rm SO}(1,n)$}},
        date={2013},
        ISSN={1073-7928},
     journal={Int. Math. Res. Not. IMRN},
      number={16},
       pages={3722\ndash 3731},
         url={https://doi-org.ezproxy.lib.ou.edu/10.1093/imrn/rns151},
      review={\MR{3090707}},
}

\bib{RatnerMeasure}{article}{
      author={Ratner, Marina},
       title={On raghunathan's measure conjecture},
        date={1991},
        ISSN={0003-486X},
     journal={Ann. of Math. (2)},
      volume={134},
      number={3},
       pages={545\ndash 607},
         url={https://doi.org/10.2307/2944357},
      review={\MR{1135878}},
}

\bib{RatnerTop}{article}{
      author={Ratner, Marina},
       title={Raghunathan's topological conjecture and distributions of unipotent flows},
        date={1991},
        ISSN={0012-7094},
     journal={Duke Math. J.},
      volume={63},
      number={1},
       pages={235\ndash 280},
         url={https://doi-org.libproxy.berkeley.edu/10.1215/S0012-7094-91-06311-8},
      review={\MR{1106945}},
}

\bib{Ri10-1}{article}{
      author={Rivi\`{e}re, G.},
       title={Entropy of semiclassical measures for nonpositively curved surfaces},
        date={2010},
     journal={Ann. Henri Poincar\'{e}},
      volume={11},
      number={6},
       pages={1085\ndash 1116},
}

\bib{Ri10-2}{article}{
      author={Rivi\`{e}re, G.},
       title={Entropy of semiclassical measures in dimension 2},
        date={2010},
     journal={Duke Math. J.},
      volume={155},
      number={2},
       pages={271\ndash 336},
}

\bib{RS94}{article}{
      author={Rudnick, Z.},
      author={Sarnak, P.},
       title={The behaviour of eigenstates of arithmetic hyperbolic manifolds},
        date={1994},
     journal={Comm. in Math. Phys.},
      volume={161},
      number={1},
       pages={195\ndash 213},
}

\bib{Shah}{incollection}{
      author={Shah, Nimish~A.},
       title={Closures of totally geodesic immersions in manifolds of constant negative curvature},
        date={1991},
   booktitle={Group theory from a geometrical viewpoint ({T}rieste, 1990)},
   publisher={World Sci. Publ., River Edge, NJ},
       pages={718\ndash 732},
      review={\MR{1170382}},
}

\bib{ShahExpanding}{article}{
      author={Shah, Nimish~A.},
       title={Limit distributions of expanding translates of certain orbits on homogeneous spaces},
        date={1996},
        ISSN={0253-4142},
     journal={Proc. Indian Acad. Sci. Math. Sci.},
      volume={106},
      number={2},
       pages={105\ndash 125},
         url={https://doi-org.ezproxy.lib.ou.edu/10.1007/BF02837164},
      review={\MR{1403756}},
}

\bib{Sh74}{article}{
      author={Shnirelman, A.},
       title={Ergodic properties of eigenfunctions},
        date={1974},
     journal={Uspehi Mat. Nauk.},
      volume={29},
      number={6},
       pages={181\ndash 182},
}

\bib{So10}{article}{
      author={Soundararajan, K.},
       title={Quantum unique ergodicity for $\operatorname{SL}_2(\mathbb{Z})\backslash \mathbb{H}$},
        date={2010},
     journal={Ann. of Math.},
      volume={172},
       pages={1529\ndash 1538},
}

\bib{ShemSilb}{misc}{
      author={Shem-Tov, Zvi},
      author={Silberman, Lior},
       title={Arithmetic quantum unique ergodicity for products of hyperbolic 2- and 3-spaces},
        date={2022},
        note={arXiv{2206.05955}},
}

\bib{ShemSilb2}{misc}{
      author={Shem-Tov, Zvi},
      author={Silberman, Lior},
       title={Homogeneity of arithmetic quantum limits for hyperbolic 4-manifolds},
        date={2024},
        note={arXiv{2402.18218}},
}

\bib{Thurston}{misc}{
      author={Thurston, William~P.},
       title={Orbit closures of unipotent flows for hyperbolic manifolds with {F}uchsian ends},
        date={1998},
        note={arXiv{9801045}},
}

\bib{Vinberg}{article}{
      author={Vinberg, \`E.~B.},
       title={The nonexistence of crystallographic reflection groups in {L}obachevski\u{\i} spaces of large dimension},
        date={1981},
        ISSN={0374-1990},
     journal={Funktsional. Anal. i Prilozhen.},
      volume={15},
      number={2},
       pages={67\ndash 68},
      review={\MR{617472}},
}

\bib{Ze19}{article}{
      author={Zelditch, S.},
       title={Mathematics of quantum chaos in 2019},
        date={2019},
     journal={Not. Am. Math. Soc.},
      volume={66},
      number={9},
       pages={1412\ndash 1422},
}

\bib{Ze87}{article}{
      author={Zelditch, S.},
       title={Uniform distribution of eigenfunctions on compact hyperbolic surfaces},
        date={1987},
     journal={Duke Math. J.},
      volume={55},
      number={4},
       pages={919\ndash 941},
}

\bib{Zw12}{book}{
      author={Zworski, M.},
       title={Semiclassical analysis},
   publisher={Amer. Math. Soc.},
        date={2012},
      volume={138},
}

\end{biblist}
\end{bibdiv}
\bibliographystyle{alpha}

\end{document}